\documentclass[10pt]{amsart}
\usepackage{amsmath,amscd,amssymb,amsthm,amsbsy,amscd,stmaryrd}
\usepackage{mathrsfs}
\input{xy}

\usepackage{pstricks}
\usepackage{color}
\usepackage{graphicx}
\usepackage[all]{xy}
\usepackage {hyperref}

\newtheorem{theorem}{Theorem}[section]
\newtheorem{lemma}[theorem]{Lemma}
\newtheorem{proposition}[theorem]{Proposition}
\newtheorem{corollary}[theorem]{Corollary}

\theoremstyle{definition}
\newtheorem{definition}[theorem]{Definition}

\newtheorem{remark}[theorem]{Remark}
\newtheorem{notation}[theorem]{Notation}

\numberwithin{equation}{section}

\begin{document}
\title[Generalized Ricci flow II]{Generalized Ricci flow II: existence for noncompact complete manifolds}
\author{Yi Li}
\address{Department of Mathematics, Shanghai Jiao Tong University, 800 Dongchuan Road, Shanghai, 200240 China}
\email{yilicms@gmail.com}






\begin{abstract} In this paper, we continue to study the generalized Ricci flow. We
give a criterion on steady gradient Ricci soliton on complete and noncompact
Riemannian manifolds that is Ricci-flat, and then introduce a natural flow
whose stable points are Ricci-flat metrics. Modifying the argument used
by Shi and List, we prove the short time existence and higher order derivatives
estimates.
\end{abstract}
\maketitle

\tableofcontents

\section{Introduction}\label{section1}

Ricci-flat metrics play an important role in geometry and physics. For compact K\"ahler manifold with trivial first Chern class, the existence of a (K\"ahler) Ricci-flat metric was proved by Yau in his famous paper \cite{Yau78} on the Calabi conjecture. In the Riemannian setting, Ricci-flat metrics are stationary solutions of the Ricci flow introduced by Hamilton \cite{H82} as a
powerful tool, together with Perelman's breakthrough \cite{P02, P03a, P03b}, to study the Poincar\'e conjecture.

In the study of the singularities of the Ricci flow,
Ricci solitons naturally arises as the self-similar solutions. From the definition, A Ricci-flat metric is indeed a Ricci soliton.

\subsection{Compact steady gradient Ricci solitons}\label{subsection1.1}

In particular, we consider a steady gradient Ricci solition which
is a triple $(M, g, f)$, where $M$ is a smooth manifold, $g$ is a Riemannian metric on $M$ and $f$ is a smooth function, such that
\begin{equation}
{\rm Ric}_{g}+\nabla^{2}_{g}f=0 \ \ \ \text{or} \ \ \
R_{ij}+\nabla_{i}\nabla_{j}f=0.\label{1.1}
\end{equation}
Hamilton \cite{H93} showed that on a compact manifold any steady gradient Ricci soliton must be Ricci-flat; this, together with Perelman's result \cite{P02} that any compact Ricci soliton is necessarily a gradient Ricci soliton, implies that any compact
steady Ricci soliton must be Ricci-flat (cf. \cite{Cao10, P02})

\subsection{Complete noncompact steady Ricci solitons}
\label{subsection1.2}

Now we suppose $(M,g,f)$ is a complete noncompact steady
gradient Ricci soliton. The simplest example is Hamilton's cigar soliton or Witten's black hole (\cite{CLN06, CCGGIIKLLN07}), which is the complete Riemann surface $({\bf R}^{2}, g_{{\rm cs}})$ where
\begin{equation*}
g_{{\rm cs}}:=\frac{dx\otimes dx+dy\otimes dy}{1+x^{2}+y^{2}}.
\end{equation*}
If we define
\begin{equation*}
f(x,y):=-\ln\left(1+x^{2}+y^{2}\right),
\end{equation*}
then
\begin{equation*}
{\rm Ric}_{g_{{\rm cs}}}+\nabla^{2}_{g_{{\rm cs}}}
f=0.
\end{equation*}
The cigar soliton is rotationally symmetric, has positive Gaussian curvature, and is asymptotic to a cylinder near infinity; moreover, up to homothety, the cigar soliton is the uniques rotationally
symmetric gradient Ricci soliton of positive curvature on ${\bf R}^{2}$ (cf. \cite{CLN06, CCGGIIKLLN07}). The classification of two-dimensional complete compact steady gradient Ricci solitons was achieved by Hamilton \cite{H93}, which states that Any complete noncompact steady gradient Ricci soliton with positive Gaussian curvature is indeed the cigar soliton.

The cigar soliton can be generalized to a rotationally symmetric
steady gradient Ricci soliton in higher dimensions on
${\bf R}^{n}$. The resulting solitons are referred to be Bryant's solitons (see \cite{CCGGIIKLLN07} for the construction), which is
rotationally symmetric and has positive Riemann curvature
operator. Other examples of steady gradient Ricci solions were constructed by Cao \cite{Cao94} and Ivey \cite{Ivey94}.

For three-dimensional case, Perelman \cite{P02} conjectured a classification of complete noncompact steady gradient Ricci soliton with positive sectional curvature which satisfies a non-collapsing assumption at infinity. Namely, A three-dimensional complete and noncompact steady gradient
Ricci soliton which is nonflat and $\kappa$-noncollapsed, is isometric
to the Bryant soliton up to scaling. Under some extra assumptions, it was
proved in \cite{B11, CaoChen12, CM11}. A complete proof was recently achieved
by Brendle \cite{B12a} and its generalization can be found in \cite{B12b}.

Another important result is Chen's result \cite{Chen09} saying that any complete noncompact steady gradient Ricci soliton has nonnegative scalar curvature. For certain cases, the lower bounded for the scalar curvature can be improved \cite{CLY11, FG13}. When the scalar curvature of a complete steady gradient Ricci soliton achieves
its minimum, Petersen and William \cite{PW09} proved that such a soliton must be Ricci-flat. On the other hand, if a complete noncompact steady gradient Ricci soliton has positive Ricci curvature and its scalar curvature achieves its maximum, then it must be diffeomorphic to the Euclidean space with the standard metric (\cite{H93, Cao10}); in particular, in this case, such a soliton is Ricci-flat.

To remove the curvature condition, we can prove the following

\begin{proposition}\label{p1.1} Suppose $M$ is a compact or complete noncompact manifold of dimension $n$. Then the following conditions are equivalent:

\begin{itemize}

\item[(i)] there exists a Ricci-flat Riemannian metric on $M$;

\item[(ii)] there exists real numbers $\alpha,\beta$, a smooth function $\phi$ on $M$, and a Riemannian metric $g$ on $M$ such that
    \begin{equation}
    0=-R_{ij}+\alpha\nabla_{i}\nabla_{j}\phi, \ \ \
    0=\Delta_{g}\phi+\beta|\nabla_{g}\phi|^{2}_{g}.\label{1.2}
    \end{equation}

\end{itemize}

\end{proposition}
The proof is given in subsection \ref{subsection2.1}.

\begin{remark}\label{r1.2} In the compact case, the second condition in (\ref{1.2}) can be removed. However, in the complete
noncompact case, the second condition in (\ref{1.2}) is
necessarily. For example, the cigar soliton is a steady gradient Ricci soliton with nonzero scalar curvature $4/(1+x^{2}+y^{2})$.
\end{remark}

The equation (\ref{1.2}) suggests us to study the parabolic flow
\begin{equation}
\partial_{t}g(t)=-2{\rm Ric}_{g(t)}+2\alpha\nabla^{2}_{g(t)}
\phi(t), \ \ \
\partial_{t}\phi_{t}=\Delta_{g(t)}\phi(t)
+\beta\left|\nabla_{g(t)}\phi(t)\right|^{2}_{g(t)}.\label{1.3}
\end{equation}

The system (\ref{1.3}) is similar to the gradient flow of
Perelman's entropy functional $\mathcal{W}$ \cite{P02}. Let $\mathfrak{Met}(M)$ denote the space of smooth Riemannian metrics on a compact smooth manifold $M$ of dimension $m$. We define Perelman's entropy functional
\begin{equation*}
\mathcal{W}: \mathfrak{Met}(M)\times C^{\infty}(M)
\times{\bf R}^{+}\longrightarrow{\bf R}
\end{equation*}
by
\begin{equation}
\mathcal{W}(g,f,\tau):=\int_{M}\left[\tau
\left(R_{g}+|\nabla_{g}f|^{2}_{g}\right)+f-m\right]
\frac{e^{-f}}{(4\pi \tau)^{m/2}}\!\ dV_{g},\label{1.4}
\end{equation}
where $dV_{g}$ stands for the volume form of $g$. Perelman's showed that the gradient flow of (\ref{1.4}) is
\begin{eqnarray}
\partial_{t}g(t)&=&-2{\rm Ric}_{g(t)}-2\nabla^{2}_{g(t)}
f(t),\nonumber\\
\partial_{t}f(t)&=&-\Delta_{g(t)}f(t)-R_{g(t)}
+\frac{m}{2\tau(t)},\label{1.5}\\
\frac{d}{dt}\tau(t)&=&-1;\nonumber
\end{eqnarray}
moreover, the entropy $\mathcal{W}$ is nondecreasing along (\ref{1.5}). Since $\mathcal{W}$ is diffeomorphic invariant, i.e.,
\begin{equation*}
\mathcal{W}(\Phi^{\ast}g,\Phi^{\ast}f,\tau)=
\mathcal{W}(g,f,\tau)
\end{equation*}
for any diffeomorphisms $\Phi$ on $M$, it follows that the system (\ref{1.5}) is equivalent to
\begin{eqnarray}
\partial_{t}g(t)&=&-2{\rm Ric}_{g(t)},\nonumber\\
\partial_{t}f(t)&=&-\Delta_{g(t)}f(t)+\left|
\nabla_{g(t)}f(t)\right|^{2}_{g(t)}-R_{g(t)}
+\frac{m}{2\tau(t)},\label{1.6}\\
\frac{d}{dt}\tau(t)&=&-1;\nonumber
\end{eqnarray}
Thus, (\ref{1.3}) is a mixture of (\ref{1.5}) and (\ref{1.6}). There also are lots of interesting generalized Ricci flows, for example,
see \cite{HHKL08, LY12, LY13, L05, M10, M12}.

\subsection{A parabolic flow}\label{subsection1.3}

In this paper, we consider a class of Ricc flow type parabolic differential equation:
\begin{eqnarray}
\partial_{t}g(t)&=&-2{\rm Ric}_{g(t)}
+2\alpha_{1}\nabla_{g(t)}\phi(t)\otimes\nabla_{g(t)}\phi(t)
+2\alpha_{2}\nabla^{2}_{g(t)}\phi(t),\label{1.7}\\
\partial_{t}\phi(t)&=&\Delta_{g(t)}
\phi(t)+\beta_{1}\left|\nabla_{g(t)}\phi(t)\right|^{2}_{g(t)}
+\beta_{2}\phi(t),\label{1.8}
\end{eqnarray}
where $\alpha_{1},\alpha_{2},\beta_{1},\beta_{2}$ are given constants. When $\alpha_{1}=\alpha_{2}=\beta_{1}
=\beta_{2}=\phi(t)=0$, the system (\ref{1.7})--(\ref{1.8}) is exactly the Ricci flow introduced by Hamilton \cite{H82}. When $\alpha_{2}
=\beta_{1}=\beta_{2}=0$, it reduces to List's flow \cite{L05}. Recently,
Hu and Shi \cite{HS12} introduced a static flow on complete noncompact
manifold that is similar to our flow. The main result is

\begin{theorem}\label{t1.3} Let $(M,g)$ be an $m$-dimensional complete and noncompact Riemannian manifold with $|{\rm Rm}_{g}|^{2}\leq k_{0}$ on $M$
and $\phi$ a smooth function on $M$ satisfying $|\phi|^{2}+|\nabla_{g}\phi|^{2}_{g}\leq k_{1}$ and $|\nabla^{2}_{g}\phi|^{2}_{g}\leq k_{2}$. Then there exists
a positive constant $T$, depending only on $m, k_{0}, k_{1}, k_{2}, \alpha_{1},\alpha_{2}, \beta_{1}, \beta_{2}$, such that the $\star$-regular $(\alpha_{1},\alpha_{2},
\beta_{1},\beta_{2})$-flow (\ref{1.7})--(\ref{1.8}) with the initial data $(g,\phi)$ has a smooth solution $(g(t),\phi(t))$ on $M\times[0,T]$ and satisfies
the following curvature estimate. For any nonnegative integer $n$, there exist uniform positive constants $C_{k}$, depending only on $m, n, k_{0}, k_{1}, k_{2}, \alpha_{1}, \alpha_{2}, \beta_{1}, \beta_{2}$, such that
\begin{equation*}
\left|\nabla^{n}_{g(t)}{\rm Rm}_{g(t)}\right|^{2}_{g(t)}\leq\frac{C_{n}}{t^{n}}, \ \ \ \left|\nabla^{n+2}_{g(t)}\phi(t)\right|^{2}_{g(t)}
\leq\frac{C_{n}}{t^{n}}
\end{equation*}
on $M\times[0,T]$.
\end{theorem}

For the definition of regular flow and $\star$-regular flow, see Definition
\ref{d2.11} and Section \ref{section3}.

\subsection{Notions and convenience}\label{subsection1.4}

Manifolds are denote by $M, N,\cdots$. If $g$ is a Riemannian metric on $M$, we write ${\rm Rm}_{g}, {\rm Ric}_{g}, R_{g},
\nabla_{g(t)}$, and $dV_{g}$ the Riemann curvature, Ricci curvature, scalar curvature, Levi-Civita connection, and volume form of $g$, respectively. In our notions, we have
\begin{equation*}
\Gamma^{k}_{ij}=\frac{1}{2}g^{k\ell}
\left(\frac{\partial}{\partial x^{i}}g_{j\ell}
+\frac{\partial}{\partial x^{j}}g_{i\ell}
-\frac{\partial}{\partial x^{\ell}}g_{ij}\right), R^{\ell}_{ijk}=\frac{\partial}{\partial x^{i}}
\Gamma^{\ell}_{jk}-\frac{\partial}{\partial x^{j}}
\Gamma^{\ell}_{ik}
+\Gamma^{p}_{jk}\Gamma^{\ell}_{ip}
-\Gamma^{p}_{ik}\Gamma^{\ell}_{jp}.
\end{equation*}
$\nabla^{2}\alpha$ stands for the Hessian of a tensor field $\alpha$. The Lie derivative along a vector field $X$ is denoted by $\mathscr{L}_{X}$.

If we have a family of Riemannian metrics indexed by time $t$, we then write $g(t)$ or $g(x,t)$ for the Riemannian metric at
time $t$. The time derivative is denoted by $\partial_{t}$ or $\partial/\partial t$.

For two differential operators $A, B$ on a manifold $M$, define
\begin{equation*}
[A,B]:=AB-BA.
\end{equation*}
For example, the Riemann curvature tensor field ${\rm Rm}_{g}$ can be written as
\begin{equation*}
{\rm Rm}_{g}(X,Y)Z=[\nabla_{X},\nabla_{Y}]Z
-\nabla_{[X,Y]}Z.
\end{equation*}

A uniform constant $C$ is a constant depending only on the given
data not on the time $t$. Different uniform constants may be
labeled by $C_{1}, C_{2},\cdots$, according to the context.

If $\mathcal{P}$ and $\mathcal{Q}$ are two quantities (may depend on time) satisfying $\mathcal{P}\leq C\mathcal{Q}$ for some
{\it positive uniform} constant $C$, then we set
\begin{equation*}
\mathcal{P}\lesssim\mathcal{Q}.
\end{equation*}
Similarly, we can define $\mathcal{P}\thickapprox\mathcal{Q}$
if $\mathcal{P}\lesssim\mathcal{Q}$ and $\mathcal{Q}\lesssim\mathcal{P}$; that is
\begin{equation*}
\frac{1}{C}\mathcal{Q}\leq\mathcal{P}\leq C\mathcal{Q}
\end{equation*}
for some positive uniform constant $C$.

We usually raise and lower indices for tensor fields; for
example,
\begin{equation*}
R_{ijk}{}^{\ell}=g^{p\ell}R_{ijkp}=R^{\ell}_{ijk}, \ \ \
R_{i}{}^{j}{}_{k}{}^{\ell}=g^{jp}g^{\ell q}R_{ipkq}.
\end{equation*}
We also use the Einstein summation for tensor fields; for example,
\begin{equation*}
\langle a,b\rangle_{g}=a_{ij}b^{ij}:=\sum_{1\leq i,j\leq m}a_{ij}b^{ij}
=\sum_{1\leq i,j,k,\ell\leq m} g^{ik}g^{j\ell}a_{ij}b_{k\ell}
\end{equation*}
for any any two $2$-tensor fields $a=(a_{ij})$
and $b=(b_{ij})$ on a Riemannian manifold $(M,g)$ of dimension $m$. Moreover we have the Ricci identity
\begin{eqnarray*}
[\nabla_{i},\nabla_{j}]
\alpha_{k_{1}\cdots k_{r}}{}^{\ell_{1}\cdots\ell_{s}}
&=&-\sum^{r}_{h=1}R^{p}_{ijk_{h}}\alpha_{k_{1}\cdots k_{h-1} pk_{h+1}\cdots h_{r}}{}^{\ell_{1}
\cdots\ell_{s}}\\
&&+ \ \sum^{s}_{h=1}R^{\ell_{h}}_{ijp}\alpha_{k_{1}\cdots k_{r}}{}^{\ell_{1}
\cdots\ell_{h-1}p \ell_{h+1}\cdots \ell_{s}},
\end{eqnarray*}
and the contracted Bianchi identities:
\begin{equation*}
\nabla^{i}R_{ij}=\frac{1}{2}\nabla_{j}R_{g}, \ \ \
\nabla^{\ell}R_{ijk\ell}=\nabla_{i}R_{jk}-\nabla_{j}R_{ik}.
\end{equation*}

If $A$ and $B$ are two tensor fields on a Riemannian manifold $(M,g)$ we denote by $A\ast B$ any quantity obtained from $A\otimes B$ by
one or more of these operations (a slightly different from that in \cite{CLN06}):
\begin{itemize}

\item[(1)] summation over pairs of matching upper and lower indices,

\item[(2)] multiplication by constants depending only on the dimension of $M$ and
the ranks of $A$ and $B$.

\end{itemize}
We also denote by $A^{\*k}$ any $k$-fold product $A\*\cdots\* A$. The above product $\langle a,b\rangle_{g}$ can be written as
\begin{equation*}
\langle a,b\rangle_{g}=a\ast b;
\end{equation*}
in order to stress the metric $g$, we also write it as
\begin{equation*}
\langle a,b\rangle_{g}=g^{-1}\ast g^{-1}\ast a\ast b.
\end{equation*}

\section{A parabolic geometric flow}\label{section2}

In this section we introduce a parabolic geometric flow motivated by (\ref{1.2}). At first we will prove Proposition \ref{p1.1}

\subsection{A characterization of Ricci-flat metrics}\label{subsection2.1}

Recall that a steady gradient Ricci soliton is a triple $(M,g,f)$ satisfying (\ref{1.1}).

\begin{proposition}\label{p2.1}{\rm (See also Proposition \ref{p1.1})} Suppose $M$ is a compact or complete noncompact manifold of dimension $m$. Then the following conditions are equivalent:

\begin{itemize}

\item[(i)] there exists a Ricci-flat Riemannian metric on $M$;

\item[(ii)] there exists real numbers $\alpha,\beta$, a smooth function $\phi$ on $M$, and a Riemannian metric $g$ on $M$ such that
    \begin{equation}
    0=-R_{ij}+\alpha\nabla_{i}\nabla_{j}\phi, \ \ \
    0=\Delta_{g}\phi+\beta|\nabla_{g}\phi|^{2}_{g}.\label{2.1}
    \end{equation}

\end{itemize}

\end{proposition}

\begin{proof} One direction (i)$\Rightarrow$(ii) is trivial, since we can take $\alpha=\beta=\phi=0$. In the following we assume that the equation (\ref{2.1}) holds for some $\alpha,\beta$,
and $\phi, g$. When $M$ is compact, a result of Hamilton \cite{H93} tells us that $g$ must be Ricci-flat. Now we assume that $M$ is complete noncompact.

Taking the trace of the first equation in (\ref{2.1}, we get
\begin{equation}
R_{g}=\alpha\Delta_{g}\phi.\label{2.2}
\end{equation}
In particular,
\begin{equation}
R_{g}=-\alpha\beta|\nabla_{g}\phi|^{2}_{g}.\label{2.3}
\end{equation}
Hence, if $\alpha\beta\geq0$, then $R_{g}\leq0$; on the other hand, by a result of Chen \cite{Chen09}, we know that any complete noncompact steady gradient Ricci soliton has nonnegative scalar
curvature. Together with those two inequalities, we must have $R_{g}=0$ and $\nabla_{g}\phi=0$ by (\ref{2.3}). Consequently, from (\ref{2.1}), we see that $R_{ij}=0$.

To deal with the case $\alpha\beta<0$, we take the derivative $\nabla^{i}$ on the first equation of (\ref{2.1}):
\begin{equation*}
0=-\frac{1}{2}\nabla_{j}R_{g}+\alpha\Delta_{g}\nabla_{j}\phi
\end{equation*}
since $\nabla^{i}R_{ij}=\frac{1}{2}\nabla_{j}R_{g}$. According to the identity $\Delta_{g}\nabla_{j}\phi=\nabla_{j}\Delta_{g}\phi
+R_{jk}\nabla^{k}\phi$, we arrive at
\begin{equation*}
0=-\frac{1}{2}\nabla_{j}R_{g}+\alpha\left(
\nabla_{j}\Delta_{g}\phi+R_{jk}\nabla^{k}\phi\right).
\end{equation*}
Using (\ref{2.1}), we obtain
\begin{eqnarray*}
0&=&-\frac{1}{2}\nabla_{j}R_{g}+\alpha\nabla_{j}
\left(-\beta|\nabla_{g}\phi|^{2}_{g}\right)
+\alpha^{2}\nabla_{j}\nabla_{k}\phi\nabla^{k}\phi\\
&=&-\frac{1}{2}\nabla_{j}
R_{g}-\alpha\beta\nabla_{j}|\nabla_{g}\phi|^{2}_{g}
+\frac{\alpha^{2}}{2}\nabla_{j}|\nabla_{g}\phi|^{2}_{g}\\
&=&\nabla_{j}\left[\left(\frac{\alpha^{2}}{2}
-\alpha\beta\right)|\nabla_{g}\phi|^{2}_{g}
-\frac{1}{2}R_{g}\right].
\end{eqnarray*}
Hence
\begin{equation*}
\alpha(\alpha-2\beta)|\nabla_{g}\phi|^{2}_{g}=R_{g}=-\alpha\beta
|\nabla_{g}\phi|^{2}_{g}
\end{equation*}
by (\ref{2.3}). Thus $\alpha(\alpha-\beta)|\nabla_{g}
\phi|^{2}_{g}=0$. If $\alpha=0$, then $R_{ij}=0$ directly followed by (\ref{2.1}). If $\alpha\neq0$ and $\alpha\neq\beta$, we must have $\nabla_{g}\phi=0$ and then $R_{ij}=0$. If $\alpha\neq0$ but $\alpha=\beta$, in this case, $\alpha\beta=\alpha^{2}>0$, contradicting with the assumption that $\alpha\beta<0$. In each case, we get a Ricci-flat metric.
\end{proof}

\subsection{Evolution equations}\label{subsection2.2}

Motivated by Proposition \ref{p2.1}, we consider a class of Ricc flow type parabolic differential equation:
\begin{eqnarray}
\partial_{t}g(t)&=&-2{\rm Ric}_{g(t)}
+2\alpha_{1}\nabla_{g(t)}\phi(t)\otimes\nabla_{g(t)}\phi(t)
+2\alpha_{2}\nabla^{2}_{g(t)}\phi(t),\label{2.4}\\
\partial_{t}\phi(t)&=&\Delta_{g(t)}
\phi(t)+\beta_{1}\left|\nabla_{g(t)}\phi(t)\right|^{2}_{g(t)}
+\beta_{2}\phi(t),\label{2.5}
\end{eqnarray}
where $\alpha_{1},\alpha_{2},\beta_{1},\beta_{2}$ are given constants. When $\alpha_{1}=\alpha_{2}=\beta_{1}
=\beta_{2}=\phi(t)=0$, the system (\ref{1.7})--(\ref{1.8}) is exactly the Ricci flow introduced by Hamilton \cite{H82}. When $\alpha_{2}
=\beta_{1}=\beta_{2}=0$, it reduces to List's flow \cite{L05}.

To compute evolution equations for (\ref{2.4})--(\ref{2.5}), we recall variation formulas stated in \cite{CLN06}. Consider a flow
\begin{equation*}
\partial_{t}g(t)=h(t)
\end{equation*}
where $h(t)$ is a family of symmetric $2$-tensor fields. Then
\begin{eqnarray*}
\partial_{t}g^{ij}&=&-g^{ik}g^{j\ell} h_{k\ell},\\
\partial_{t}\Gamma^{k}_{ij}&=&
\frac{1}{2}g^{k\ell}\left(\nabla_{i}h_{j\ell}
+\nabla_{j}h_{i\ell}-\nabla_{\ell}h_{ij}\right),\\
\partial_{t}R^{\ell}_{ijk}&=&\frac{1}{2}
g^{\ell p}\left(\nabla_{i}\nabla_{j}h_{kp}
+\nabla_{i}\nabla_{k}h_{jp}-\nabla_{i}\nabla_{p}h_{jk}\right.\\
&& \ -\left.\nabla_{j}\nabla_{i}h_{kp}-\nabla_{j}\nabla_{k}
h_{ip}+\nabla_{j}\nabla_{p}h_{ik}\right),\\
\partial_{t}R_{jk}&=&
\frac{1}{2}g^{pq}\left(\nabla_{q}\nabla_{j}
h_{kp}+\nabla_{q}\nabla_{k}h_{jp}
-\nabla_{q}\nabla_{p}h_{jk}-\nabla_{j}\nabla_{k}h_{qp}\right),\\
\partial_{t}R_{g(t)}&=&
-\Delta_{g(t)}{\rm tr}_{g(t)}
h(t)+{\rm div}_{g(t)}\left({\rm div}_{g(t)}
h(t)\right)-\langle h(t),{\rm Ric}_{g(t)}\rangle_{g(t)},\\
\partial_{t}dV_{g(t)}&=&\frac{1}{2}{\rm tr}_{g(t)}h(t)\!\ dV_{g(t)}.
\end{eqnarray*}
We now take
\begin{equation*}
h(t):=-2{\rm Ric}_{g(t)}
+2\alpha_{1}d\phi(t)\otimes d\phi(t)+2\alpha_{2}
\nabla^{2}_{g(t)}\phi(t).
\end{equation*}

\begin{lemma}\label{l2.2} Under (\ref{2.4})--(\ref{2.5}), we have
\begin{eqnarray*}
\partial_{t}\Gamma^{k}_{ij}&=&-\nabla_{i}R_{j}{}^{k}-\nabla_{j}R_{i}{}^{k}
+\nabla^{k}R_{ij}
+2\alpha_{1}\nabla_{i}\nabla_{j}\phi(t)\cdot\nabla^{k}\phi(t)\\
&&+ \ \alpha_{2}\nabla^{k}\nabla_{i}\nabla_{j}\phi(t)-\alpha_{2}\left(R_{i}{}^{k}{}_{j}{}^{p}
\nabla_{p}\phi(t)+R_{j}{}^{k}{}_{i}{}^{p}\nabla_{p}\phi(t)\right).
\end{eqnarray*}
\end{lemma}

\begin{proof} Compute
\begin{eqnarray*}
\partial_{t}\Gamma^{k}_{ij}
&=&-\nabla_{i}R_{j}{}^{k}-\nabla_{j}R_{i}{}^{k}
+\nabla^{k}R_{ij}+2\alpha_{1}\nabla_{i}\nabla_{j}\phi(t)\cdot\nabla^{k}
\phi(t)\\
&&+ \ \alpha_{2}g^{k\ell}\left(\nabla_{i}\nabla_{j}\nabla_{\ell}
\phi(t)+\nabla_{j}\nabla_{i}\nabla_{\ell}\phi(t)
-\nabla_{\ell}\nabla_{i}\nabla_{j}\phi(t)\right).
\end{eqnarray*}
According to the Ricci identity, we have
\begin{eqnarray*}
\nabla_{i}\nabla_{j}\nabla_{\ell}\phi(t)
&=&\nabla_{i}\nabla_{\ell}\nabla_{j}\phi(t) \ \
= \ \ \nabla_{\ell}\nabla_{i}\nabla_{j}\phi(t)
-R^{p}_{i\ell j}\nabla_{p}\phi(t),\\
\nabla_{j}\nabla_{i}\nabla_{\ell}\phi(t)
&=&\nabla_{j}\nabla_{\ell}\nabla_{i}\phi(t) \ \
= \ \ \nabla_{\ell}\nabla_{j}\nabla_{i}\phi(t)
-R^{p}_{j\ell i}\nabla_{p}\phi(t);
\end{eqnarray*}
thus we prove the desired result.
\end{proof}

\begin{lemma}\label{l2.3} Under (\ref{2.4})--(\ref{2.5}), we have
\begin{eqnarray*}
\partial_{t}R_{ij}&=&\Delta_{g(t)}R_{ij}-2R_{ik}R^{k}{}_{j}
+2R_{pijq}R^{pq}
-2\alpha_{1}R_{pijq}\nabla^{p}\phi(t)\nabla^{q}\phi(t)\\
&&+ \ 2\alpha_{1}\Delta_{g(t)}\phi(t)\cdot\nabla_{i}\nabla_{j}\phi(t)
-2\alpha_{1}\nabla_{i}\nabla_{k}\phi(t)\nabla^{k}\nabla_{j}\phi(t)\\
&&+ \ \alpha_{2}\left(R_{i}{}^{p}\nabla_{p}\nabla_{j}
\phi(t)
+R_{j}{}^{p}\nabla_{p}\nabla_{i}\phi(t)+\nabla_{p}R_{ij}\nabla^{p}\phi(t)\right).
\end{eqnarray*}
\end{lemma}

\begin{proof} Note that
\begin{equation*}
\partial_{t}R_{ij}=-\frac{1}{2}\Delta_{g(t)}h_{ij}-\frac{1}{2}\nabla_{i}
\nabla_{j}\left(g^{pq}h_{pq}\right)+\frac{1}{2}g^{pq}\left(\nabla_{p}\nabla_{j}
h_{iq}+\nabla_{p}\nabla_{i}h_{jq}\right).
\end{equation*}
Denote by $I_{i}$, $i=1,2,3,4$, the $i$th term on the right-hand side of the above equation. For $I_{1}$ we have
\begin{eqnarray*}
I_{1}&=&-\frac{1}{2}\Delta_{g(t)}
\left[-2R_{ij}+2\alpha_{1}\nabla_{i}\phi(t)\nabla_{j}\phi(t)
+2\alpha_{2}\nabla_{i}\nabla_{j}\phi(t)\right]\\
&=&\Delta_{g(t)}R_{ij}-\alpha_{1}\Delta_{g(t)}\nabla_{i}\phi(t)
\nabla_{j}\phi(t)-\alpha_{1}
\nabla_{i}\phi(t)\Delta_{g(t)}\nabla_{j}\phi(t)\\
&&- \ 2\alpha_{1}\nabla_{k}\nabla_{i}\phi(t)\nabla^{k}
\nabla_{j}\phi(t)-\alpha_{2}\Delta_{g(t)}
\left(\nabla_{i}\nabla_{j}\phi(t)\right).
\end{eqnarray*}
Since $|\nabla_{g(t)}\phi|^{2}_{g(t)}$ is a function, it
follows $\nabla_{i}\nabla_{j}|\nabla_{g(t)}\phi(t)|^{2}_{g(t)}
=\nabla_{j}\nabla_{i}|\nabla_{g(t)}\phi(t)|^{2}_{g(t)}$. Hence
\begin{eqnarray*}
I_{2}&=&-\frac{1}{2}\nabla_{i}\nabla_{j}
\left(-2R_{g(t)}+2\alpha_{1}|\nabla_{g(t)}
\phi(t)|^{2}_{g(t)}+2\alpha_{2}\Delta_{g(t)}\phi(t)\right)\\
&=&\nabla_{i}\nabla_{j}R_{g(t)}
-\frac{1}{2}\alpha_{1}\nabla_{i}\nabla_{j}|\nabla_{g(t)}
\phi(t)|^{2}_{g(t)}-\frac{1}{2}\alpha_{1}
\nabla_{j}\nabla_{i}|\nabla_{g(t)}
\phi(t)|^{2}_{g(t)}\\
&&- \ \alpha_{2}\nabla_{i}
\nabla_{j}\left(\Delta_{g(t)}\phi(t)\right)\\
&=&\nabla_{i}\nabla_{j}R_{g(t)}
-\alpha_{1}\left(\nabla_{i}\nabla_{j}\nabla_{k}
\phi(t)+\nabla_{j}\nabla_{i}\nabla_{k}\phi(t)\right)\nabla^{k}
\phi(t)\\
&&- \ 2\alpha_{1}\nabla_{i}\nabla_{k}\phi(t)\nabla_{j}
\nabla^{k}\phi(t)-\alpha_{2}\nabla_{i}\nabla_{j}\left(\Delta_{g(t)}\phi(t)\right).
\end{eqnarray*}
The symmetry of $I_{3}$ and $I_{4}$ allows us to consider only one term, saying for example $I_{3}$. Then
\begin{eqnarray*}
I_{3}&=&\frac{1}{2}g^{pq}\nabla_{p}\nabla_{j}
\left(-2R_{iq}+2\alpha_{1}\nabla_{i}\phi(t)\nabla_{q}\phi(t)+2\alpha_{2}\nabla_{i}
\nabla_{q}\phi(t)\right)\\
&=&-g^{pq}\left(\nabla_{j}\nabla_{p}
R_{iq}-R^{k}_{pji}R_{kq}-R^{k}_{pjq}R_{ik}\right)+
\alpha_{1}g^{pq}\left[\nabla_{p}\nabla_{j}\nabla_{i}\phi(t)
\nabla_{q}\phi(t)\right.\\
&& \ +\left.\nabla_{j}\nabla_{i}\phi(t)
\nabla_{p}\nabla_{q}\phi(t)+\nabla_{p}
\nabla_{i}\phi(t)\nabla_{j}\nabla_{q}\phi(t)
+\nabla_{i}\phi(t)\nabla_{p}\nabla_{j}\nabla_{q}\phi(t)\right]\\
&&+ \ \alpha_{2}\nabla^{q}\nabla_{j}
\left(\nabla_{i}\nabla_{q}\phi(t)\right);
\end{eqnarray*}
since
\begin{equation*}
\nabla_{p}\nabla_{j}\nabla_{i}\phi(t)
=\nabla_{p}\nabla_{i}\nabla_{j}\phi(t)
=\nabla_{i}\nabla_{p}\nabla_{j}\phi(t)-R^{k}_{pij}\nabla_{k}\phi(t),
\end{equation*}
it follows that
\begin{eqnarray*}
I_{3}&=&-\frac{1}{2}\nabla_{i}\nabla_{j}R_{g(t)}
+R_{pijq}R^{pq}-R_{ik}R^{k}{}_{j}
+\alpha_{1}\left[\nabla_{i}\nabla_{j}\nabla_{p}\phi(t)
\nabla^{p}\phi(t)\right.\\
&& \ -\left.R_{pijq}\nabla^{p}\phi(t)\nabla^{q}\phi(t)\right]
+\alpha_{1}\left[\Delta_{g(t)}\phi(t)\nabla_{i}
\nabla_{j}\phi(t)+\nabla_{i}\nabla_{p}\phi(t)\nabla_{j}
\nabla^{p}\phi(t)\right.\\
&& \ +\left.\nabla_{i}\phi(t)\Delta_{g(t)}\nabla_{j}\phi(t)
\right]+\alpha_{2}\nabla^{q}\nabla_{j}\left(\nabla_{i}\nabla_{q}\phi(t)\right).
\end{eqnarray*}
Consequently, we arrive at
\begin{eqnarray*}
\partial_{t}R_{ij}&=&\Delta_{g(t)}R_{ij}-2R_{ik}R^{k}{}_{j}
+2R_{pijq}R^{pq}
-2\alpha_{1}R_{pijq}\nabla^{p}\phi(t)\nabla^{q}\phi(t)\\
&&+ \ 2\alpha_{1}\Delta_{g(t)}\phi(t)\cdot\nabla_{i}\nabla_{j}\phi(t)
-2\alpha_{1}\nabla_{i}\nabla_{k}\phi(t)\nabla^{k}\nabla_{j}\phi(t)
+\Lambda,
\end{eqnarray*}
where
\begin{eqnarray*}
\Lambda&:=&-\alpha_{2}\Delta_{g(t)}
\left(\nabla_{i}\nabla_{j}\phi(t)\right)
-\alpha_{2}\nabla_{i}\nabla_{j}\left(\Delta_{g(t)}
\phi(t)\right)\\
&&+ \ \alpha_{2}\nabla^{q}\nabla_{j}
\left(\nabla_{i}\nabla_{q}\phi(t)\right)
+\alpha_{2}\nabla^{q}\nabla_{i}\left(\nabla_{j}
\nabla_{q}\phi(t)\right).
\end{eqnarray*}
This term can be simplified as
\begin{eqnarray*}
\Lambda&=&\alpha_{2}\bigg[\nabla^{q}\nabla_{j}\left(
\nabla_{q}\nabla_{i}\phi(t)\right)
+\nabla^{q}\nabla_{i}\left(\nabla_{q}\nabla_{j}\phi(t)\right)\\
&& \ -\Delta_{g(t)}\left(\nabla_{i}\nabla_{j}\phi(t)\right)
-\nabla_{i}\nabla_{j}\left(\Delta_{g(t)}
\phi(t)\right)\bigg]\\
&=&\alpha_{2}\bigg[\nabla^{q}\left(\nabla_{q}
\nabla_{j}\nabla_{i}\phi(t)-R^{p}_{jqi}\nabla_{p}
\phi(t)\right)+\nabla^{q}\left(\nabla_{q}\nabla_{i}
\nabla_{j}\phi(t)-R^{p}_{iqj}\nabla_{p}\phi(t)\right)\\
&&- \ \Delta_{g(t)}\left(\nabla_{i}\nabla_{j}
\phi(t)\right)-\nabla_{i}\nabla_{j}\left(\Delta_{g(t)}
\phi(t)\right)\bigg]\\
&=&\alpha_{2}\bigg[\Delta_{g(t)}
\left(\nabla_{i}\nabla_{j}\phi(t)\right)
-\nabla_{i}\nabla_{j}\left(\Delta_{g(t)}\phi(t)\right)
-2R_{ipjq}\nabla^{p}\nabla^{q}\phi(t)\\
&&- \ \nabla^{q}R_{jqip}\cdot\nabla^{p}\phi(t)-\nabla^{q}
R_{iqjp}\cdot\nabla^{p}\phi(t)\bigg]\\
&=&\alpha_{2}\bigg[\Delta_{g(t)}
\left(\nabla_{i}\nabla_{j}\phi(t)\right)
-\nabla_{i}\nabla_{j}\left(\Delta_{g(t)}\phi(t)\right)
-2R_{ipjq}\nabla^{p}\nabla^{q}\phi(t)\\
&&- \ \nabla_{i}R_{jp}\cdot\nabla^{p}\phi(t)
-\nabla_{j}R_{ip}\cdot\nabla^{p}\phi(t)
+2\nabla_{p}R_{ij}\cdot\nabla^{p}\phi(t)\bigg]
\end{eqnarray*}
where we used the contract Bianchi identity $\nabla^{p}R_{jk\ell p}
=\nabla_{j}R_{k\ell}-\nabla_{k}R_{j\ell}$ in the last line. The final step is to simplify the difference $[\Delta_{g(t)},\nabla_{i}\nabla_{j}]\phi(t)$. According to the Ricci identity, we have
\begin{eqnarray*}
[\Delta_{g(t)},\nabla_{i}\nabla_{j}]
\phi(t)&=&\nabla^{k}\nabla_{k}\nabla_{i}\nabla_{j}\phi(t) \ \
= \ \ \nabla^{k}\left(\nabla_{i}\nabla_{k}\nabla_{j}
\phi(t)-R^{\ell}_{kij}\nabla_{\ell}\phi(t)\right)\\
&=&\nabla_{k}\nabla_{i}\nabla^{k}\nabla_{j}
\phi(t)-\nabla^{k}R^{\ell}_{kij}\nabla_{\ell}\phi(t)
-R^{\ell}_{kij}\nabla^{k}\nabla_{\ell}\phi(t)\\
&=&\nabla_{i}\nabla_{k}\nabla_{j}\nabla^{k}\phi(t)
-R^{\ell}_{kij}\nabla_{\ell}\nabla^{k}\phi(t)+R^{k}_{ki\ell}
\nabla_{j}\nabla^{\ell}\phi(t)\\
&&- \ \nabla^{k}R_{\ell jik}\nabla^{\ell}\phi(t)-R_{kij\ell}\nabla^{k}\nabla^{\ell}
\phi(t)\\
&=&\nabla_{i}\nabla_{j}\left(\Delta_{g(t)}\phi(t)\right)
+\nabla_{i}R_{j\ell}\cdot\nabla^{\ell}\phi(t)+\nabla_{j}
R_{i\ell}\nabla^{\ell}\phi(t)\\
&&+ \ R_{j\ell}\nabla_{i}\nabla_{\ell}\phi(t)+R_{i\ell}
\nabla_{j}\nabla^{\ell}\phi(t)\\
&&- \ 2R_{kij\ell}\nabla^{k}\nabla^{\ell}\phi(t)
-\nabla_{\ell}R_{ij}\nabla^{\ell}\phi(t).
\end{eqnarray*}
Consequently, $\Lambda=\alpha_{2}[R_{i\ell}\nabla^{\ell}
\nabla_{j}\phi(t)+R_{j\ell}\nabla^{\ell}\nabla_{i}\phi(t)
+\nabla_{\ell}R_{ij}\nabla^{\ell}\phi(t)]$.
\end{proof}

\begin{lemma}\label{l2.4} Under (\ref{2.4})--(\ref{2.5}), we have
\begin{eqnarray*}
\partial_{t}R_{g(t)}&=&\Delta_{g(t)}R_{g(t)}
+2\left|{\rm Ric}_{g(t)}\right|^{2}_{g(t)}
+2\alpha_{1}|\Delta_{g(t)}\phi(t)|^{2}_{g(t)}
-2\alpha_{1}\left|\nabla^{2}_{g(t)}
\phi(t)\right|^{2}_{g(t)}\\
&&- \ 4\alpha_{1}\left\langle{\rm Ric}_{g(t)},
\nabla_{g(t)}\phi(t)\otimes\nabla_{g(t)}
\phi(t)\right\rangle_{g(t)}\\
&&+ \ \alpha_{2}\left\langle \nabla_{g(t)}R_{g(t)},\nabla_{g(t)}\phi(t)\right\rangle_{g(t)}.
\end{eqnarray*}
\end{lemma}

\begin{proof} By the above formula for $\partial_{t}R_{g(t)}$, we obtain
\begin{eqnarray*}
\partial_{t}R_{g(t)}
&=&\Delta_{g(t)}R_{g(t)}+2|{\rm Ric}_{g(t)}|^{2}_{g(t)}+2
\alpha_{1}|\Delta_{g(t)}\phi(t)|^{2}_{g(t)}
-2\alpha_{1}\left|\nabla^{2}_{g(t)}
\phi(t)\right|^{2}_{g(t)}\\
&&- \ 2\alpha_{1}R^{ij}\nabla_{i}\phi(t)\nabla_{j}\phi(t)-2\alpha_{1}\left(\Delta_{g(t)}
\nabla_{i}\phi(t)-\nabla_{i}\Delta_{g(t)}\phi(t)\right)
\nabla^{i}\phi(t)+I
\end{eqnarray*}
where
\begin{equation*}
I:=-2\alpha_{2}\Delta_{g(t)}\left(\Delta_{g(t)}
\phi(t)\right)
+2\alpha_{2}\nabla^{i}\nabla^{j}\left(\nabla_{i}\nabla_{j}
\phi(t)\right)
-2\alpha_{2}R^{ij}\nabla_{i}\nabla_{j}\phi(t).
\end{equation*}
By the Ricci identity we have
\begin{eqnarray*}
\Delta_{g(t)}\nabla_{i}\phi(t)&=&\nabla^{j}\nabla_{j}\nabla_{i}\phi(t) \ \ = \ \ \nabla^{j}\nabla_{i}\nabla_{j}\phi(t) \ \
= \ \ \nabla_{j}\nabla_{i}\nabla^{j}\phi(t)\\
&=&\nabla_{i}\Delta_{g(t)}
\phi(t)+R^{j}_{jik}\nabla^{k}\phi(t) \ \
= \ \ \nabla_{i}\Delta_{g(t)}\phi(t)+R_{ik}\nabla^{k}\phi(t),
\end{eqnarray*}
and
\begin{eqnarray*}
\nabla^{i}\nabla_{j}\left(
\nabla_{i}\nabla_{j}\phi(t)\right)
&=&\nabla^{i}\left(\Delta_{g(t)}\nabla_{i}\phi(t)\right) \ \
= \ \ \nabla^{i}\left(\nabla_{i}\Delta_{g(t)}\phi(t)
+R_{ik}\nabla^{k}\phi(t)\right)\\
&=&\Delta_{g(t)}\left(\Delta_{g(t)}
\phi(t)\right)+\frac{1}{2}\nabla_{k}R_{g(t)}
\nabla^{k}\phi(t)+R_{ik}\nabla^{i}\nabla^{k}\phi(t).
\end{eqnarray*}
Consequently, we get
\begin{equation*}
I=\alpha_{2}\left\langle\nabla_{g(t)}R_{g(t)},
\nabla_{g(t)}\phi(t)\right\rangle_{g(t)}
\end{equation*}
and then the desired formula.
\end{proof}

Following Hamilton, we introduce the tensor field
\begin{equation}
B_{ijk\ell}:=-g^{pr}g^{qs}R_{ipjq}R_{kr\ell s}.
\end{equation}
Note that $B_{ji\ell k}=B_{ijk\ell}$ and
$B_{ijk\ell}=B_{k\ell ij}$.

\begin{lemma}\label{l2.5} Under (\ref{2.4})--(\ref{2.5}), we have
\begin{eqnarray*}
\partial_{t}R_{ijk\ell}
&=&\Delta_{g(t)}R_{ijk\ell}+2(B_{ijk\ell}-B_{ij\ell k}
+B_{ikj\ell}-B_{i\ell jk})\\
&&- \ (R_{i}{}^{p}R_{pjk\ell}+R_{j}{}^{p}R_{ipk\ell}
+R_{k}{}^{p}R_{ijp\ell}+R_{\ell}{}^{p}R_{ijkp})\\
&&+ \ 2\alpha_{1}\bigg[\nabla_{i}\nabla_{\ell}\phi(t)
\nabla_{j}\nabla_{k}
\phi(t)-\nabla_{i}\nabla_{k}\phi(t)\nabla_{j}
\nabla_{\ell}\phi(t)\bigg]\\
&&+ \ \alpha_{2}\bigg[\nabla^{p}R_{ijk\ell}\nabla_{p}\phi(t)
-R_{ijk}{}^{p}\nabla_{p}\nabla_{\ell}\phi(t)
+R^{p}{}_{jk\ell}\nabla_{i}\nabla_{p}\phi(t)\\
&&+ \ R_{i}{}^{p}{}_{k\ell}\nabla_{j}\nabla_{p}\phi(t)
+R_{ij}{}^{p}{}_{\ell}\nabla_{k}\nabla_{p}
\phi(t)\bigg].
\end{eqnarray*}
\end{lemma}

\begin{proof} Recall the evolution equation
\begin{equation*}
\partial_{t}R^{\ell}_{ijk}=\frac{1}{2}g^{\ell p}
\left(\nabla_{i}\nabla_{k}h_{jp}+\nabla_{j}\nabla_{p}
h_{ik}
-\nabla_{i}\nabla_{p}h_{jk}-\nabla_{j}\nabla_{k}
h_{ip}-R^{q}_{ijk}h_{qp}-R^{q}_{ijp}h_{kq}\right)
\end{equation*}
where $\partial_{t}g_{ij}=h_{ij}$. Applying the above formula to $h_{ij}=-2R_{ij}+2\alpha_{1}\nabla_{i}\phi(t)\nabla_{j}\phi(t)
+2\alpha_{2}\nabla_{i}\nabla_{j}\phi(t)$ implies
\begin{eqnarray*}
\partial_{t}R^{\ell}_{ijk}
&=&g^{\ell p}\left(\nabla_{i}\nabla_{p}R_{jk}
+\nabla_{j}\nabla_{k}R_{ip}-\nabla_{i}\nabla_{k}
R_{jp}-\nabla_{j}\nabla_{p}R_{ik}\right)+g^{\ell p}\\
&&\bigg[\alpha_{1}\nabla_{i}
\bigg(\nabla_{k}\nabla_{j}\phi(t)\nabla_{p}\phi(t)
+\nabla_{j}\phi(t)\nabla_{k}\nabla_{p}\phi(t)\bigg)
+\alpha_{2}\nabla_{i}\nabla_{k}\nabla_{j}\nabla_{p}\phi(t)\\
&&+ \ \alpha_{1}\nabla_{j}\bigg(\nabla_{p}
\nabla_{i}\phi(t)\nabla_{k}\phi(t)+\nabla_{i}\phi(t)
\nabla_{p}\nabla_{k}\phi(t)\bigg)
+\alpha_{2}\nabla_{j}\nabla_{p}\nabla_{i}\nabla_{k}\phi(t)\\
&&- \ \alpha_{1}\nabla_{i}
\bigg(\nabla_{p}\nabla_{j}\phi(t)\nabla_{k}\phi(t)
+\nabla_{j}\phi(t)\nabla_{p}\nabla_{k}\phi(t)\bigg)
-\alpha_{2}\nabla_{i}\nabla_{p}\nabla_{j}\nabla_{k}
\phi(t)\\
&&- \ \alpha_{1}\nabla_{j}
\bigg(\nabla_{k}\nabla_{i}\phi(t)\nabla_{p}\phi(t)
+\nabla_{i}\phi(t)\nabla_{k}\nabla_{p}\phi(t)\bigg)
-\alpha_{2}\nabla_{j}\nabla_{k}\nabla_{i}\nabla_{p}
\phi(t)\bigg]\\
&&+ \ g^{\ell p}\bigg[R^{q}_{ijk}\bigg(
R_{qp}-\alpha_{1}\nabla_{q}\phi(t)\nabla_{p}\phi(t)
-\alpha_{2}\nabla_{q}\nabla_{p}\phi(t)\bigg)\\
&&+ \ R^{q}_{ijp}\bigg(R_{kq}-\alpha_{1}
\nabla_{k}\phi(t)\nabla_{q}\phi(t)-\alpha_{2}\nabla_{k}
\nabla_{q}\phi(t)\bigg)\bigg]:=I_{1}+I_{2}+I_{3},
\end{eqnarray*}
where
\begin{eqnarray*}
I_{1}&:=&g^{\ell p}\bigg(\nabla_{i}
\nabla_{p}R_{jk}+\nabla_{j}\nabla_{k}
R_{ip}-\nabla_{i}\nabla_{k}R_{jp}-\nabla_{j}\nabla_{p}
R_{ik}+R^{q}_{ijk}R_{qp}+R^{q}_{ijp}R_{kq}\bigg),\\
I_{3}&:=&g^{\ell p}
\bigg[R^{q}_{ijk}\bigg(-\alpha_{1}
\nabla_{q}\phi(t)\nabla_{p}\phi(t)-\alpha_{2}
\nabla_{q}\nabla_{p}\phi(t)\bigg)\\
&&+ \ R^{q}_{ijp}\bigg(-\alpha_{1}\nabla_{k}
\phi(t)\nabla_{q}\phi(t)
-\alpha_{2}\nabla_{k}\nabla_{q}\phi(t)\bigg)\bigg]\\
&=&-\bigg(\alpha_{1}R^{q}_{ijk}\nabla_{q}\phi(t)\nabla^{\ell}\phi(t)
+\alpha_{2}R^{q}_{ijk}\nabla_{q}\nabla^{\ell}\phi(t)\bigg)\\
&&- \ \bigg(\alpha_{1}R_{ij}{}^{\ell q}
\nabla_{k}\phi(t)\nabla_{q}\phi(t)
+\alpha_{2}R_{ij}{}^{\ell q}\nabla_{k}\nabla_{q}\phi(t)\bigg),\\
I_{2}&:=&\text{the rest terms}.
\end{eqnarray*}
According to \cite{CLN06, H82} we have
\begin{eqnarray*}
I_{1}&=&\Delta_{g(t)}R^{\ell}_{ijk}
+g^{pq}\left(R^{r}_{ijp}R^{\ell}_{rqk}
-2R^{r}_{pik}R^{\ell}_{jqr}+2R^{\ell}_{pir}R^{r}_{jqk}\right)\\
&&- \ R_{i}{}^{r}R^{\ell}_{rjk}
-R_{j}{}^{r}R^{\ell}_{irk}
-R_{k}{}^{r}R^{\ell}_{ijr}
+R_{r}{}^{\ell}R^{r}_{ijk}.
\end{eqnarray*}
It can be showed that
\begin{eqnarray*}
I_{2}&=&g^{\ell p}\bigg[\alpha_{1}\bigg(\nabla_{i}\nabla_{j}
\nabla_{k}\phi(t)\nabla_{p}\phi(t)+\nabla_{j}\nabla_{i}\nabla_{p}
\phi(t)\nabla_{k}\phi(t)\\
&&- \ \nabla_{i}\nabla_{j}\nabla_{p}
\phi(t)\nabla_{k}\phi(t)-\nabla_{j}\nabla_{i}
\nabla_{k}\phi(t)\nabla_{p}\phi(t)\bigg)
+2\alpha_{1}\bigg(\nabla_{k}
\nabla_{j}\phi(t)\nabla_{i}\nabla_{p}
\phi(t)\\
&&- \ \nabla_{k}\nabla_{i}\phi(t)\nabla_{j}
\nabla_{p}\phi(t)\bigg)\bigg]
+\alpha_{2}\bigg(\nabla_{i}\nabla_{k}
\nabla_{j}\nabla_{p}\phi(t)+\nabla_{j}
\nabla_{p}\nabla_{i}\nabla_{k}\phi(t)\\
&&- \ \nabla_{j}\nabla_{k}\nabla_{i}\nabla_{p}\phi(t)
-\nabla_{i}\nabla_{p}
\nabla_{j}\nabla_{k}\phi(t)\bigg)\\
&=&-\alpha_{1}R^{q}_{ijk}\nabla_{q}\phi(t)\nabla^{\ell}
\phi(t)+\alpha_{1}R_{ij}{}^{\ell q}
\nabla_{q}\phi(t)\nabla_{k}\phi(t)
+2\alpha_{1}\bigg(\nabla_{i}\nabla^{\ell}
\phi(t)\nabla_{k}\nabla_{j}
\phi(t)\\
&&- \ \nabla_{i}\nabla_{k}\phi(t)
\nabla_{j}\nabla^{\ell}\phi(t)\bigg)
+\alpha_{2}\bigg(\nabla_{i}\nabla_{k}
\nabla_{j}\nabla^{\ell}\phi(t)+\nabla_{j}\nabla^{\ell}
\nabla_{i}\nabla_{k}\phi(t)\\
&&- \ \nabla_{j}\nabla_{k}\nabla_{i}\nabla^{\ell}
\phi(t)-\nabla_{i}\nabla^{\ell}\nabla_{j}
\nabla_{k}\phi(t)\bigg);
\end{eqnarray*}
together with $I_{3}$, we arrive at
\begin{eqnarray*}
I_{2}+I_{3}&=&-2\alpha_{1}R^{q}_{ijk}\nabla_{q}\phi(t)
\nabla^{\ell}\phi(t)
-\alpha_{2}\left(R^{q}_{ijk}
\nabla_{q}\nabla^{\ell}\phi(t)
+R_{ij}{}^{\ell q}\nabla_{k}\nabla_{q}\phi(t)\right)\\
&&+ \ 2\alpha_{1}
\bigg(\nabla_{i}\nabla^{\ell}
\phi(t)\nabla_{k}\nabla_{j}
\phi(t)-\nabla_{i}\nabla_{k}\phi(t)
\nabla_{j}\nabla^{\ell}\phi(t)\bigg)\\
&&+ \ \alpha_{2}\bigg(\nabla_{i}\nabla_{k}
\nabla_{j}\nabla^{\ell}\phi(t)+\nabla_{j}\nabla^{\ell}
\nabla_{i}\nabla_{k}\phi(t)\\
&&- \ \nabla_{j}\nabla_{k}\nabla_{i}\nabla^{\ell}
\phi(t)-\nabla_{i}\nabla^{\ell}\nabla_{j}
\nabla_{k}\phi(t)\bigg).
\end{eqnarray*}
Using the Ricci identity, the last term on the right-hand side of $I_{2}+I_{3}$ can be simplified as
\begin{eqnarray*}
\nabla_{i}\nabla_{k}\nabla_{j}\nabla^{\ell}
\phi(t)-\nabla_{i}\nabla^{\ell}\nabla_{j}\nabla_{k}\phi(t)
&=&g^{\ell p}\bigg(\nabla_{i}\nabla_{k}\nabla_{j}\nabla_{p}
\phi(t)-\nabla_{i}\nabla_{p}\nabla_{j}\nabla_{k}\phi(t)\bigg)\\
&=&g^{\ell p}
\nabla_{i}\bigg(\nabla_{k}\nabla_{p}
\nabla_{j}\phi(t)
-\nabla_{p}\nabla_{k}\nabla_{j}\phi(t)\bigg)\\
&=&-g^{\ell p}\nabla_{i}\bigg(R^{q}_{kpj}\nabla_{q}\phi(t)\bigg)\\
&=&-\nabla_{i}R_{k}{}^{\ell}{}_{j}{}^{q}
\nabla_{q}\phi(t)-R_{k}{}^{\ell}{}_{j}{}^{q}
\nabla_{i}\nabla_{q}\phi(t),\\
\nabla_{j}\nabla^{\ell}\nabla_{i}\nabla_{k}
\phi(t)-\nabla_{j}\nabla_{k}\nabla_{i}
\nabla^{\ell}\phi(t)
&=&g^{\ell p}\nabla_{j}
\bigg(\nabla_{p}\nabla_{i}\nabla_{k}
\phi(t)-\nabla_{k}
\nabla_{i}
\nabla_{p}\phi(t)\bigg)\\
&=&g^{\ell p}\nabla_{j}
\bigg(-R^{q}_{pki}\nabla_{q}\phi(t)\bigg)\\
&=&-\nabla_{j}R^{\ell}{}_{ki}{}^{q}
\nabla_{q}\phi(t)
-R^{\ell}{}_{ki}{}^{q}
\nabla_{j}\nabla_{q}\phi(t).
\end{eqnarray*}
Consequently,
\begin{eqnarray*}
I_{2}+I_{3}&=&
-2\alpha_{1}R^{q}_{ijk}\nabla_{q}\phi(t)\nabla^{\ell}
\phi(t)\\
&&+ \ 2\alpha_{1}\bigg(\nabla_{i}\nabla^{\ell}
\phi(t)\nabla_{k}\nabla_{j}\phi(t)
-\nabla_{i}\nabla_{k}
\phi(t)\nabla_{j}\nabla^{\ell}\phi(t)\bigg)\\
&&+ \ \alpha_{2}\bigg(\nabla^{q}R_{ijk}{}^{\ell}
-R_{k}{}^{\ell}{}_{j}{}^{q}
\nabla_{i}\nabla_{q}\phi(t)
-R^{\ell}{}_{ki}{}^{q}\nabla_{j}
\nabla_{q}\phi(t)\\
&& \ -R_{ijk}{}^{q}\nabla_{q}\nabla^{\ell}\phi(t)
-R_{ij}{}^{\ell q}\nabla_{k}\nabla_{q}\phi(t)\bigg),
\end{eqnarray*}
because
\begin{equation*}
\nabla_{i}R_{k}{}^{\ell}{}_{j}{}^{q}
+\nabla_{j}R^{\ell}{}_{ki}{}^{q}
=g^{\ell s}g^{pq}\bigg(\nabla_{i}
R_{jpks}+\nabla_{j}R_{piks}\bigg)
=-g^{\ell s}g^{qp}\nabla_{p}R_{ijks}
=-\nabla^{q}R_{ijk}{}^{\ell}.
\end{equation*}
Replacing $\ell$ by $s$ in $I_{1}, I_{2}, I_{3}$, we have
\begin{equation*}
\partial_{t}R^{s}_{ijk}=\tilde{I}_{1}+\tilde{I}_{2}+\tilde{I}_{3},
\end{equation*}
where we denote by $\tilde{I}_{i}$ the corresponding terms; hence
\begin{eqnarray*}
\partial_{t}R_{ijk\ell}&=&\partial_{t}\bigg(g_{\ell s}R^{s}_{ijk}\bigg) \ \ = \ \ \partial_{t}g_{\ell s}R^{s}_{ijk}
+g_{\ell s}\partial_{t}R^{s}_{ijk}\\
&=&\bigg(-2R_{\ell s}+2\alpha_{1}
\nabla_{\ell}\phi(t)
\nabla_{s}\phi(t)+2\alpha_{2}
\nabla_{\ell}\nabla_{s}\phi(t)\bigg)R^{s}_{ijk}\\
&&+ \ g_{\ell s}\bigg(\tilde{I}_{1}+\tilde{I}_{2}+\tilde{I}_{3}
\bigg)\\
&=&\bigg(-2R_{\ell s}R^{s}_{ijk}
+g_{\ell s}\tilde{I}_{1}\bigg)\\
&&+ \ 2\alpha_{1}R^{s}_{ijk}\nabla_{\ell}\phi(t)
\nabla_{s}\phi(t)+2\alpha_{2}R^{s}_{ijk}\nabla_{\ell}
\nabla_{s}\phi(t)+g_{\ell s}\bigg(\tilde{I}_{2}+\tilde{I}_{3}\bigg).
\end{eqnarray*}
The first bracket on the right-hand side follows
from Hamilton's computation \cite{CLN06, H82}; the rest
terms can be computed from the expressions for $\tilde{I}_{2}$ and $\tilde{I}_{3}$.
\end{proof}

Next we compute evolution equations for $\phi(t)$.

\begin{lemma}\label{l2.6} Under (\ref{2.4})--(\ref{2.5}), we have
\begin{eqnarray*}
\partial_{t}|\nabla_{g(t)}\phi(t)|^{2}_{g(t)}
&=&\Delta_{g(t)}|\nabla_{g(t)}\phi(t)|^{2}_{g(t)}
+2\beta_{2}|\nabla_{g(t)}
\phi(t)|^{2}_{g(t)}\\
&&- \ 2\left|\nabla^{2}_{g(t)}
\phi(t)\right|^{2}_{g(t)}
-2\alpha_{1}|\nabla_{g(t)}\phi(t)|^{4}_{g(t)}\\
&&+ \ (4\beta_{1}-2\alpha_{2})\left\langle\nabla_{g(t)}
\phi(t)\otimes\nabla_{g(t)}\phi(t),\nabla^{2}_{g(t)}\phi(t)\right\rangle_{g(t)}.
\end{eqnarray*}
\end{lemma}

\begin{proof} Using (\ref{2.5}) we have
\begin{eqnarray*}
\partial_{t}\nabla_{i}\phi(t)&=&
\nabla_{i}\partial_{t}\phi(t) \ \ = \ \
\nabla_{i}\left(\Delta_{g(t)}\phi(t)
+\beta_{1}|\nabla_{g(t)}\phi(t)|^{2}_{g(t)}
+\beta_{2}\phi(t)\right)\\
&=&\Delta_{g(t)}\nabla_{i}\phi(t)
-R_{ij}\nabla^{j}\phi(t)+2\beta_{1}\nabla^{j}\phi(t)
\nabla_{i}\nabla_{j}\phi(t)+\beta_{2}
\nabla_{i}\phi(t),
\end{eqnarray*}
where we use the identity $\nabla_{i}\Delta_{g(t)}
\phi(t)=\Delta_{g(t)}\nabla_{i}\phi(t)
-R_{ij}\nabla^{j}\phi(t)$. Using (\ref{2.4}) we then get
\begin{eqnarray*}
\partial_{t}|\nabla_{g(t)}\phi(t)|^{2}_{g(t)}
&=&\partial_{t}\bigg(g^{ij}\nabla_{i}\phi(t)\nabla_{j}\phi(t)\bigg)\\
&=&\partial_{t}g^{ij}\nabla_{i}\phi(t)
\nabla_{j}\phi(t)+2g^{ij}\nabla_{j}\phi(t)\partial_{t}\nabla_{i}\phi(t)\\
&=&\bigg(2R^{ij}-2\alpha_{1}\nabla^{i}
\phi(t)\nabla^{j}\phi(t)
-2\alpha_{2}\nabla^{i}\nabla^{j}\phi(t)\bigg)
\nabla_{i}\phi(t)\nabla_{j}\phi(t)\\
&&+ \ 2g^{ij}\nabla_{j}\phi(t)
\bigg(\Delta_{g(t)}\nabla_{i}\phi(t)
-R_{ik}\nabla^{k}\phi(t)\\
&&+ \ 2\beta_{1}\nabla^{k}\phi(t)\nabla_{i}\nabla_{k}
\phi(t)
+\beta_{2}\nabla_{i}\phi(t)\bigg)\\
&=&-2\alpha_{1}|\nabla_{g(t)}\phi(t)|^{4}_{g(t)}
-2\alpha_{2}\nabla^{i}\nabla^{j}\phi(t)
\nabla_{i}\phi(t)\nabla_{j}\phi(t)\\
&&+ \ 2\nabla^{i}\phi(t)\Delta_{g(t)}
\nabla_{i}\phi(t)+4\beta_{1}
\nabla_{i}\nabla_{k}
\phi(t)\nabla^{i}\phi(t)\nabla^{k}\phi(t)\\
&&+ \ 2\beta_{2}|\nabla_{g(t)}\phi(t)|^{2}_{g(t)}
\end{eqnarray*}
which implies the desired equation.
\end{proof}

\begin{lemma}\label{l2.7} Under (\ref{2.4})--(\ref{2.5}), we have
\begin{eqnarray*}
\partial_{t}(\nabla_{i}\nabla_{j}
\phi(t))&=&\Delta_{g(t)}(\nabla_{i}\nabla_{j}
\phi(t))+2R_{pijq}\nabla^{p}\nabla^{q}\phi(t)+\beta_{2}\nabla_{i}\nabla_{j}\phi(t)\\
&&- \ R_{ip}\nabla^{p}\nabla_{j}\phi(t)
-R_{jp}\nabla^{p}\nabla_{i}\phi(t)
-2\alpha_{1}|\nabla_{g(t)}\phi(t)|^{2}_{g(t)}
\nabla_{i}\nabla_{j}\phi(t)\\
&&+ \ (2\beta_{1}-\alpha_{2})
\nabla_{k}\phi(t)\nabla_{k}\nabla_{i}\nabla_{j}
\phi(t)+2\beta_{1}\nabla_{i}\nabla^{k}\phi(t)
\nabla_{j}\nabla_{k}\phi(t)\\
&&+ \ 2(\beta_{1}-\alpha_{2})
R_{pijq}\nabla^{p}\phi(t)\nabla^{q}\phi(t).
\end{eqnarray*}
\end{lemma}

\begin{proof} Compute
\begin{eqnarray*}
\partial_{t}\left(\nabla_{i}\nabla_{j}
\phi(t)\right)&=&\partial_{t}\left[\partial_{i}
\partial_{j}\phi(t)-\Gamma^{k}_{ij}\partial_{k}\phi(t)\right]\\
&=&\partial_{i}\partial_{j}\left(\partial_{t}
\phi(t)\right)
-\left(\partial_{t}\Gamma^{k}_{ij}\right)\partial_{k}\phi(t)
-\Gamma^{k}_{ij}\cdot\partial_{t}\left(\partial_{k}\phi(t)\right)\\
&=&\nabla_{i}\nabla_{j}\left(\partial_{t}
\phi(t)\right)-\partial_{t}\Gamma^{k}_{ij}\cdot
\partial_{k}\phi(t).
\end{eqnarray*}
Using (\ref{2.5}) we have
\begin{eqnarray*}
\nabla_{i}\nabla_{j}
\left(\partial_{t}\phi(t)\right)
&=&\nabla_{i}\nabla_{j}
\left(\Delta_{g(t)}\phi(t)
+\beta_{1}|\nabla_{g(t)}\phi(t)|^{2}_{g(t)}
+\beta_{2}\phi(t)\right)\\
&=&\nabla_{i}
\left(\nabla_{k}\nabla_{j}\nabla^{k}\phi(t)
-R_{j\ell}\nabla^{\ell}\phi(t)\right)
+2\beta_{1}\left(\nabla_{i}\nabla^{k}\phi(t)\nabla_{j}
\nabla_{k}\phi(t)\right.\\
&& +\left.\nabla^{k}\phi(t)\nabla_{i}\nabla_{j}
\nabla_{k}\phi(t)\right)+\beta_{2}
\nabla_{i}\nabla_{j}\phi(t)\\
&=&\nabla_{i}\nabla_{k}\nabla_{j}\nabla^{k}
\phi(t)-\nabla_{i}R_{j\ell}\nabla^{\ell}\phi(t)
-R_{j\ell}\nabla_{i}\nabla^{\ell}\phi(t)+\beta_{2}\nabla_{i}\nabla_{j}\phi(t)\\
&&+ \ 2\beta_{1}
\left(\nabla_{i}\nabla^{k}\phi(t)
\nabla_{j}\nabla_{k}\phi(t)
+\nabla^{k}\phi(t)\nabla_{i}\nabla_{j}
\nabla_{k}\phi(t)\right)\\
&=&\nabla_{k}\nabla_{i}\nabla_{j}
\nabla^{k}\phi(t)-R^{\ell}_{ikj}\nabla_{\ell}
\nabla^{k}\phi(t)
+R^{k}_{ik\ell}\nabla_{j}\nabla^{\ell}\phi(t)\\
&&- \ \nabla_{i}R_{j\ell}\nabla^{\ell}\phi(t)
-R_{j\ell}\nabla_{i}\nabla^{\ell}
\phi(t)+\beta_{2}\nabla_{i}\nabla_{j}
\phi(t)\\
&&+ \ 2\beta_{1}\left(\nabla_{i}
\nabla^{k}\phi(t)\nabla_{j}
\nabla_{k}\phi(t)
+\nabla^{k}\phi(t)\nabla_{i}
\nabla_{j}\nabla_{k}\phi(t)\right).
\end{eqnarray*}
Since
\begin{eqnarray*}
\nabla_{k}\nabla_{i}
\nabla_{j}\nabla^{k}
\phi(t)
&=&\nabla^{k}\nabla_{i}\nabla_{j}\nabla_{k}
\phi(t) \ \ = \ \ \nabla^{k}\nabla_{i}
\nabla_{k}\nabla_{j}\phi(t)\\
&=&\nabla^{k}\left(\nabla_{k}
\nabla_{i}\nabla_{j}\phi(t)-R^{\ell}_{ikj}\nabla_{\ell}
\phi(t)\right)\\
&=&\Delta_{g(t)}
\left(\nabla_{i}\nabla_{j}\phi(t)\right)
-\nabla^{k}R^{\ell}_{ikj}\nabla_{\ell}
\phi(t)-R^{\ell}_{ikj}\nabla^{k}\nabla_{\ell}\phi(t),
\end{eqnarray*}
we have
\begin{eqnarray*}
\nabla_{i}\nabla_{j}(\partial_{t}\phi(t))
&=&\Delta_{g(t)}\left(\nabla_{i}\nabla_{j}
\phi(t)\right)
-R_{i\ell}\nabla^{\ell}\nabla_{j}
\phi(t)
-R_{j\ell}\nabla^{\ell}\nabla_{i}\phi(t)+\beta_{2}\nabla_{i}\nabla_{j}
\phi(t)\\
&&- \ \left(\nabla_{i}R_{j\ell}
+\nabla_{j}R_{i\ell}-\nabla_{\ell}R_{ij}\right)
\nabla^{\ell}\phi(t)-2R^{\ell}_{ikj}\nabla_{\ell}\nabla^{k}
\phi(t)\\
&&+ \ 2\beta_{1}\left(\nabla_{i}
\nabla^{k}\phi(t)\nabla_{j}\nabla_{k}
\phi(t)+\nabla^{k}\phi(t)
\nabla_{i}\nabla_{j}\nabla_{k}
\phi(t)\right).
\end{eqnarray*}
Using Lemma \ref{2.2} implies
\begin{eqnarray*}
\partial_{t}\left(\nabla_{i}
\nabla_{j}\phi(t)\right)&=&\Delta_{g(t)}\left(\nabla_{i}
\nabla_{j}\phi(t)\right)
-2R_{ikj\ell}\nabla^{k}\nabla^{\ell}\phi(t)
-R_{i\ell}\nabla^{\ell}\nabla_{j}\phi(t)\\
&&- \ R_{j\ell}\nabla^{\ell}\nabla_{i}
\phi(t)+\beta_{2}\nabla_{i}\nabla_{j}
\phi(t)+2\beta_{1}
\nabla_{i}\nabla^{k}
\phi(t)\nabla_{j}\nabla_{k}
\phi(t)\\
&&+ \ 2\beta_{1}\nabla^{k}
\phi(t)\nabla_{i}\nabla_{j}\nabla_{k}
\phi(t)-2\alpha_{1}|\nabla_{g(t)}\phi(t)|^{2}_{g(t)}
\nabla_{i}\nabla_{j}\phi(t)\\
&&- \ \alpha_{2}\nabla^{k}\phi(t)\nabla_{k}\nabla_{i}
\nabla_{j}\phi(t)+2\alpha_{2}
R_{ikj\ell}\nabla^{\ell}\phi(t)\nabla^{k}\phi(t).
\end{eqnarray*}
Now Lemma \ref{l2.7} follows from $\nabla_{i}
\nabla_{j}\nabla_{k}\phi(t)=\nabla_{k}
\nabla_{i}\nabla_{j}\phi(t)
-R^{\ell}_{ikj}\nabla_{\ell}
\phi(t)$.
\end{proof}

\begin{lemma} Under (\ref{2.4})--(\ref{2.5}), we have
\begin{eqnarray*}
\partial_{t}\left(\nabla_{i}
\phi(t)\nabla_{j}
\phi(t)\right)&=&\Delta_{g(t)}\left(\nabla_{i}\phi(t)\nabla_{j}\phi(t)\right)
-\nabla^{k}\phi(t)\left(R_{ik}\nabla_{j}
\phi(t)+R_{jk}\nabla_{i}\phi(t)\right)\\
&&- \ 2\nabla_{i}\nabla^{k}\phi(t)\nabla_{j}\nabla_{k}\phi(t)+2\beta_{2}\nabla_{i}
\phi(t)\nabla_{j}\phi(t)\\
&&+ \ 2\beta_{1}\nabla^{k}\phi(t)
\left(\nabla_{i}\phi(t)\nabla_{j}\nabla_{k}
\phi(t)+\nabla_{j}\phi(t)
\nabla_{i}\nabla_{k}\phi(t)\right).
\end{eqnarray*}
\end{lemma}

\begin{proof} From the evolution equation for $\nabla_{i}\phi(t)$ obtained in the proof of Lemma \ref{l2.6}, we get
\begin{eqnarray*}
\partial_{t}\left(\nabla_{i}
\phi(t)\nabla_{j}\phi(t)\right)
&=&\nabla_{j}\phi(t)\bigg(\Delta_{g(t)}
\nabla_{i}\phi(t)-R_{ik}\nabla^{k}\phi(t)
+2\beta_{1}\nabla^{k}\phi(t)\nabla_{i}\nabla_{k}\phi(t)\\
&&+ \ \beta_{2}\nabla_{i}\phi(t)\bigg)
+\nabla_{i}\phi(t)\bigg(\Delta_{g(t)}
\nabla_{j}(t)-R_{jk}\nabla^{k}
\phi(t)\\
&&+ \ 2\beta_{1}\nabla^{k}\phi(t)
\nabla_{j}\nabla_{k}\phi(t)+\beta_{2}
\nabla_{j}\phi(t)\bigg)
\end{eqnarray*}
which implies the equation.
\end{proof}

\subsection{Regular flows on compact manifolds}\label{subsection2.2}

Let $(g(t),\phi(t))_{t\in[0,T)}$ be the solution of (\ref{2.4})--(\ref{2.5}) on a compact $m$-manifold $M$ with the initial value
$(\tilde{g},\tilde{\phi})$. Define
\begin{equation}
\tilde{c}:=\max_{M}|\nabla_{\tilde{g}}
\tilde{\phi}|^{2}_{\tilde{g}}, \ \ \
D:=\frac{1}{4}|2\beta_{1}-\alpha_{2}|^{2}-\alpha_{1}.
\label{2.7}
\end{equation}

\begin{proposition} Suppose $(g(t),\phi(t))_{t\in[0,T)}$ is the solution of (\ref{2.4})--(\ref{2.5}) on a compact $m$-manifold $M$ with the initial value $(\tilde{g},\tilde{\phi})$. Then we have
\begin{itemize}

\item[(1)] {\bf Case 1: $4\beta_{1}-2\alpha_{2}=0$.}

\begin{itemize}

\item[(1.1)] If $\alpha_{1}>0$ and $\beta_{2}>0$, then
\begin{equation*}
|\nabla_{g(t)}\phi(t)|^{2}_{g(t)}
\leq\frac{\tilde{c}\beta_{2}e^{2\beta_{2}t}}{\tilde{c}
\alpha_{1}e^{2\beta_{2}t}+(\beta_{2}-\tilde{c}\alpha_{1})}.
\end{equation*}

\item[(1.2)] If $\alpha_{1}>0$ and $\beta_{2}\leq0$, then
\begin{equation*}
|\nabla_{g(t)}\phi(t)|^{2}_{g(t)}
\leq\tilde{c}.
\end{equation*}

\item[(1.3)] If $\alpha_{1}=0$, then
\begin{equation*}
\left|\nabla_{g(t)}\phi(t)\right|^{2}_{g(t)}
\leq\tilde{c}e^{2\beta_{2}t}.
\end{equation*}

\item[(1.4)] If $\alpha_{1}<0$ and $\beta_{2}\leq0$, then
\begin{equation*}
\left|\nabla_{g(t)}
\phi(t)\right|^{2}_{g(t)}
\leq\frac{\tilde{c}}{1+2\alpha_{1}\tilde{c}t}.
\end{equation*}

\item[(1.5)] If $\alpha_{1}<0$ and $\beta_{2}>0$, then
\begin{equation*}
\left|\nabla_{g(t)}
\phi(t)\right|^{2}_{g(t)}
\leq\frac{\tilde{c}\beta_{2}e^{2\beta_{2}t}}{\beta_{2}
+\tilde{c}\alpha_{1}
\left(e^{2\beta_{2}t}-1\right)}.
\end{equation*}

\end{itemize}

\item[(2)] {\bf Case 2: $4\beta_{1}-2\alpha_{2}\neq0$.}

\begin{itemize}

\item[(2.1)] If $D<0$ and $\beta_{2}>0$, then
\begin{equation*}
|\nabla_{g(t)}\phi(t)|^{2}_{g(t)}
\leq\frac{\tilde{c}\beta_{2}e^{2\beta_{2}t}}{-\tilde{c}
De^{2\beta_{2}t}+(\beta_{2}+\tilde{c}D)}.
\end{equation*}

\item[(2.2)] If $D<0$ and $\beta_{2}\leq0$, then
\begin{equation*}
|\nabla_{g(t)}\phi(t)|^{2}_{g(t)}
\leq\tilde{c}.
\end{equation*}

\item[(2.3)] If $D=0$, then
\begin{equation*}
\left|\nabla_{g(t)}\phi(t)\right|^{2}_{g(t)}
\leq\tilde{c}e^{2\beta_{2}t}.
\end{equation*}

\item[(2.4)] If $D>0$ and $\beta_{2}\leq0$, then
\begin{equation*}
\left|\nabla_{g(t)}
\phi(t)\right|^{2}_{g(t)}
\leq\frac{\tilde{c}}{1-2D\tilde{c}t}.
\end{equation*}

\item[(2.5)] If $D>$ and $\beta_{2}>0$, then
\begin{equation*}
\left|\nabla_{g(t)}
\phi(t)\right|^{2}_{g(t)}
\leq\frac{\tilde{c}\beta_{2}e^{2\beta_{2}t}}{\beta_{2}
-\tilde{c}D\left(e^{2\beta_{2}t}-1\right)}.
\end{equation*}

\end{itemize}

\end{itemize}

\end{proposition}

\begin{proof} For any time $t$, we have
\begin{equation*}
\left\langle\nabla_{g(t)}
\phi(t)\otimes\nabla_{g(t)}\phi(t),\nabla^{2}_{g(t)}
\phi(t)\right\rangle_{g(t)}
\leq\left|\nabla_{g(t)}\phi(t)\right|^{2}_{g(t)}
\left|\nabla^{2}_{g(t)}\phi(t)\right|_{g(t)}.
\end{equation*}
By Lemma \ref{l2.6} we have
\begin{eqnarray*}
\partial_{t}|\nabla_{g(t)}\phi(t)|^{2}_{g(t)}
&\leq&\Delta_{g(t)}|\nabla_{g(t)}\phi(t)|^{2}_{g(t)}+2\beta_{2}|\nabla_{g(t)}
\phi(t)|^{2}_{g(t)}\\
&&- \ 2\left|\nabla^{2}_{g(t)}
\phi(t)\right|^{2}_{g(t)}-2\alpha_{1}|\nabla_{g(t)}\phi(t)|^{4}_{g(t)}\\
&&+ \ |4\beta_{1}-2\alpha_{2}||\nabla_{g(t)}
\phi(t)|^{2}\left|\nabla^{2}_{g(t)}
\phi(t)\right|.
\end{eqnarray*}
For convenience, set
\begin{equation*}
u(t):=|\nabla_{g(t)}\phi(t)|^{2}_{g(t)}, \ \ \ v(t):=\left|\nabla^{2}_{g(t)}
\phi(t)\right|_{g(t)}.
\end{equation*}
Then
\begin{equation*}
\partial_{t}u(t)\leq\Delta_{g(t)}u(t)+2\beta_{2}u(t)
-2v^{2}(t)-2\alpha_{1}u^{2}(t)+|4\beta_{1}-2\alpha_{2}|u(t)v(t).
\end{equation*}

(1) Case 1: $4\beta_{1}-2\alpha_{2}=0$. In this case, the above inequality becomes
\begin{equation*}
\partial_{t}u(t)\leq\Delta_{g(t)}u(t)
+2\beta_{2}u(t)-2\alpha_{1}u^{2}(t).
\end{equation*}
If $\alpha_{1}\geq0$, then
\begin{equation*}
\partial_{t}u(t)\leq\Delta_{g(t)}u(t)+2\beta_{2}u(t)
\end{equation*}
and
\begin{equation*}
\partial_{t}\big(e^{-2\beta_{2}t}u(t)\big)
=e^{-2\beta_{2}t}\big(-2\beta_{2}u(t)+\partial_{t}u(t)\big)
\leq\Delta\big(e^{-2\beta_{2}t}u(t)\big)
\end{equation*}
from which we obtain $u(t)\leq \tilde{c}e^{2\beta_{2}t}$ by the
maximum principle.

If $\alpha_{1}<0$ and $\beta_{2}\leq0$, then
\begin{equation*}
\partial_{t}u(t)\leq\Delta_{g(t)}u(t)-2\alpha_{1}u^{2}(t)
\end{equation*}
and $u_{t}\leq\frac{u(0)}{1+2\alpha_{1}u(0)t}\leq\frac{\tilde{c}}{1+
2\alpha_{1}\tilde{c}t}$. On the other hand, if $\beta_{2}>0$, then
\begin{equation*}
u(t)\leq\frac{\tilde{c}\beta_{2}e^{2\beta_{2}t}}{\beta_{2}
+\tilde{c}\alpha_{1}(e^{2\beta_{2}t}-1)}
\end{equation*}
since the solution to the ordinary differential equation
\begin{equation*}
\frac{dU(t)}{dt}=2\beta_{2}U(t)-2\alpha_{1}U^{2}(t), \ \ \
U(0)=\tilde{c}
\end{equation*}
is of the form
\begin{equation*}
U(t)=\frac{\tilde{c}\beta_{2}e^{2\beta_{2}t}}{\beta_{2}
+\tilde{c}\alpha_{1}(e^{2\beta_{2}t}-1)}.
\end{equation*}

(2) Case 2: $4\beta_{1}-2\alpha_{2}\neq0$. Using the inequality $ab\leq\epsilon a^{2}+\frac{1}{4\epsilon}
b^{2}$ for any nonnegative numbers $a,b$ and any positive number $\epsilon$, we obtain
\begin{eqnarray*}
\partial_{t}u(t)&\leq&\Delta_{g(t)}
u(t)+2\beta_{2}u(t)-2v^{2}(t)-2\alpha_{1}u^{2}(t)\\
&&+ \ |4\beta_{1}-2\alpha_{2}|
\bigg(\epsilon v^{2}(t)+\frac{1}{4\epsilon}u^{2}\bigg)\\
&=&\Delta_{g(t)}u(t)+2\beta_{2}u(t)
-\left(\epsilon|4\beta_{1}-2\alpha_{2}|-2\right)v^{2}(t)\\
&&+ \ \left(\frac{|4\beta_{1}-2\alpha_{2}|}{4\epsilon}
-2\alpha_{1}\right)u^{2}(t).
\end{eqnarray*}
Choosing $\epsilon:=2/|4\beta_{1}-2\alpha_{2}|$ implies
\begin{equation*}
\partial_{t}u(t)\leq\Delta_{g(t)}u(t)
+2\beta_{2}u(t)+2D u^{2}(t),
\end{equation*}
where $D$ is given in (\ref{2.7}). This is just the case (1) if we replace $\alpha_{1}$ by $-D$. The following discussion can be
obtained.
\end{proof}

\begin{corollary}\label{c2.10} Suppose $(g(t),\phi(t))_{t\in[0,T)}$ is the solution of (\ref{2.4})--(\ref{2.5}) on a compact $m$-manifold $M$ with the initial value $(\tilde{g},\tilde{\phi})$. If $\alpha_{1},\alpha_{2},\beta_{1},\beta_{2}$ satisfy one of
the conditions
\begin{itemize}

\item[(i)] $\beta_{2}\leq0$ and $4\alpha_{1}\geq(2\beta_{1}-\alpha_{2})^{2}$, or

\item[(ii)] $\beta_{2}>0$ and $\frac{4}{\tilde{c}}\beta_{2}+|2\beta_{1}
-\alpha_{2}|^{2}\geq4\alpha_{1}>(2\beta_{1}-\alpha_{2})^{2}$,

\end{itemize}
then
\begin{equation}
|\nabla_{g(t)}\phi(t)|^{2}_{g(t)}\leq C\label{2.8}
\end{equation}
on $M\times[0,T)$, where $C$ is a positive constant
depending only on $\alpha_{1}, \beta_{2}$, and $|\nabla_{\tilde{g}}
\tilde{\phi}|^{2}_{\tilde{g}}$.
\end{corollary}

In particular, we recover List's result \cite{L05} for $(\alpha_{1},
\alpha_{2},\beta_{1},\beta_{2})=(4,0,0,0)$. Anther example is $(\alpha_{1},\alpha_{2},\beta_{1},\beta_{2})=(\alpha_{1}, 2\beta_{1},\beta_{1},0)$, where $\alpha_{1}\geq0$ and $\beta_{1}
\in{\bf R}$.

\begin{definition}\label{d2.11} We say the flow (\ref{2.4})--(\ref{2.5}) is {\it regular}, if the constants $\alpha_{1},\alpha_{2},\beta_{1},\beta_{2}$ satisfy the conditions (i) or (ii) in Corollary \ref{c2.10}. An $(\alpha_{1},\alpha_{2},
\beta_{1},\beta_{2})$-flow is {\it $\star$-regular} if the associated $(\alpha_{1},0,
\beta_{1}-\alpha_{2},\beta_{2})$-flow (see Proposition \ref{p2.12}) is regular.
\end{definition}

Clearly that there are no relations between regular flows
and $\star$-regular in general. For example, $(1,1,0,0)$-flow is regular but
not $\star$-regular, while $(1,1,2,0)$-flow is $\star$-regular but not regular.

\subsection{Reduction to $(\alpha_{1},0,\beta_{1},
\beta_{2})$-flow}

Let $(\bar{g}(t),\bar{\phi}(t))$ be the solution of
(\ref{2.4})--(\ref{2.5}); that is,
\begin{eqnarray*}
\partial_{t}\bar{g}(t)&=&-2{\rm Ric}_{\bar{g}(t)}
+2\alpha_{1}\nabla_{\bar{g}(t)}\bar{\phi}(t)
\otimes\nabla_{\bar{g}(t)}\bar{\phi}(t)
+2\alpha_{2}\nabla^{2}_{\bar{g}(t)}
\bar{\phi}(t),\\
\partial_{t}\bar{\phi}(t)
&=&\Delta_{\bar{g}(t)}\bar{\phi}(t)
+\beta_{1}|\nabla_{\bar{g}(t)}\bar{\phi}(t)|^{2}_{
\bar{g}(t)}+\beta_{2}\bar{\phi}(t).
\end{eqnarray*}
Consider a $1$-parameter family of diffeomorphisms $\Phi(t):
M\to M$ by
\begin{equation}
\frac{d}{dt}\Phi(t)=-\alpha_{2}\nabla_{\bar{g}(t)}
\bar{\phi}(t), \ \ \ \Phi(0)={\rm Id}_{M}.\label{2.9}
\end{equation}
The above system of ODE is always solvable. Define
\begin{equation}
g(t):=[\Phi(t)]^{\ast}\bar{g}(t), \ \ \
\phi(t):=[\Phi(t)]^{\ast}\bar{\phi}(t).\label{2.10}
\end{equation}
Then
\begin{eqnarray*}
\partial_{t}g(t)&=&[\Phi(t)]^{\ast}\bigg(\partial_{t}
\bar{g}(t)\bigg)
-\alpha_{2}[\Phi(t)]^{\ast}\bigg(
\mathscr{L}_{\nabla_{\bar{g}(t)}\bar{\phi}(t)}
\bar{g}(t)\bigg)\\
&=&[\Phi(t)]^{\ast}
\bigg(-2{\rm Ric}_{\bar{g}(t)}
+2\alpha_{1}\nabla_{\bar{g}(t)}\bar{\phi}(t)
\otimes\nabla_{\bar{g}(t)}\bar{\phi}(t)
+2\alpha_{2}\nabla^{2}_{\bar{g}(t)}
\bar{\phi}(t)\bigg)\\
&&- \ 2\alpha_{2}[\Phi(t)]^{\ast}\nabla^{2}_{\bar{g}(t)}
\bar{\phi}(t)\\
&=&-2{\rm Ric}_{g(t)}+2\alpha_{1}\nabla_{g(t)}
\phi(t)\otimes\nabla_{g(t)}\phi(t),\\
\partial_{t}\phi(t)&=&\partial_{t}
\bigg([\Phi(t)]^{\ast}\bar{\phi}(t)\bigg)\\
&=&\partial_{t}\left(\bar{\phi}(t)\circ\Phi(t)\right)\\
&=&\partial_{t}\bar{\phi}(t)\circ\Phi(t)-\alpha_{2}|\nabla_{\bar{g}(t)}
\bar{\phi}(t)|^{2}_{\bar{g}(t)}\\
&=&\Delta_{g(t)}\phi(t)+(\beta_{1}-\alpha_{2})
|\nabla_{g(t)}\phi(t)|^{2}_{g(t)}+\beta_{2}\phi(t).
\end{eqnarray*}

\begin{proposition}\label{p2.12} Under a $1$-parameter family
of diffeomorphisms given by (\ref{2.9}), any solution of an $(\alpha_{1},\alpha_{2},\beta_{1},\beta_{2})$-flow is equivalent to a solution of $(\alpha_{1},0,\beta_{1}-\alpha_{2},\beta_{2})$-
flow.
\end{proposition}

\subsection{De Turck's trick}\label{subsection2.5}

By Proposition \ref{p2.12}, we suffice to study $(\alpha_{1},0,
\beta_{1},\beta_{2})$-flow:
\begin{eqnarray}
\partial_{t}g(t)&=&-2{\rm Ric}_{g(t)}
+2\alpha_{1}\nabla_{g(t)}\phi(t)\otimes\nabla_{g(t)}\phi(t),
\label{2.11}\\
\partial_{t}\phi(t)&=&\Delta_{g(t)}\phi(t)
+\beta_{1}\left|\nabla_{g(t)}\phi(t)\right|^{2}_{g(t)}
+\beta_{2}\phi(t).\label{2.12}
\end{eqnarray}

Let $(M,\tilde{g})$ be an $m$-dimensional compact or complete noncompact Riemannian manifold with $\tilde{g}
=\tilde{g}_{ij}dx^{i}\otimes dx^{j}$ and $\tilde{\phi}$ a
smooth function on $M$. Let $(\hat{g}(t),\hat{\phi}(t))_{t\in[0,T]}$ be a solution of (\ref{2.11})--(\ref{2.12}) with the initial data $(\tilde{g},\tilde{\phi})$, that is,
\begin{eqnarray}
\partial_{t}\hat{g}(t)&=&-2{\rm Ric}_{\hat{g}(t)}
+2\alpha_{1}\nabla_{\hat{g}(t)}
\hat{\phi}(t)\otimes\nabla_{\hat{g}(t)}
\hat{\phi}(t),\label{2.13}\\
\partial_{t}\hat{\phi}(t)
&=&\Delta_{\hat{g}(t)}
\hat{\phi}(t)+\beta_{1}
\left|\nabla_{\hat{g}(t)}
\hat{\phi}(t)\right|^{2}_{\hat{g}(t)}
+\beta_{2}\hat{\phi}(t)\label{2.14}
\end{eqnarray}
with $(\hat{g}(0),\hat{\phi}(0))=(\tilde{g},\tilde{\phi})$.

\begin{notation}\label{n2.13} If $\hat{g}(t)$ is a time-dependent Riemannian metric, its components are written as
$\hat{g}_{ij}$ or $\hat{g}_{ij}(x,t)$ when we want to
indicate space and time. The corresponding components of ${\rm Rm}_{\hat{g}(t)}, {\rm Ric}_{\hat{g}(t)}$ and $
\nabla_{\hat{g}(t)}$ are $\hat{R}_{ijk\ell}, \hat{R}_{ij}$
and $\hat{\nabla}_{i}$, respectively. In the form of local
components, we always omit space and time variables for
convenience.
\end{notation}

Locally, the system (\ref{2.13})--(\ref{2.14}) is of the form
\begin{equation}
\partial_{t}\hat{g}_{ij}=-2\hat{R}_{ij}
+2\alpha_{1}\hat{\nabla}_{i}\hat{\phi}
\hat{\nabla}_{j}\hat{\phi}, \ \ \
\partial_{t}\hat{\phi}=\hat{\Delta}\hat{\phi}
+\beta_{1}|\hat{\nabla}\hat{\phi}|^{2}_{\hat{g}}+\beta_{2}
\hat{\phi}\label{2.15}
\end{equation}
with $(\hat{g}_{ij}(0),\hat{\phi}(0))=(\tilde{g}_{ij},
\tilde{\phi})$. The system (\ref{2.15}) is not strictly parabolic even for the
case $\alpha_{1}=\beta_{1}=\beta_{2}$. As in the Ricci flow (see \cite{S89}) we consider one-parameter family of diffeomorphisms $(\Psi_{t})_{t\in[0,T]}$ on $M$ as follows: Let
\begin{equation}
g(t):=\Psi^{\ast}_{t}\hat{g}(t)=g_{ij}(x,t)dx^{i}
\otimes dx^{j}, \ \ \ \phi(t):=\Psi^{\ast}_{t}
\hat{\phi}(t), \ \ \ t\in[0,T]\label{2.16}
\end{equation}
and $\Psi_{t}(x):=y(x,t)$ be the solution of the
quasilinear first order system
\begin{equation}
\partial_{t}y^{\alpha}=\frac{\partial}{\partial x^{k}}
y^{\alpha}\cdot g^{\beta\gamma}\left(\Gamma^{k}_{\beta\gamma}-\tilde{\Gamma}^{k}_{\beta\gamma}
\right), \ \ \ y^{\alpha}(x,0)=x^{\alpha},\label{2.17}
\end{equation}
where $\Gamma$ and $\tilde{\Gamma}$ are Christoffel symbols
of $g$ and $\hat{g}$ respectively. As in \cite{L05, S89}, we have
\begin{equation}
\partial_{t}g_{ij}=-2R_{ij}+2\alpha_{1}\nabla_{i}
\phi\nabla_{j}\phi+\nabla_{i}V_{j}+\nabla_{j}V_{i},\label{2.18}
\end{equation}
where
\begin{equation*}
V_{i}:=g_{ik}g^{\beta\gamma}\left(\Gamma^{k}_{\beta\gamma}
-\tilde{\Gamma}^{k}_{\beta\gamma}\right).\label{2.19}
\end{equation*}
Similarly, we have
\begin{eqnarray*}
\partial_{t}\phi(x,t)&=&\partial_{t}\hat{\phi}(y,t) \ \
= \ \ \partial_{t}\hat{\phi}(y,t)+\frac{\partial}{\partial y^{\alpha}}\hat{\phi}(y,t)\partial_{t}y^{\alpha}\\
&=&\bigg(\hat{g}^{\alpha\beta}(y,t)
\hat{\nabla}_{\alpha}\hat{\nabla}_{\beta}
\hat{\phi}(y,t)+\beta_{1}
\hat{g}^{\alpha\beta}(y,t)
\hat{\nabla}_{\alpha}\hat{\phi}(y,t)
\hat{\nabla}_{\beta}\hat{\phi}(y,t)
+\beta_{2}\hat{\phi}(y,t)\bigg)\\
&&+ \ \frac{\partial}{\partial y^{\alpha}}
\hat{\phi}(y,t)\frac{\partial y^{\alpha}}{\partial x^{k}} g^{ik}V_{i}\\
&=&\frac{\partial}{\partial x^{k}}\hat{\phi}(y,t)\cdot g^{ik}V_{i}
+\beta_{1}g^{ij}\frac{\partial y^{\alpha}}{\partial x^{i}}
\frac{\partial y^{\beta}}{\partial x^{j}}
\hat{\nabla}_{\alpha}\hat{\phi}(y,t)\hat{\nabla}_{\beta}
\hat{\phi}(y,t)\\
&&+ \ \beta_{2}\hat{\phi}(y,t)+g^{ij}
\frac{\partial y^{\alpha}}{\partial x^{i}}
\frac{\partial y^{\beta}}{\partial x^{j}}\hat{\nabla}_{\alpha}
\hat{\nabla}_{\beta}\hat{\phi}(y,t)\\
&=&\Delta\phi(x,t)+\beta_{1}|\nabla\phi|^{2}_{g}+\beta_{2}\phi(x,t)
+\langle V,\nabla\phi\rangle_{g}.
\end{eqnarray*}
Here, $\Delta$ and $\nabla$ are Laplacian and Levi-Civita connection of $g$ accordingly. Hence, under the one-parameter family of
diffeomorphisms $(\Psi_{t})_{t\in[0,T]}$ on $M$, (\ref{2.15}) is equivalent to
\begin{eqnarray}
\partial_{t}g_{ij}&=&-2R_{ij}+2\alpha_{1}\nabla_{i}
\phi\nabla_{j}\phi+\nabla_{i}V_{j}+\nabla_{j}V_{i},\label{2.19}\\
\partial_{t}\phi&=&\Delta\phi+\beta_{1}|\nabla\phi|^{2}_{g}
+\beta_{2}\phi+\langle V,\nabla\phi\rangle_{g}\label{2.20}
\end{eqnarray}
with $(g_{ij}(0),\phi(0))=(\tilde{g}_{ij},\tilde{\phi})$.

\begin{lemma}\label{l2.14} the system (\ref{2.19})--(\ref{2.20}) is strictly parabolic. Moreover
\begin{eqnarray}
\partial_{t}g_{ij}
&=&g^{\alpha\beta}\tilde{\nabla}_{\alpha}
\tilde{\nabla}_{\beta}g_{ij}+g^{\alpha\beta}
g_{ip}\tilde{g}^{pq}\tilde{R}_{j\alpha q\beta}
+g^{\alpha\beta}g_{jp}\tilde{g}^{pq}\tilde{R}_{i\alpha q\beta}
\nonumber\\
&&+ \ \frac{1}{2}g^{\alpha\beta}g^{pq}\bigg(
\tilde{\nabla}_{i}g_{p\alpha}\tilde{\nabla}_{j}
g_{q\beta}
+2\tilde{\nabla}_{\alpha}g_{jp}\tilde{\nabla}_{q}g_{i\beta}
-2\tilde{\nabla}_{\alpha}g_{jp}
\tilde{\nabla}_{\beta}g_{iq}\label{2.21}\\
&&- \ 2\tilde{\nabla}_{j}g_{p\alpha}
\tilde{\nabla}_{\beta}g_{iq}
-2\tilde{\nabla}_{i}g_{p\alpha}\tilde{\nabla}_{\beta}
g_{jq}\bigg)+2\alpha_{1}\tilde{\nabla}_{i}\phi\tilde{\nabla}_{j}
\phi,\nonumber\\
\partial_{t}\phi&=&g^{ij}\tilde{\nabla}_{i}\tilde{\nabla}_{j}
\phi+\beta_{1}|\tilde{\nabla}\phi|^{2}_{g}+\beta_{2}\phi.\label{2.22}
\end{eqnarray}
\end{lemma}

\begin{proof} The first equation (\ref{2.21}) directly follows from the computations made in \cite{L05, S89} and the only difference is the sign of the Riemann curvature tensors used in this
paper. To (\ref{2.22}), we first observe that $\Delta\phi
+\langle V,\nabla\phi\rangle_{g}=g^{ij}\tilde{\nabla}_{i}
\tilde{\nabla}_{j}\phi$ as showed in \cite{L05}; then
\begin{equation*}
\partial_{t}\phi=g^{ij}\tilde{\nabla}_{i}
\tilde{\nabla}_{j}\phi+\beta_{1}g^{ij}\tilde{\nabla}_{i}
\phi\tilde{\nabla}_{j}\phi+\beta_{2}\phi
\end{equation*}
since $\nabla\phi=d\phi=\tilde{\nabla}\phi$.
\end{proof}

\section{Complete and noncompact case}\label{section3}

In this section we study the flow (\ref{2.4})--(\ref{2.5}) on
complete and noncompact Riemannian manifolds. The main result of this paper is

\begin{theorem}\label{t3.1} Let $(M,g)$ be an $m$-dimensional complete and noncompact
Riemannian manifold with $|{\rm Rm}_{g}|^{2}_{g}\leq k_{0}$ on $M$,
where $k_{0}$ is a positive constant, and let $\phi$ be a smooth function on $M$ satisfying $|\phi|^{2}+|\nabla_{g}\phi|^{2}_{g}\leq k_{1}$ and $|\nabla^{2}_{g}
\phi|^{2}_{g}
\leq k_{2}$. Then there exists a constant $T=
T(m,k_{0},k_{1})>0$, depending only on $m$ and $k_{0}, k_{1}$, such that any $\star$-regular $(\alpha_{1},\alpha_{2},\beta_{1},\beta_{2})$-flow (\ref{2.4})--(\ref{2.5}) has a smooth solution $(g(t),\phi(t))$ on $M\times[0,T]$ and satisfies
the following curvature estimate. For any nonnegative integer $n$, there
exist constants $C_{k}>0$, depending only on $m,n,k_{0}, k_{1}, k_{2}$, such that
\begin{equation}
\left|\nabla^{k}_{g(t)}{\rm Rm}_{g(t)}\right|^{2}_{g(t)}
\leq\frac{C_{k}}{t^{k}}, \ \ \ \left|\nabla^{k}_{g(t)}\phi(t)\right|^{2}_{g(t)}
\leq\frac{C_{k}}{t^{?}}\label{3.1}
\end{equation}
on $M\times[0,T]$.
\end{theorem}

By Proposition \ref{p2.12}, we suffice to study a regular
$(\alpha_{1},0,\beta_{1},\beta_{2})$-flow:
\begin{eqnarray*}
\partial_{t}\hat{g}(t)&=&-2{\rm Ric}_{\hat{g}(t)}+2\alpha_{1}\nabla_{
\hat{g}(t)}
\hat{\phi}(t)\otimes\nabla_{\hat{g}(t)}\hat{\phi}(t),\\
\partial_{t}\hat{\phi}(t)&=&\Delta_{\hat{g}(t)}
\hat{\phi}(t)+\beta_{1}|\nabla_{\hat{g}(t)}\hat{\phi}(t)|^{2}_{\hat{g}(t)}
+\beta_{2}\hat{\phi}(t),
\end{eqnarray*}
where $(\tilde{g},\tilde{\phi})=(\hat{g}(0),\hat{\phi}(0))$ is a fixed pair
consisting of a Riemannian metric $\hat{g}$ and a smooth function
$\hat{\phi}$. According to De Turck's trick, the above system of parabolic
partial differential equations are reduced to (\ref{2.10})--(\ref{2.20}).

Suppose that $D\subset M$ is a domain with boundary $\partial D$ a compact
smooth $(m-1)$-dimensional submanifold of $M$, and the closure $
\overline{D}:=D\cup\partial D$ is a compact subset of $M$. We shall shove the
following Dirichlet boundary problem:
\begin{eqnarray}
\partial_{t}g_{ij}&=&-2R_{ij}+2\alpha_{1}\nabla_{i}\phi\nabla_{j}
\phi+\nabla_{i}V_{j}+\nabla_{j}V_{i}, \ \ \ \text{in} \ D\times[0,T],\nonumber\\
\partial_{t}\phi&=&\Delta\phi+\beta_{1}|\nabla\phi|^{2}_{g}
+\beta_{2}\phi+\langle V,\nabla\phi\rangle_{g}, \ \ \ \text{in} \
D\times[0,T],\label{3.2}\\
(g_{ij},\phi)&=&(\tilde{g}_{ij},\tilde{\phi}), \ \ \ \text{on} \
D_{T}.\nonumber
\end{eqnarray}
where
\begin{equation}
D_{T}:=(D\times\{0\})\cup(\partial D\cup[0,T])\label{3.3}
\end{equation}
stands for the parabolic boundary of the domain $D\times[0,T]$.

Consider the assumption
\begin{equation}
|{\rm Rm}_{\tilde{g}}|^{2}_{\tilde{g}}
\leq k_{0}, \ \ \ |\nabla_{\tilde{g}}\tilde{\phi}|^{2}_{\tilde{\phi}}
\leq k_{1}.\label{3.4}
\end{equation}

\subsection{Zeroth order estimates}\label{subsection3.1}

Suppose that $(g_{ij},\phi)$ is a solution of (\ref{3.2}). For each positive
integer $n$, define
\begin{equation}
u=u(x,t):=g^{\alpha_{1}\beta_{1}}\tilde{g}_{\beta_{1}\alpha_{2}}
g^{\alpha_{2}\beta_{2}}\tilde{g}_{\beta_{2}\alpha_{3}}
\cdots g^{\alpha_{n}\beta_{n}}\tilde{g}_{\beta_{n}
\alpha_{1}}\label{3.5}
\end{equation}
on $D\times[0,T]$.

\begin{lemma}\label{l3.2} If $|{\rm Rm}_{\tilde{g}}|^{2}_{\tilde{g}}
\leq k_{0}$, then the function $u=u(x,t)$ satisfies
\begin{equation}
\partial_{t}u\leq g^{\alpha\beta}\tilde{\nabla}_{\alpha}
\tilde{\nabla}_{\beta}u+2nm\sqrt{k_{0}} u^{1+\frac{1}{n}} \ \ \
\text{in} \ D\times[0,T], \ \ \ u=m \ \ \ \text{on} \ D_{T}.\label{3.6}
\end{equation}
\end{lemma}

\begin{proof} The proof is identically similar to that of
\cite{L05, S89}; for completeness, we give a self-contained proof. Since $
\partial_{t}g^{ij}=-g^{ik}g^{j\ell}\partial_{t}g_{k\ell}$ and
$\tilde{\nabla}_{\beta}g^{ij}=-g^{ip}g^{jq}\tilde{\nabla}_{\beta}
g_{pq}$, it follows from (\ref{2.21}) that
\begin{eqnarray}
\partial_{t}g^{ij}&=&g^{\alpha\beta}\tilde{\nabla}_{\alpha}
\tilde{\nabla}_{\beta}g^{ij}+g^{\alpha\beta}
g^{j\ell}g^{ik}\tilde{\nabla}_{\alpha}g^{ik}\tilde{\nabla}_{\beta}
g_{k\ell}+g^{\alpha\beta}
g^{ik}\tilde{\nabla}_{\alpha}g^{j\ell}\tilde{\nabla}_{\beta}
g_{k\ell}\nonumber\\
&&- \ g^{\alpha\beta}g^{ik}
g^{j\ell}g_{kp}\tilde{g}^{pq}
\tilde{R}_{\ell\alpha q\beta}
-g^{\alpha\beta}g^{ik}g^{j\ell}
g_{p\ell}\tilde{g}^{pq}\tilde{R}_{k\alpha q\beta}
-2\alpha_{1}g^{ik}g^{j\ell}\tilde{\nabla}_{k}\phi\tilde{\nabla}_{\ell}\phi\nonumber\\
&&+ \ \frac{1}{2}g^{\alpha\beta}g^{pq}
g^{j\ell}\bigg(2\tilde{\nabla}_{\alpha}g_{p\ell}
\tilde{\nabla}_{\beta}g_{kq}
+2\tilde{\nabla}_{\ell}g_{p\alpha}
\tilde{\nabla}_{\beta}g_{kq}
+2\tilde{\nabla}_{k}g_{p\alpha}\tilde{\nabla}_{\beta}
g_{q\ell}\label{3.7}\\
&&- \ 2\tilde{\nabla}_{\alpha}g_{\ell p}
\tilde{\nabla}_{q}g_{k\beta}
-\tilde{\nabla}_{k}g_{p\alpha}
\tilde{\nabla}_{\ell}g_{q\beta}\bigg).\nonumber
\end{eqnarray}
Choosing a normal coordinate system such that
\begin{equation*}
\tilde{g}_{ij}=\delta_{ij}, \ \ \ g_{ij}=\lambda_{i}\delta_{ij},
\end{equation*}
we conclude from (\ref{3.7}) that
\begin{eqnarray}
\partial_{t}g^{ij}
&=&g^{\alpha\beta}\tilde{\nabla}_{\alpha}
\tilde{\nabla}_{\beta}g^{ij}
-\frac{2}{\lambda_{i}\lambda_{j}\lambda_{k}\lambda_{\ell}}
\tilde{\nabla}_{\ell}g_{ik}
\tilde{\nabla}_{\ell}g_{jk}
-\frac{1}{\lambda_{i}\lambda_{k}}\tilde{R}_{ikjk}
-\frac{1}{\lambda_{j}\lambda_{k}}\tilde{R}_{ikjk}\nonumber\\
&&+\ \frac{1}{2\lambda_{i}\lambda_{j}\lambda_{k}\lambda_{\ell}}
\bigg(2\tilde{\nabla}_{k}g_{\ell j}
\tilde{\nabla}_{k}g_{i\ell}
+2\tilde{\nabla}_{j}g_{\ell k}\tilde{\nabla}_{k}g_{i\ell}
+2\tilde{\nabla}_{i}g_{\ell k}
\tilde{\nabla}_{k}g_{\ell j}\label{3.8}\\
&&- \ 2\tilde{\nabla}_{k}g_{j\ell}
\tilde{\nabla}_{\ell}g_{ik}-\tilde{\nabla}_{i}g_{\ell k}
\tilde{\nabla}_{j}g_{\ell k}\bigg)-\frac{2\alpha_{1}}{\lambda_{i}
\lambda_{j}}\tilde{\nabla}_{i}\phi\tilde{\nabla}_{j}\phi.\nonumber
\end{eqnarray}
From $u=\sum^{n}_{i=1}(1/\lambda_{i})^{n}$, we obtain
\begin{eqnarray*}
\partial_{t}u&=&n\left(\frac{1}{\lambda_{i}}\right)^{n-1}
\partial_{t}g^{ii}\\
&=&\bigg(\frac{n}{\lambda^{n-1}_{i}}
g^{\alpha\beta}\tilde{\nabla}_{\alpha}
\tilde{\nabla}_{\beta}g^{ii}-\frac{2n}{\lambda^{n}_{i}
\lambda_{k}}\tilde{R}_{ikik}-\frac{n}{\lambda^{n+1}_{i}
\lambda_{k}\lambda_{\ell}}\tilde{\nabla}_{\ell}
g_{ik}\tilde{\nabla}_{\ell}g_{ik}\\
&&+ \ \frac{n}{2\lambda^{n+1}_{i}
\lambda_{k}\lambda_{\ell}}\left(4\tilde{\nabla}_{i}
g_{\ell k}\tilde{\nabla}_{k}g_{i\ell}
-2\tilde{\nabla}_{k}g_{i\ell}
\tilde{\nabla}_{\ell}g_{ik}-\tilde{\nabla}_{i}
g_{\ell k}\tilde{\nabla}_{i}g_{\ell k}\right)\\
&&- \ \frac{2\alpha_{1}n}{\lambda^{n+1}_{i}}|\nabla_{\tilde{g}}
\phi|^{2}_{\tilde{g}}\bigg)\\
&=&\frac{n}{\lambda^{n-1}_{i}}g^{\alpha\beta}
\tilde{\nabla}_{\alpha}\tilde{\nabla}_{\beta}
g^{ii}-\frac{2n}{\lambda^{n}_{i}\lambda_{k}}
\tilde{R}_{ikik}-\frac{2\alpha_{1}n}{\lambda^{n+1}_{i}}
|\nabla_{\tilde{g}}\phi|^{2}_{\tilde{g}}\\
&&- \ \frac{n}{2\lambda^{n+1}_{i}\lambda_{k}
\lambda_{\ell}}\left(\tilde{\nabla}_{k}g_{i\ell}
+\tilde{\nabla}_{\ell}g_{ik}-\tilde{\nabla}_{i}
g_{\ell k}\right)^{2}.
\end{eqnarray*}
On the other hand, it is not hard to see that
\begin{equation*}
g^{\alpha\beta}
\tilde{\nabla}_{\alpha}\tilde{\nabla}_{\beta}
u=\frac{n}{\lambda^{n-1}_{i}}
g^{\alpha\beta}\tilde{\nabla}_{\alpha}
\tilde{\nabla}_{\beta}g^{ii}
+ng^{\alpha\beta}\tilde{\nabla}_{\alpha}g^{ij}
\tilde{\nabla}_{\beta}g^{ij}
\sum^{n-2}_{a=0}\frac{1}{\lambda^{n-2-a}_{i}\lambda^{a}_{j}},
\end{equation*}
and then
\begin{eqnarray}
\partial_{t}u&=&g^{\alpha\beta}\tilde{\nabla}_{\alpha}
\tilde{\nabla}_{\beta}u-\frac{n}{\lambda_{k}}
\sum^{n-2}_{a=0}\frac{1}{\lambda^{n-2-a}_{i}
\lambda^{a}_{j}}\left(\tilde{\nabla}_{k}g_{ij}\right)^{2}
-\frac{2n}{\lambda^{n}_{i}\lambda_{k}}
\tilde{R}_{ikik}\nonumber\\
&&- \ \frac{2\alpha_{1}n}{\lambda^{n+1}_{i}}
|\nabla_{\tilde{g}}\phi|^{2}_{\tilde{g}}-\frac{n}{2\lambda^{n+1}_{i}\lambda_{k}\lambda_{\ell}}
\left(\tilde{\nabla}_{k}g_{i\ell}+\tilde{\nabla}_{\ell}
g_{ik}-\tilde{\nabla}_{i}g_{\ell k}\right)^{2}.\label{3.9}
\end{eqnarray}
Since $|{\rm Rm}_{\tilde{g}}|_{\tilde{g}}\leq\sqrt{k_{0}}$, it follows
from (\ref{3.9}) that
\begin{equation*}
\partial_{t}u\leq g^{\alpha\beta}\tilde{\nabla}_{\alpha}
\tilde{\nabla}_{\beta}u+2n\sqrt{k_{0}}\left(\sum^{m}_{j=1}
\frac{1}{\lambda_{j}}\right)u.
\end{equation*}
According to H\"older's inequality
\begin{equation*}
\sum^{m}_{j=1}\frac{1}{\lambda_{j}}\leq\left(\sum^{m}_{j=1}
\frac{1}{\lambda^{n}_{j}}\right)^{1/n}
\left(\sum^{m}_{j=1}1^{n'}\right)^{1/n'}
=mu^{1/n}, \ \ \ \frac{1}{n}+\frac{1}{n'}=1;
\end{equation*}
therefore $\partial_{t}u\leq g^{\alpha\beta}
\tilde{\nabla}_{\alpha}\tilde{\nabla}_{\beta}u
+2nm\sqrt{k_{0}}u^{1+\frac{1}{n}}$.
\end{proof}

As showed in \cite{L05, S89}, the lower bound of $g_{ij}$ now directly follows from Lemma \ref{l3.2}.

\begin{lemma}\label{l3.3} If $|{\rm Rm}_{\tilde{g}}|^{2}_{\tilde{g}}
\leq k_{0}$, then, for any $\delta\in(0,1)$, we have
\begin{equation}
g(t)\geq(1-\delta)\tilde{g}\label{3.10}
\end{equation}
on $D\times[0,T_{-}(\delta,m,k_{0})]$, where $T_{-}(\delta,m,
k_{0}):=\frac{1}{2\sqrt{k_{0}}}(\frac{1}{m})^{1+1/n}
[1-(\frac{1}{2})^{1/n}]$, $n$ is a positive integer, and
\begin{equation*}
\frac{\ln(2m)}{\ln(1/(1-\delta))}\leq n<
\frac{\ln(2m)}{\ln(1/(1-\delta))}+1.
\end{equation*}
\end{lemma}

Since we consider the regular $(\alpha_{1},0,\beta_{1},\beta_{2})$-flow, we conclude from Corollary \ref{2.10} that $|\nabla_{\hat{g}(t)}
\hat{\phi}(t)|^{2}_{\hat{g}(t)}\leq C$, where $C$ is a positive
constant depending only on $\alpha_{1}, \beta_{2}$, and $|\nabla_{\tilde{g}}
\tilde{\phi}|^{2}_{\tilde{g}}$. Following the arguments in
\cite{L05, S89}, we have an upper bound of $g_{ij}$.

\begin{lemma}\label{l3.4} If $|\widetilde{{\rm Rm}}|^{2}_{\tilde{g}}
\leq k_{0}$ and $|\tilde{\nabla}\tilde{\phi}|^{2}_{\tilde{g}}
\leq k_{1}$, then, for any $\theta>0$, we have
\begin{equation}
g(t)\leq(1+\theta)\tilde{g}\label{3.11}
\end{equation}
on $D\times[0,T_{+}(\theta,m,k_{0},k_{1},\alpha_{1},\beta_{2})]$, where
$T_{+}(\delta,m,k_{0},k_{1},\alpha_{1},\beta_{2})$ is a positive constant
depending only on $\theta,m,k_{0},k_{1},\alpha_{1}$, and $\beta_{2}$.
\end{lemma}

From Lemma \ref{l3.3} and Lemma \ref{l3.4}, we have

\begin{theorem}\label{t3.5} Suppose that $|\widetilde{{\rm Rm}}
|^{2}_{\tilde{g}}\leq k_{0}$ and $|\tilde{\nabla}\tilde{\phi}|^{2}_{\tilde{g}}
\leq k_{1}$ on $M$. If $(g(t),\phi(t))$ is a solution of (\ref{3.2}), then,
for any $\epsilon\in(0,1)$, we have
\begin{equation}
(1-\epsilon)\tilde{g}\leq g(t)\leq(1+\epsilon)\tilde{g}\label{3.12}
\end{equation}
on $\overline{D}\times[0,T(\epsilon,m,k_{0},k_{1},\alpha_{1},\beta_{2})]$,
where $T(\epsilon,m,k_{0},k_{1},\alpha_{1},\beta_{2})$ is a positive
constant depending only on $\epsilon, m, k_{0}, k_{1}, \alpha_{1},
\beta_{2}$.
\end{theorem}

\subsection{Existence of the De Turck flow}\label{subsection3.2}

We establish the short time existence of the De Turck flow (\ref{3.2})
on the whole manifold $M$. Fix a point $x_{0}\in M$ and let
$B_{\tilde{g}}(x_{0},r)$ be the metric ball of radius $r$ centered at $x_{0}$
with respect to the metric $\tilde{g}$.

\begin{lemma}\label{l3.6} Given positive constants $r,\delta,T$. Suppose that $(g(t),
\phi(t))$ is a solution of (\ref{3.2}) on $B_{\tilde{g}}(x_{0},
r+\delta)\times[0,T]$, that is,
\begin{eqnarray*}
\partial_{t}g_{ij}&=&-2R_{ij}+2\alpha_{1}
\nabla_{i}\phi\nabla_{j}\phi+\nabla_{i}V_{j}+\nabla_{j}V_{i}, \ \ \
\text{in} \ B_{\tilde{g}}(x_{0},r+\delta)\times[0,T],\nonumber\\
\partial_{t}\phi&=&\Delta\phi+\beta_{1}|\nabla\phi|^{2}_{g}
+\beta_{2}\phi+\langle V,\nabla\phi\rangle_{g}, \ \ \
\text{in} \ B_{\tilde{g}}(x_{0},r+\delta)\times[0,T],\\
(g_{ij},\phi)&=&(\tilde{g}_{ij},\tilde{\phi}), \ \ \ \text{on} \
D_{T},\nonumber
\end{eqnarray*}
and $|\widetilde{{\rm Rm}}|^{2}_{\tilde{g}}
\leq k_{0}, |\tilde{\nabla}\tilde{\phi}|^{2}_{\tilde{g}}
\leq k_{1}$ on $M$. If
\begin{equation}
\left(1-\frac{1}{80000(1+\alpha^{2}_{1}+\beta^{2}_{1})m^{10}}\right)\tilde{g}\leq g(t)
\leq\left(1+\frac{1}{80000(1+\alpha^{2}_{1}+\beta^{2}_{1})m^{10}}
\right)\tilde{g}\label{3.13}
\end{equation}
on $B_{\tilde{g}}(x_{0},r+\delta)\times[0,T]$, then there exists a positive
constant $C=C(m,r,\delta,T, \tilde{g},k_{1})$ depending only on $m, r,
\delta, T, \tilde{g}$, and $k_{1}$, such that
\begin{equation}
|\tilde{\nabla}g|^{2}_{\tilde{g}}\leq C, \ \ \ |\tilde{\nabla}
\phi|^{2}_{\tilde{g}}\leq C\label{3.14}
\end{equation}
on $B_{\tilde{g}}(x_{0},r+\frac{\delta}{2})\times[0,T]$.
\end{lemma}

\begin{proof} Using the $\ast$-notion, we can write (\ref{2.21}) as
\begin{equation*}
\partial_{t}g_{ij}=g^{\alpha\beta}\tilde{\nabla}_{\alpha}
\tilde{\nabla}_{\beta}g_{ij}+g^{-1}\ast g\ast\widetilde{{\rm Rm}}
+g^{-1}\ast g^{-1}\ast\tilde{\nabla}g\ast\tilde{\nabla}g
+2\alpha_{1}\tilde{\nabla}_{i}\phi\tilde{\nabla}_{j}\phi.
\end{equation*}
Then
\begin{eqnarray}
\partial_{t}\tilde{\nabla}g_{ij}&=&g^{\alpha\beta}\left(\nabla
\tilde{\nabla}_{\alpha}\tilde{\nabla}_{\beta}g_{ij}\right)
+g^{-1}\ast g^{-1}\ast\tilde{\nabla}g\ast\tilde{\nabla}\tilde{\nabla}g
+2\alpha_{1}\tilde{\nabla}_{i}\phi\tilde{\nabla}\tilde{\nabla}_{j}\phi\nonumber\\
&&+ \ g^{-1}\ast g^{-1}\ast\tilde{\nabla}g\ast g\ast\widetilde{{\rm Rm}}
+g^{-1}\ast\tilde{\nabla}g\ast\widetilde{{\rm Rm}}\label{3.15}\\
&&+ \ g^{-1}\ast g\ast\tilde{\nabla}\widetilde{{\rm Rm}}
+g^{-1}\ast g^{-1}\ast g^{-1}\ast\tilde{\nabla}g
\ast\tilde{\nabla}g\ast\tilde{\nabla}g;\nonumber
\end{eqnarray}
since
\begin{eqnarray*}
\tilde{\nabla}\tilde{\nabla}_{\alpha}\tilde{\nabla}_{\beta}g_{ij}
&=&\tilde{\nabla}_{\alpha}\tilde{\nabla}\tilde{\nabla}_{\beta}
g_{ij}+\widetilde{{\rm Rm}}\ast\tilde{\nabla}g\\
&=&\tilde{\nabla}_{\alpha}\left(\tilde{\nabla}_{\beta}
\tilde{\nabla}g_{ij}+\widetilde{{\rm Rm}}\ast g\right)
+\widetilde{{\rm Rm}}\ast\tilde{\nabla}g\\
&=&\tilde{\nabla}_{\alpha}\tilde{\nabla}_{\beta}
\tilde{\nabla}g_{ij}+g\ast\tilde{\nabla}\widetilde{{\rm Rm}}
+\tilde{\nabla}g\ast\widetilde{{\rm Rm}},
\end{eqnarray*}
we conclude from (\ref{3.15}) that
\begin{eqnarray}
\partial_{t}\tilde{\nabla}g_{ij}
&=&g^{\alpha\beta}\left(\tilde{\nabla}_{\alpha}
\tilde{\nabla}_{\beta}\tilde{\nabla}g_{ij}\right)
+g^{-1}\ast g^{-1}\ast\tilde{\nabla}g\ast\tilde{\nabla}
\tilde{\nabla}g+2\alpha_{1}\tilde{\nabla}_{i}
\phi\tilde{\nabla}\tilde{\nabla}_{j}
\phi\nonumber\\
&&+ \ g^{-1}\ast g^{-1}\ast\tilde{\nabla}g
\ast g\ast\widetilde{{\rm Rm}}+g^{-1}\ast\tilde{\nabla}g\ast\widetilde{{\rm Rm}}\label{3.16}\\
&&+ \ g^{-1}\ast g\ast\tilde{\nabla}\widetilde{{\rm Rm}}
+g^{-1}\ast g^{-1}\ast g^{-1}\ast\tilde{\nabla}g
\ast\tilde{\nabla}g\ast\tilde{\nabla}g.\nonumber
\end{eqnarray}
It follows from (\ref{3.16}) that
\begin{eqnarray}
\partial_{t}|\tilde{\nabla}g|^{2}_{\tilde{g}}
&=&2\tilde{g}^{\alpha\beta}\tilde{g}^{ik}
\tilde{g}^{j\ell}\tilde{\nabla}_{\beta}g_{k\ell}\cdot\partial_{t}
\tilde{\nabla}_{\alpha}g_{ij}\nonumber\\
&=&g^{\alpha\beta}\tilde{\nabla}_{\alpha}
\tilde{\nabla}_{\beta}|\tilde{\nabla}g|^{2}_{\tilde{g}}
-2g^{\alpha\beta}\left\langle\tilde{\nabla}_{\alpha}\tilde{\nabla}g,
\tilde{\nabla}_{\beta}\tilde{\nabla}g\right\rangle_{\tilde{g}}
+\tilde{\nabla}g\ast\tilde{\nabla}\phi\ast\tilde{\nabla}
\tilde{\nabla}\phi\nonumber\\
&&+\ \widetilde{{\rm Rm}}\ast g^{-1}\ast g^{-1}
\ast g\ast\tilde{\nabla}g\ast\tilde{\nabla}g
+\widetilde{{\rm Rm}}\ast g^{-1}\ast\tilde{\nabla}g
\ast\tilde{\nabla}g\label{3.17}\\
&&+ \ g^{-1}\ast g\ast\tilde{\nabla}
\widetilde{{\rm Rm}}\ast\tilde{\nabla}g
+g^{-1}\ast g^{-1}\ast\tilde{\nabla}g\ast\tilde{\nabla}g
\ast\tilde{\nabla}\tilde{\nabla}g\nonumber\\
&&+ \ g^{-1}\ast g^{-1}\ast g^{-1}
\ast\tilde{\nabla}g\ast\tilde{\nabla}g
\ast\tilde{\nabla}g\ast\tilde{\nabla}g.\nonumber
\end{eqnarray}
Since the closure $\overline{B_{\tilde{g}}(x_{0},r+\delta)}$ is compact,
we have
\begin{equation}
|\tilde{\nabla}\widetilde{{\rm Rm}}|_{\tilde{g}}\lesssim1\label{3.18}
\end{equation}
on $\overline{B_{\tilde{g}}(x_{0},r+\delta)}$, where
$\lesssim$ depends on $r, \delta, \tilde{g}$. From (\ref{3.13}) we get
\begin{equation}
\frac{1}{2}\tilde{g}\leq g(t)\leq 2\tilde{g}\label{3.19}
\end{equation}
on $B_{\tilde{g}}(x_{0},r+\delta)\times[0,T]$. According to (\ref{3.18})
and (\ref{3.19}), we arrive at
\begin{eqnarray}
\widetilde{{\rm Rm}}\ast g^{-1}\ast g^{-1}
\ast\tilde{\nabla}g\ast\tilde{\nabla}g&\lesssim&|\tilde{\nabla}
g|^{2}_{\tilde{g}},\nonumber\\
\widetilde{{\rm Rm}}\ast g^{-1}
\ast\tilde{\nabla}g\ast\tilde{\nabla}g&\lesssim
&|\tilde{\nabla}g|^{2}_{\tilde{g}},\label{3.20}\\
\tilde{\nabla}\widetilde{{\rm Rm}}\ast g^{-1}
\ast g\ast\tilde{\nabla}g&\lesssim&|\tilde{\nabla}g|_{\tilde{g}},\nonumber
\end{eqnarray}
where $\lesssim$ depends on $m, r, \delta, \tilde{g}$. From the explicit
formulas we can see
\begin{eqnarray*}
\tilde{\nabla}g\ast\tilde{\nabla}\phi\ast\tilde{\nabla}
\tilde{\nabla}\phi&=&4\alpha_{1}
\tilde{g}^{\alpha\beta}\tilde{g}^{ik}\tilde{g}^{j\ell}
\tilde{\nabla}_{\beta}g_{k\ell}\tilde{\nabla}_{i}
\phi\tilde{\nabla}_{\alpha}\tilde{\nabla}_{j}\phi\\
&=&4\alpha_{1}\sum_{1\leq i,j,\alpha\leq m}
\tilde{\nabla}_{\alpha}g_{ij}\tilde{\nabla}_{i}
\phi\tilde{\nabla}_{\alpha}\tilde{\nabla}_{j}\phi\\
&\leq&4\alpha_{1}m^{3}|\tilde{\nabla}g|_{\tilde{g}}
|\tilde{\nabla}\phi|_{\tilde{g}}|\tilde{\nabla}\tilde{\nabla}
\phi|_{\tilde{g}}
\end{eqnarray*}
where we used a normal coordinate system of $\tilde{g}$. Similarly,
\begin{eqnarray*}
&&g^{-1}\ast g^{-1}\ast\tilde{\nabla}g
\ast\tilde{\nabla}g\ast\tilde{\nabla}\tilde{\nabla}g\\
&=&\tilde{g}^{\gamma\delta}
g^{ik}\tilde{g}^{j\ell}g^{\alpha\beta}
g^{pq}\tilde{\nabla}_{\delta}g_{k\ell}
\bigg(\tilde{\nabla}_{\gamma}\tilde{\nabla}_{i}
g_{p\alpha}\tilde{\nabla}_{j}g_{q\beta}+\tilde{\nabla}_{i}g_{p\alpha}
\tilde{\nabla}_{\gamma}\tilde{\nabla}_{j}
g_{q\beta}+2\tilde{\nabla}_{\gamma}
\tilde{\nabla}_{\alpha}g_{jp}\tilde{\nabla}_{q}g_{i\beta}\\
&&+ \ 2\tilde{\nabla}_{\alpha}g_{jp}\tilde{\nabla}_{\gamma}
\tilde{\nabla}_{q}g_{i\beta}
-2\tilde{\nabla}_{\gamma}\tilde{\nabla}_{\alpha}
g_{jp}\tilde{\nabla}_{\beta}g_{iq}-2\tilde{\nabla}_{\alpha}
g_{jp}\tilde{\nabla}_{\gamma}\tilde{\nabla}_{\beta}
g_{iq}-2\tilde{\nabla}_{\gamma}\tilde{\nabla}_{j}
g_{p\alpha}\tilde{\nabla}_{\beta}g_{iq}\\
&&- \ 2\tilde{\nabla}_{j}g_{p\alpha}
\tilde{\nabla}_{\gamma}\tilde{\nabla}_{\beta}
g_{iq}-2\tilde{\nabla}_{\gamma}\tilde{\nabla}_{i}
g_{p\alpha}\tilde{\nabla}_{\beta}
g_{jq}-2\tilde{\nabla}_{i}g_{p\alpha}
\tilde{\nabla}_{\gamma}\tilde{\nabla}_{\beta}
g_{jq}\bigg)\\
&=&\sum_{1\leq i,j,\alpha,\gamma,p\leq m}g^{\alpha\alpha}
g^{pp}\tilde{\nabla}_{\gamma}g_{ij}\bigg(\tilde{\nabla}_{\gamma}
\tilde{\nabla}_{i}g_{p\alpha}\tilde{\nabla}_{j}g_{p\alpha}
+\tilde{\nabla}_{i}g_{p\alpha}\tilde{\nabla}_{j\gamma}
g_{p\alpha}\\
&&+ \ 2\tilde{\nabla}_{\gamma}\tilde{\nabla}_{\alpha}
g_{jp}\tilde{\nabla}_{p}g_{i\alpha}
+2\tilde{\nabla}_{\alpha}g_{jp}\tilde{\nabla}_{\gamma}
\tilde{\nabla}_{p}g_{i\alpha}-2\tilde{\nabla}_{\gamma}
\tilde{\nabla}_{\alpha}g_{jp}\tilde{\nabla}_{\alpha}
g_{ip}\\
&&- \ 2\tilde{\nabla}_{\alpha}g_{jp}
\tilde{\nabla}_{\gamma}\tilde{\nabla}_{\alpha}g_{ip}-
2\tilde{\nabla}_{\gamma}\tilde{\nabla}_{j}g_{p\alpha}
\tilde{\nabla}_{\alpha}g_{ip}-2\tilde{\nabla}_{j}
g_{p\alpha}\tilde{\nabla}_{\gamma}\tilde{\nabla}_{\alpha}g_{ip}\\
&&- \ 2\tilde{\nabla}_{\gamma}\tilde{\nabla}_{i}g_{p\alpha}
\tilde{\nabla}_{\alpha}g_{jp}-2\tilde{\nabla}_{i}g_{p\alpha}
\tilde{\nabla}_{\gamma}\tilde{\nabla}_{\alpha}g_{jp}\bigg)\\
&\leq&4m^{5}(8\times 2+2)|\tilde{\nabla}g|^{2}_{\tilde{g}}
|\tilde{\nabla}\tilde{\nabla}g|_{\tilde{g}} \ \ = \ \
72 m^{5}|\tilde{\nabla}g|^{2}_{\tilde{g}}
|\tilde{\nabla}\tilde{\nabla}g|_{\tilde{g}},
\end{eqnarray*}
and
\begin{eqnarray*}
&& g^{-1}\ast g^{-1}\ast g^{-1}\ast\tilde{\nabla}g
\ast\tilde{\nabla}g\ast\tilde{\nabla}g\ast\tilde{\nabla}g\\
&=&\tilde{g}^{\gamma\delta}
\tilde{g}^{ik}\tilde{g}^{j\ell}
\tilde{\nabla}_{\delta}g_{k\ell}
\bigg(\tilde{\nabla}_{\gamma}g^{\alpha\beta}
g^{pq}+g^{\alpha\beta}\tilde{\nabla}_{\gamma}g^{pq}\bigg)
\bigg(\tilde{\nabla}_{i}g_{p\alpha}
\tilde{\nabla}_{j}g_{q\beta}
+2\tilde{\nabla}_{\alpha}g_{jp}\tilde{\nabla}_{q}g_{i\beta}\\
&&- \ 2\tilde{\nabla}_{\alpha}g_{jp}
\tilde{\nabla}_{\beta}g_{iq}-2\tilde{\nabla}_{j}g_{p\alpha}
\tilde{\nabla}_{\beta}g_{iq}-2\tilde{\nabla}_{i}g_{p\alpha}
\tilde{\nabla}_{\beta}g_{jq}\bigg)\\
&\leq&2^{3}n^{6}(4\times 2+1)\times 2|\tilde{\nabla}g|^{4}_{\tilde{g}} \ \ = \ \
144 m^{6}|\tilde{\nabla}g|^{4}_{\tilde{g}}.
\end{eqnarray*}
Thus
\begin{eqnarray}
\tilde{\nabla}g\ast\tilde{\nabla}\phi\ast\tilde{\nabla}
\tilde{\nabla}\phi&\leq&4\alpha_{1}m^{3}|\tilde{\nabla}g|_{\tilde{g}}
|\tilde{\nabla}\phi|_{\tilde{g}}|\tilde{\nabla}\tilde{\nabla}
\phi|_{\tilde{g}},\nonumber\\
g^{-1}\ast g^{-1}\ast\tilde{\nabla}g\ast\tilde{\nabla}g
\ast\tilde{\nabla}\tilde{\nabla}g&\leq&72 m^{5}|\tilde{\nabla}g|^{2}_{\tilde{g}}
|\tilde{\nabla}\tilde{\nabla}g|_{\tilde{g}},\label{3.21}\\
g^{-1}\ast g^{-1}\ast g^{-1}\ast\tilde{\nabla}g
\ast\tilde{\nabla}g\ast\tilde{\nabla}g\ast\tilde{\nabla}g&\leq&
144 m^{6}|\tilde{\nabla}g|^{4}_{\tilde{g}}.\nonumber
\end{eqnarray}
Furthermore, using(\ref{3.19}), we get
\begin{equation}
g^{\alpha\beta}\left\langle\tilde{\nabla}_{\alpha}
\tilde{\nabla}g,\tilde{\nabla}_{\beta}\tilde{\nabla}g\right\rangle_{\tilde{g}}
=g^{\alpha\beta}\tilde{g}^{ik}\tilde{g}^{j\ell}
\tilde{g}^{\gamma\delta}\tilde{\nabla}_{\alpha}
\tilde{\nabla}_{\gamma}g_{ij}
\tilde{\nabla}_{\beta}\tilde{\nabla}_{\delta}
g_{k\ell}\geq\frac{1}{2}|\tilde{\nabla}\tilde{\nabla}g|^{2}_{\tilde{g}}.\label{3.22}
\end{equation}
Substituting (\ref{3.20}), (\ref{3.21}), and (\ref{3.22}) into
(\ref{3.16}) implies
\begin{eqnarray}
\partial_{t}|\tilde{\nabla}g|^{2}_{\tilde{g}}
&\leq&g^{\alpha\beta}\tilde{\nabla}_{\alpha}
\tilde{\nabla}_{\beta}|\tilde{\nabla}g|^{2}_{\tilde{g}}
-|\tilde{\nabla}^{2}g|^{2}_{\tilde{g}}
+C_{1}|\tilde{\nabla}g|^{2}_{\tilde{g}}
+C_{1}|\tilde{\nabla}g|_{\tilde{g}}\nonumber\\
&&+ \ 72 m^{5}|\tilde{\nabla}g|^{2}_{\tilde{g}}
|\tilde{\nabla}^{2}g|_{\tilde{g}}+144 m^{6}|\tilde{\nabla}g|^{4}_{\tilde{g}}
+4\alpha_{1} m^{3}|\tilde{\nabla}g|_{\tilde{g}}|\tilde{\nabla}\phi|_{\tilde{g}}
|\tilde{\nabla}^{2}\phi|_{\tilde{g}}\label{3.23}
\end{eqnarray}
for some positive constant $C_{1}$ depending only on $m,r,\delta,\tilde{g}$.

Using (\ref{2.22}), we have
\begin{eqnarray*}
\partial_{t}\tilde{\nabla}_{k}\phi&=&\tilde{\nabla}_{k}
\bigg(g^{ij}\tilde{\nabla}_{i}\tilde{\nabla}_{j}\phi+\beta_{1}
|\tilde{\nabla}\phi|^{2}_{g}+\beta_{2}\phi\bigg)\\
&=&\tilde{\nabla}_{k}g^{ij}\tilde{\nabla}_{i}
\tilde{\nabla}_{j}\phi+g^{ij}\tilde{\nabla}_{k}\tilde{\nabla}_{i}
\tilde{\nabla}_{j}\phi+\beta_{1}\tilde{\nabla}_{k}
\left(g^{ij}\tilde{\nabla}_{i}\phi\tilde{\nabla}_{j}\phi\right)
+\beta_{2}\tilde{\nabla}_{k}\phi;
\end{eqnarray*}
from the Ricci identity,
\begin{equation*}
\tilde{\nabla}_{k}\tilde{\nabla}_{i}\tilde{\nabla}_{j}
\phi=\tilde{\nabla}_{i}\tilde{\nabla}_{k}\tilde{\nabla}_{j}\phi
-\tilde{R}_{kijp}\tilde{\nabla}^{\ell}\phi
=\tilde{\nabla}_{i}\tilde{\nabla}_{j}\tilde{\nabla}_{k}
\phi-\tilde{R}_{kijp}\tilde{\nabla}^{p}\phi,
\end{equation*}
we obtain
\begin{eqnarray}
\partial_{t}\tilde{\nabla}_{k}\phi&=&g^{ij}\tilde{\nabla}_{i}
\tilde{\nabla}_{j}(\tilde{\nabla}_{k}\phi)
-g^{ij}\tilde{R}_{kij p}\tilde{\nabla}^{p}\phi+
\tilde{\nabla}_{k}g^{ij}\tilde{\nabla}_{i}\tilde{\nabla}_{j}
\phi\nonumber\\
&&+ \ \beta_{1}\tilde{\nabla}_{k}g^{ij}\tilde{\nabla}_{i}
\phi\tilde{\nabla}_{j}\phi+2\beta_{1}g^{ij}\tilde{\nabla}_{j}\phi
\tilde{\nabla}_{i}\tilde{\nabla}_{k}\phi+\beta_{2}\tilde{\nabla}_{k}
\phi.\label{3.24}
\end{eqnarray}
The evolution equation (\ref{3.24}) gives us the following equation
\begin{eqnarray}
\partial_{t}|\tilde{\nabla}\phi|^{2}_{\tilde{g}}
&=&\partial_{t}\left(\tilde{g}^{k\ell}\tilde{\nabla}_{k}\phi\tilde{\nabla}_{\ell}
\phi\right) \ \ = \ \ 2\tilde{g}^{k\ell}\tilde{\nabla}_{\ell}\phi
\left(\partial_{t}\tilde{\nabla}_{k}\phi\right)\nonumber\\
&=&g^{ij}\left(\tilde{\nabla}_{i}\tilde{\nabla}_{j}
(\tilde{\nabla}_{k}\phi)\cdot 2\tilde{g}^{k\ell}\tilde{\nabla}_{\ell}
\phi\right)-2g^{ij}\tilde{g}^{k\ell}
\tilde{R}_{kijp}\tilde{\nabla}_{\ell}\phi\tilde{\nabla}^{p}\phi\nonumber\\
&&+ \ 2\tilde{g}^{k\ell}\tilde{\nabla}_{k}g^{ij}\tilde{\nabla}_{\ell}
\phi\tilde{\nabla}_{i}\tilde{\nabla}_{j}\phi
+2\beta_{1}\tilde{g}^{k\ell}\tilde{\nabla}_{k}g^{ij}\tilde{\nabla}_{i}
\phi\tilde{\nabla}_{j}\phi\tilde{\nabla}_{\ell}\phi\label{3.25}\\
&&+ \ 4\beta_{1}\tilde{g}^{k\ell}g^{ij}
\tilde{\nabla}_{j}\phi\tilde{\nabla}_{\ell}\phi
\tilde{\nabla}_{i}\tilde{\nabla}_{k}\phi
+2\beta_{2}\tilde{g}^{k\ell}\tilde{\nabla}_{k}
\phi\tilde{\nabla}_{\ell}\phi.\nonumber
\end{eqnarray}
The identity
\begin{equation*}
g^{ij}\tilde{\nabla}_{i}\tilde{\nabla}_{j}
|\tilde{\nabla}\phi|^{2}_{\tilde{g}}
=g^{ij}\left(\tilde{\nabla}^{i}\tilde{\nabla}_{j}(\tilde{\nabla}_{k}
\phi)\cdot 2\tilde{g}^{k\ell}\tilde{\nabla}_{\ell}
\phi\right)+2g^{ij}\tilde{g}^{k\ell}\tilde{\nabla}_{i}\tilde{\nabla}_{k}
\phi\tilde{\nabla}_{j}\tilde{\nabla}_{\ell}\phi
\end{equation*}
together with (\ref{3.18}), implies that
\begin{equation}
g^{ij}\left(\tilde{\nabla}_{i}\tilde{\nabla}_{j}(
\tilde{\nabla}_{k}\phi)\cdot 2\tilde{g}^{k\ell}\tilde{\nabla}_{\ell}
\phi\right)\leq g^{ij}\tilde{\nabla}_{i}\tilde{\nabla}_{j}
|\tilde{\nabla}\phi|^{2}_{\tilde{g}}
-|\tilde{\nabla}^{2}\phi|^{2}_{\tilde{g}}.\label{3.26}
\end{equation}
Using (\ref{3.18}) again, we find that
\begin{eqnarray*}
-2g^{ij}\tilde{g}^{k\ell}
\tilde{R}_{kijp}\tilde{\nabla}_{\ell}\phi
\tilde{\nabla}^{p}\phi&=&
-2g^{ij}\tilde{g}^{k\ell}\tilde{g}^{pq}
\tilde{R}_{kijp}\tilde{\nabla}_{\ell}\phi\tilde{\nabla}_{q}\phi\\
&\lesssim&|\tilde{\nabla}\phi|^{2}_{\tilde{g}},\\
2\tilde{g}^{k\ell}\tilde{\nabla}_{k}g^{ij}
\tilde{\nabla}_{\ell}\phi\tilde{\nabla}_{i}\tilde{\nabla}_{j}
\phi&=&-2\tilde{g}^{k\ell}g^{ip}g^{qj}
\tilde{\nabla}_{k}g_{pq}\tilde{\nabla}_{\ell}\phi\tilde{\nabla}_{i}
\tilde{\nabla}_{j}\phi\\
&\leq&8 m^{3}|\tilde{\nabla}g|_{\tilde{g}}|\tilde{\nabla}
\phi|_{\tilde{g}}|\tilde{\nabla}^{2}\phi|_{\tilde{g}},\\
\tilde{g}^{k\ell}\tilde{\nabla}_{k}
g^{ij}\tilde{\nabla}_{i}\phi\tilde{\nabla}_{j}\phi
\tilde{\nabla}_{\ell}\phi&=&-\tilde{g}^{k\ell}
g^{ip}g^{jq}\tilde{\nabla}_{k}g_{pq}
\tilde{\nabla}_{i}\phi\tilde{\nabla}_{j}\phi\tilde{\nabla}_{\ell}\phi\\
&\leq& 4m^{3}|\tilde{\nabla}g|_{\tilde{g}}
|\tilde{\nabla}\phi|^{3}_{\tilde{g}},\\
\tilde{g}^{k\ell}
g^{ij}\tilde{\nabla}_{j}\phi\tilde{\nabla}_{\ell}
\phi\tilde{\nabla}_{i}\tilde{\nabla}_{k}\phi&\leq&4 m^{2}
|\tilde{\nabla}\phi|^{2}_{\tilde{g}}|\tilde{\nabla}^{2}\phi|_{\tilde{g}},
\end{eqnarray*}
where $\lesssim$ depends on $m, r, \delta, \tilde{g}$. Substituting
those estimates and (\ref{3.26}) into (\ref{3.25}), we arrive at
\begin{eqnarray}
\partial_{t}|\tilde{\nabla}\phi|^{2}_{\tilde{g}}
&\leq&g^{ij}\tilde{\nabla}_{i}\tilde{\nabla}_{j}
|\tilde{\nabla}\phi|^{2}_{\tilde{g}}
-|\tilde{\nabla}^{2}\phi|^{2}_{\tilde{g}}
+C_{2}|\tilde{\nabla}\phi|^{2}_{\tilde{g}}
+8m^{3}|\tilde{\nabla}g|_{\tilde{g}}
|\tilde{\nabla}\phi|_{\tilde{g}}|\tilde{\nabla}^{2}\phi|_{\tilde{g}}\nonumber\\
&&+ \ 8\beta_{1}m^{3}|\tilde{\nabla}g|_{\tilde{g}}
|\tilde{\nabla}\phi|^{3}_{\tilde{g}}
+16\beta_{1}m^{2}|\tilde{\nabla}\phi|^{2}_{\tilde{g}}
|\tilde{\nabla}^{2}\phi|_{\tilde{g}}
+2\beta_{2}|\tilde{\nabla}\phi|^{2}_{\tilde{g}}.\label{3.27}
\end{eqnarray}
As in \cite{L05}, we consider the vector-valued tensor field
\begin{equation}
\boldsymbol{\Theta}(t):=(g(t),\phi(t))\label{3.28}
\end{equation}
and define
\begin{equation*}
\tilde{\nabla}^{k}\boldsymbol{\Theta}(t):=\left(\tilde{\nabla}^{k}g(t),
\tilde{\nabla}^{k}\phi(t)\right), \ \ \
|\tilde{\nabla}^{k}\boldsymbol{\Theta}(t)|^{2}_{\tilde{g}}
:=|\tilde{\nabla}^{k}g(t)|^{2}_{\tilde{g}}
+|\tilde{\nabla}^{k}\phi(t)|^{2}_{\tilde{g}}.\label{3.29}
\end{equation*}
From (\ref{3.23}) and (\ref{3.27}), we obtain
\begin{eqnarray}
\partial_{t}|\tilde{\nabla}\boldsymbol{\Theta}|^{2}_{\tilde{g}}
&\leq&g^{\alpha\beta}\tilde{\nabla}_{\alpha}\tilde{\nabla}_{\beta}
|\tilde{\nabla}\boldsymbol{\Theta}|^{2}_{\tilde{g}}-|\tilde{\nabla}^{2}
\boldsymbol{\Theta}|^{2}_{\tilde{g}}+
C_{3}|\tilde{\nabla}\boldsymbol{\Theta}|^{2}_{\tilde{g}}
+C_{3}|\tilde{\nabla}\boldsymbol{\Theta}|_{\tilde{g}}\nonumber\\
&&+ \ (80+4\alpha_{1}+16\beta_{1})m^{5}|\tilde{\nabla}\boldsymbol{\Theta}
|^{2}_{\tilde{g}}|\tilde{\nabla}^{2}\boldsymbol{\Theta}|_{\tilde{g}}
+(144+8\beta_{1})m^{6}|\tilde{\nabla}\boldsymbol{\Theta}|^{4}_{\tilde{g}},
\label{3.29}
\end{eqnarray}
where $C_{3}$ is positive constant depending on $m, r, \delta, \tilde{g}$
and $\beta_{2}$. The inequality (\ref{3.29}) is similar to the equation (11) in
page 247 of \cite{S89} and the equation (3.28) in page 36 of \cite{L05}, so that
the proof is essentially without change anything. However, for our flow, we want to find the
positive constant $\epsilon>0$ with $(1-\epsilon)\tilde{g}\leq g\leq(1+\epsilon)
\tilde{g}$ on $B_{\tilde{g}}(x_{0},r+\delta)\times[0,T]$ such that both functions $|\tilde{\nabla}
g|^{2}_{\tilde{g}}$ and $|\tilde{\nabla}\phi|^{2}_{\tilde{g}}$ are bounded
from above. The mentioned positive constant $\epsilon$ depends on $m, \alpha_{1},
\beta_{1}$, and $\beta_{2}$; the aim of the following computations is to
find an explicit formula for $\epsilon$. We shall follow Shi's idea but do
more slight work on calculus, in particular on the positive uniform constants we are going to
obtain.

By the inequality
\begin{equation}
ab\leq\epsilon a^{2}+\frac{1}{4\epsilon}b^{2}, \ \ \ a,b\in{\bf R}, \ \ \
\epsilon>0,\label{3.30}
\end{equation}
we have
\begin{equation*}
(80+4\alpha_{1}+16\beta_{1})m^{5}
|\tilde{\nabla}\boldsymbol{\Theta}|^{2}_{\tilde{g}}
|\tilde{\nabla}^{2}\boldsymbol{\Theta}|_{\tilde{g}}
\leq\frac{1}{2}|\tilde{\nabla}^{2}\boldsymbol{\Theta}|^{2}_{\tilde{g}}
+2(40+2\alpha_{1}+8\beta_{1})^{2}m^{10}|\tilde{\nabla}
\boldsymbol{\Theta}|^{4}_{\tilde{g}}
\end{equation*}
and
\begin{equation*}
C_{3}|\tilde{\nabla}\boldsymbol{\Theta}|_{\tilde{g}}
\leq C_{3}\frac{1+|\tilde{\nabla}\boldsymbol{\Theta}|^{2}_{\tilde{g}}}{2}
=\frac{C_{3}}{2}+\frac{C_{3}}{2}|\tilde{\nabla}\boldsymbol{\Theta}|^{2}_{\tilde{g}}.
\end{equation*}
As a consequence of (\ref{3.29}), we conclude that
\begin{eqnarray}
\partial_{t}|\tilde{\nabla}\boldsymbol{\Theta}|^{2}_{\tilde{g}}
&\leq&g^{\alpha\beta}\tilde{\nabla}_{\alpha}
\tilde{\nabla}_{\beta}|\tilde{\nabla}\boldsymbol{\Theta}|^{2}_{\tilde{g}}
-\frac{1}{2}|\tilde{\nabla}^{2}\boldsymbol{\Theta}|^{2}_{\tilde{g}}
+C_{4}|\tilde{\nabla}\boldsymbol{\Theta}|^{2}_{\tilde{g}}
+C_{4}+\left(3344+8\alpha^{2}_{1}
\right.\nonumber\\
&& \ +\left.320\alpha_{1}+64\alpha_{1}\beta_{1}+128\beta^{2}_{1}+1280\beta_{1}+8\beta_{1}\right)
m^{10}|\tilde{\nabla}\boldsymbol{\Theta}|^{4}_{\tilde{g}}\label{3.31}
\end{eqnarray}
for some positive constant $C_{4}$ depending only on $m, r, \delta,\tilde{g}$ and $\beta_{2}$.

Given $\epsilon:=\frac{1}{Am^{10}}\leq\frac{1}{2}$, where $A$ is a
positive constant depending on $\alpha_{1}, \beta_{1}$ and chosen later. If we choose a normal
coordinate system so that
\begin{equation*}
\tilde{g}_{ij}=\delta_{ij}, \ \ \ g_{ij}=\lambda_{i}\delta_{ij},
\end{equation*}
then we have
\begin{equation}
1-\epsilon\leq\lambda_{k}\leq 1+\epsilon, \ \ \
\frac{1}{2}\leq\lambda_{k}\leq 2, \ \ \ k=1,\cdots,m.\label{3.32}
\end{equation}
Define
\begin{equation}
n:=\frac{1}{\epsilon}, \ \ \ a:=\frac{n}{4},\label{3.33}
\end{equation}
and
\begin{equation}
\varphi=\varphi(x,t):=a+\sum_{1\leq k\leq m}
\lambda^{n}_{k}, \ \ \ (x,t)\in B_{\tilde{g}}(x_{0},r+\delta)\times[0,T].\label{3.34}
\end{equation}
By the formula (16) in page 248 of \cite{S89}, we can compute
\begin{eqnarray}
\partial_{t}\varphi&=&n\sum_{1\leq k,\alpha,\beta\leq m}
\lambda^{n-1}_{k}g^{\alpha\beta}\tilde{\nabla}_{\alpha}
\tilde{\nabla}_{\beta}g_{kk}+2n\sum_{1\leq k,\alpha\leq m}
\lambda^{n-1}_{k}\frac{\lambda_{k}}{\lambda_{\alpha}}
\tilde{R}_{k\alpha k\alpha}\nonumber\\
&&+ \ \frac{n}{2}\sum_{1\leq k,\alpha,p\leq m}
\frac{\lambda^{n-1}_{k}}{\lambda_{\alpha}\lambda_{p}}
\bigg(\tilde{\nabla}_{k}g_{p\alpha}
\tilde{\nabla}_{k}g_{p\alpha}+2\tilde{\nabla}_{\alpha}g_{kp}
\tilde{\nabla}_{p}g_{k\alpha}-2\tilde{\nabla}_{\alpha}
g_{kp}\tilde{\nabla}_{\alpha}g_{kp}\label{3.35}\\
&&- \ 2\tilde{\nabla}_{k}g_{p\alpha}
\tilde{\nabla}_{\alpha}g_{kp}-2\tilde{\nabla}_{k}g_{p\alpha}
\tilde{\nabla}_{\alpha}g_{kp}\bigg).\nonumber
\end{eqnarray}
Since the second and the third on the right-hand side of (\ref{3.35}) is bounded
from above by $4n(1+\epsilon)^{n}m^{2}\sqrt{k_{0}}$ and
\begin{equation*}
\frac{n}{2}m^{3}(1+\epsilon)^{n-1}
4(4\times 2+1)|\tilde{\nabla}g|^{2}_{\tilde{g}}
=18 nm^{3}(1+\epsilon)^{n-1}|\tilde{\nabla}g|^{2}_{\tilde{g}},
\end{equation*}
respectively, it follows that
\begin{equation}
\partial_{t}\varphi\leq n\sum_{1\leq k,\alpha,\beta\leq m}
\lambda^{n-1}_{k}g^{\alpha\beta}\tilde{\nabla}_{\alpha}
\tilde{\nabla}_{\beta}g_{kk}+C_{5}+18 nm^{3}(1+\epsilon)^{n-1}
|\tilde{\nabla}g|^{2}_{\tilde{g}}\label{3.36}
\end{equation}
where $C_{5}$ is a positive constant depending only on $m, A$, and $k_{0}$.
On the other hand, from (\ref{3.36}) and the equation (19) in page 249 of \cite{S89}, we have
\begin{equation}
\partial_{t}\varphi
\leq g^{\alpha\beta}\tilde{\nabla}_{\alpha}
\tilde{\nabla}_{\beta}\varphi+C_{5}+\left[18 m^{3}n
(1+\epsilon)^{n-1}-\frac{n(n-1)}{2}(1-\epsilon)^{n-2}\right]
|\tilde{\nabla}g|^{2}_{\tilde{g}}.\label{3.37}
\end{equation}
Instead of the inequalities (20), (21), and (22) in page 249
of \cite{S89}, we will prove uniform inequalities as follows (recall that
$\epsilon:=1/n$):
\begin{itemize}

\item[(a)] For any $n\geq2$ and any $m$, we have
\begin{equation}
18 m^{3}n(1+\epsilon)^{n-1}\leq\frac{54 m^{3}}{n+1}n^{2}.\label{3.38}
\end{equation}

\item[(b)] For any $n\geq2$, we have
\begin{equation}
(1-\epsilon)^{n-2}\geq\frac{1}{4}.\label{3.39}
\end{equation}

\item[(c)] For any $n\geq2$, we have
\begin{equation}
\frac{n(n-1)}{2}(1-\epsilon)^{n-2}\geq\frac{n^{2}}{8}.\label{3.40}
\end{equation}

\end{itemize}
To prove (\ref{3.38}), we write
\begin{equation*}
18 m^{3}n(1+\epsilon)^{n-1}=18 nm^{3}\frac{n}{n+1}\left(1+\frac{1}{n}
\right)^{n}=\frac{18 m^{3}n^{2}}{n+1}\left(1+\frac{1}{n}\right)^{n};
\end{equation*}
since the function $(1+1/x)^{x}$, $x>0$, is in increasing in $x$, it follows
that
\begin{equation*}
18 m^{3}n(1+\epsilon)^{n-1}\leq\frac{18 m^{3}n^{2}}{n+1}
e\leq\frac{54 m^{3}}{n+1}n^{2}.
\end{equation*}
Since the function $(1-1/x)^{x}$, $x\geq 2$, is increasing in $x$, we obtain
\begin{equation*}
(1-\epsilon)^{n-2}=\left(1-\frac{1}{n}\right)^{n-2}
=\frac{n^{2}}{(n-1)^{2}}\left(1-\frac{1}{n}\right)^{n}
\geq\left(1-\frac{1}{n}\right)^{n}\geq\left(1-\frac{1}{2}\right)^{2}
=\frac{1}{4}.
\end{equation*}
From the proof of (\ref{3.39}), together with (\ref{3.38}), we arrive at
\begin{equation*}
\frac{n(n-1)}{2}(1-\epsilon)^{n-2}
\geq\frac{n(n-1)}{2}\frac{n^{2}}{(n-1)^{2}}\frac{1}{4}
=\frac{n^{2}}{8}\frac{n}{n-1}\geq\frac{n^{2}}{8}.
\end{equation*}
Substituting (\ref{3.38}) and (\ref{3.40}) into (\ref{3.37}) implies
\begin{equation*}
\partial_{t}\varphi\leq g^{\alpha\beta}\tilde{\nabla}_{\alpha}
\tilde{\nabla}_{\beta}\varphi+C_{5}-\left(\frac{1}{8}-\frac{54m^{3}}{n+1}
\right)n^{2}|\tilde{\nabla}g|^{2}_{\tilde{g}};
\end{equation*}
choosing
\begin{equation}
n\geq 864 m^{10}>864 m^{3},\label{3.41}
\end{equation}
we find that
\begin{equation}
\partial_{t}\varphi\leq g^{\alpha\beta}\tilde{\nabla}_{\alpha}
\tilde{\nabla}_{\beta}\varphi+C_{5}-\frac{n^{2}}{16}|\tilde{\nabla}
g|^{2}_{\tilde{g}}.\label{3.42}
\end{equation}
As the equation (24) in page 249 of \cite{S89}, we have
\begin{eqnarray}
\partial_{t}\left(\varphi|\tilde{\nabla}\boldsymbol{\Theta}|^{2}_{\tilde{g}}
\right)&\leq&g^{\alpha\beta}\tilde{\nabla}_{\alpha}
\tilde{\nabla}_{\beta}\left(\varphi|\tilde{\nabla}\boldsymbol{\Theta}|^{2}_{\tilde{g}}\right)
-2g^{\alpha\beta}\tilde{\nabla}_{\alpha}\varphi\tilde{\nabla}_{\beta}
|\tilde{\nabla}\boldsymbol{\Theta}|^{2}_{\tilde{g}}
-\frac{1}{2}\varphi|\tilde{\nabla}^{2}\boldsymbol{\Theta}|^{2}_{\tilde{g}}
\nonumber\\
&&+ \ \left(3344+8\alpha^{2}_{1}+320\alpha_{1}+64\alpha_{1}\beta_{1}
+128\beta^{2}_{1}+1280\beta_{1}+8\beta_{1}\right)\times\label{3.43}\\
&&m^{10}\varphi
|\tilde{\nabla}\boldsymbol{\Theta}|^{4}_{\tilde{g}}+C_{4}\varphi|\tilde{\nabla}\boldsymbol{\Theta}|^{2}_{\tilde{g}}
+C_{4}\varphi+C_{5}|\tilde{\nabla}\boldsymbol{\Theta}|^{2}_{\tilde{g}}
-\frac{n^{2}}{16}|\tilde{\nabla}\boldsymbol{\Theta}|^{2}_{\tilde{g}}
|\tilde{\nabla}g|^{2}_{\tilde{g}}.\nonumber
\end{eqnarray}
According to Corollary \ref{c2.10}, we get $|\nabla\phi|^{2}_{g}\lesssim1$ and hence
\begin{equation*}
|\tilde{\nabla}\phi|^{2}_{\tilde{g}}
=\tilde{g}^{\alpha\beta}\tilde{\nabla}_{\alpha}
\tilde{\nabla}_{\beta}\phi\leq 2|\nabla\phi|^{2}_{g}\lesssim1
\end{equation*}
where $\lesssim$ depends on $\alpha_{1}, \beta_{2}$ and $k_{1}$. Consequently,
(\ref{3.43}) can be written as
\begin{eqnarray}
\partial_{t}\left(\varphi|\tilde{\nabla}\boldsymbol{\Theta}|^{2}_{\tilde{g}}
\right)&\leq&g^{\alpha\beta}\tilde{\nabla}_{\alpha}
\tilde{\nabla}_{\beta}\left(\varphi|\tilde{\nabla}\boldsymbol{\Theta}|^{2}_{\tilde{g}}\right)
-2g^{\alpha\beta}\tilde{\nabla}_{\alpha}\varphi\tilde{\nabla}_{\beta}
|\tilde{\nabla}\boldsymbol{\Theta}|^{2}_{\tilde{g}}
-\frac{1}{2}\varphi|\tilde{\nabla}^{2}\boldsymbol{\Theta}|^{2}_{\tilde{g}}
\nonumber\\
&&+ \ \left(3344+8\alpha^{2}_{1}+320\alpha_{1}+64\alpha_{1}\beta_{1}
+128\beta^{2}_{1}+1280\beta_{1}+8\beta_{1}\right)\times\label{3.44}\\
&&m^{10}\varphi
|\tilde{\nabla}\boldsymbol{\Theta}|^{4}_{\tilde{g}}+C_{4}\varphi|\tilde{\nabla}\boldsymbol{\Theta}|^{2}_{\tilde{g}}
+C_{4}\varphi+C_{6}|\tilde{\nabla}\boldsymbol{\Theta}|^{2}_{\tilde{g}}
-\frac{n^{2}}{16}|\tilde{\nabla}\boldsymbol{\Theta}|^{4}_{\tilde{g}},\nonumber
\end{eqnarray}
where $C_{6}$ is a positive constant depending only on $m, A, k_{0}, n, \alpha_{1},
\beta_{2}$, and $k_{1}$. From (\ref{3.32}) and (\ref{3.34}),
\begin{equation}
a+m(1-\epsilon)^{n}\leq\varphi\leq a+m(1+\epsilon)^{n}\label{3.45}
\end{equation}
on $B_{\tilde{g}}(x_{0},r+\delta)\times[0,T]$, we arrive at (recall from
(\ref{3.41}) that $n=A m^{10}$ with $A\geq 864$)
\begin{equation*}
Cm^{10}\varphi \leq C m^{10}\left[\frac{n}{4}+m\left(1+
\frac{1}{n}\right)^{n}\right]\leq Cm^{10}\left(\frac{n}{4}+3m\right)
\leq C m^{10}\frac{n}{2}=\frac{n^{2}}{2A/C},
\end{equation*}
where
\begin{equation}
C:=3344+8\alpha^{2}_{1}+320\alpha_{1}+64\alpha_{1}\beta_{1}
+128\beta^{2}_{1}+1280\beta_{1}+8\beta_{1}.\label{3.46}
\end{equation}
If we choose
\begin{equation}
A\geq 16 C,\label{3.47}
\end{equation}
then
\begin{equation}
Cm^{10}\varphi|\tilde{\nabla}\boldsymbol{\Theta}|^{4}_{\tilde{g}}
-\frac{n^{2}}{16}|\tilde{\nabla}\boldsymbol{\Theta}|^{4}_{\tilde{g}}
\leq-\frac{n^{2}}{32}|\tilde{\nabla}\boldsymbol{\Theta}|^{4}_{\tilde{g}}.
\label{3.48}
\end{equation}
On the other hand, by the argument in the proof of (28) in page 250 of \cite{S89}, we have
\begin{equation}
-2g^{\alpha\beta}\tilde{\nabla}_{\alpha}\varphi
\tilde{\nabla}_{\beta}|\tilde{\nabla}\boldsymbol{\Theta}|^{2}_{\tilde{g}}
\leq\frac{\varphi}{2}|\tilde{\nabla}^{2}\boldsymbol{\Theta}|^{2}_{\tilde{g}}
+\frac{288 n^{2} m^{10}}{\varphi}|\tilde{\nabla}
\boldsymbol{\Theta}|^{4}_{\tilde{g}}.\label{3.49}
\end{equation}
Plugging (\ref{3.48}) and (\ref{3.49}) into (\ref{3.44}) implies
\begin{eqnarray}
\partial_{t}\left(\varphi|\tilde{\nabla}\boldsymbol{\Theta}|^{2}_{\tilde{g}}
\right)&\leq&g^{\alpha\beta}\tilde{\nabla}_{\alpha}\tilde{\nabla}_{\beta}
\left(\varphi|\tilde{\nabla}\boldsymbol{\Theta}|^{2}_{\tilde{g}}
\right)+\frac{288 n^{2}m^{10}}{\varphi}|\tilde{\nabla}
\boldsymbol{\Theta}|^{4}_{\tilde{g}}\nonumber\\
&&- \ \frac{n^{2}}{32}|\tilde{\nabla}\boldsymbol{\Theta}|^{4}_{\tilde{g}}
+C_{4}\varphi+(C_{4}+C_{6})\varphi|\tilde{\nabla}
\boldsymbol{\Theta}|^{2}_{\tilde{g}},\label{3.50}
\end{eqnarray}
because $\varphi\geq a=\frac{n}{4}\geq1$. According to $\varphi\geq a=
\frac{n}{4}=\frac{A}{4}m^{10}$, we conclude that
\begin{equation*}
\frac{28 n^{2}m^{10}}{\varphi}
\leq 1152 nm^{10}=\frac{n^{2}}{A/1152}\leq\frac{n^{2}}{64},
\end{equation*}
where we choose
\begin{equation}
A\geq 1152\times 64=73728.\label{3.51}
\end{equation}
Consequently,
\begin{equation}
\partial_{t}\left(\varphi|\tilde{\nabla}
\boldsymbol{\Theta}|^{2}_{\tilde{g}}\right)
\leq g^{\alpha\beta}\tilde{\nabla}_{\alpha}
\tilde{\nabla}_{\beta}\left(\varphi|\tilde{\nabla}\boldsymbol{\Theta}
|^{2}_{\tilde{g}}\right)-\frac{n^{2}}{64}|\tilde{\nabla}
\boldsymbol{\Theta}|^{4}_{\tilde{g}}+C_{4}\varphi+
(C_{4}+C_{6})\varphi|\tilde{\nabla}\boldsymbol{\Theta}|^{2}_{\tilde{g}}.
\label{3.52}
\end{equation}
Using the following inequality
\begin{eqnarray*}
\varphi&\leq& a+m(1+\epsilon)^{n} \ \ = \ \
\frac{n}{4}+m\left(1+\frac{1}{n}\right)^{n} \ \
\leq \ \ \frac{n}{4}+3m\\
&\leq&\frac{n}{4}+\frac{3n}{A} \ \ = \ \
\left(\frac{1}{4}+\frac{3}{A}\right)n \ \
\leq \ \ \frac{18435}{73728}n \ \ \leq \ \ 0.26 n,
\end{eqnarray*}
by (\ref{3.51}), we get
\begin{equation*}
\frac{n^{2}}{64}|\tilde{\nabla}
\boldsymbol{\Theta}|^{4}_{\tilde{g}}
=\frac{n^{2}}{64\varphi^{2}}\varphi^{2}
|\tilde{\nabla}\boldsymbol{\Theta}|^{2}_{\tilde{g}}
\geq\frac{n^{2}}{4.3264 n^{2}}\varphi^{2}
|\tilde{\nabla}\boldsymbol{\Theta}|^{4}_{\tilde{g}}
\geq\frac{1}{5}\varphi^{2}|\tilde{\nabla}g|^{4}_{\tilde{g}}.
\end{equation*}
Defining
\begin{equation}
\psi:=\varphi|\tilde{\nabla}\boldsymbol{\Theta}|^{2}_{\tilde{g}},\label{3.53}
\end{equation}
we obtain from the above the inequality and (\ref{3.52}) that
\begin{equation}
\partial_{t}\psi\leq g^{\alpha\beta}
\tilde{\nabla}_{\alpha}\tilde{\nabla}_{\beta}
-\frac{1}{5}\psi^{2}+(C_{4}+C_{6})\psi+C_{4}n
\leq g^{\alpha\beta}\tilde{\nabla}_{\alpha}\tilde{\nabla}_{\beta}
\psi-\frac{1}{10}\psi^{2}+C_{7}\label{3.54}
\end{equation}
on $B_{\tilde{g}}(x_{0},r+\delta)\times[0,T]$, for some positive constant
$C_{7}$ depending only $m, A, k_{0},n,
\alpha_{1}, \beta_{2}$, and $k_{1}$.

Using the cutoff function and going through the argument in \cite{S89}, we can prove that
\begin{equation*}
|\tilde{\nabla}\boldsymbol{\Theta}|^{2}_{\tilde{g}}\lesssim1
\end{equation*}
on $B_{\tilde{g}}(x_{0},r+\frac{\delta}{2})
\times[0,T]$, where $\lesssim$ depends on $m, r, \delta, T, \tilde{g},
k_{1}$. Note that that $(1-\frac{1}{A m^{10}})\tilde{g}
\leq g\leq(1+\frac{1}{A m^{10}}\tilde{g}$, where
\begin{equation*}
A\geq\max(73728, 16C).
\end{equation*}
From the definition (\ref{3.46}), we can estimate
\begin{eqnarray*}
C&\leq&3344+8\alpha^{2}_{1}+160+160\alpha^{2}_{1}
+32(\alpha^{2}_{1}+\beta^{2}_{1})\\
&&+ \ 64+64\beta^{2}_{1}+640
+640\beta^{2}_{1}+4+4\beta^{2}_{1}\\
&\leq&4212+200\alpha^{2}_{1}+740\beta^{2}_{1}.
\end{eqnarray*}
Then we may choose $A=80000(1+\alpha^{2}_{1}+\beta^{2}_{1})$.
\end{proof}

By the same method we can prove the higher order derivatives estimates for $g$.

\begin{lemma}\label{l3.7} Under the assumption in Lemma \ref{l3.6} where we
furthermore assume $|\phi|^{2}\leq k_{1}$, for
any nonnegative integer $n$, there exist positive constants $C_{n}=C(m,n,r,\delta,T,
\tilde{g}, k_{1})$ depending only on $m, n, r, \delta, T, \tilde{g}$, and $k_{1}$, such that
\begin{equation}
|\tilde{\nabla}^{n}g|^{2}_{\tilde{g}}\leq C_{n}, \ \ \
|\tilde{\nabla}^{n}\phi|^{2}_{\tilde{g}}\leq C_{n}\label{3.55}
\end{equation}
on $B_{\tilde{g}}(x_{0}, r+\frac{\delta}{n+1})\times[0,T]$.
\end{lemma}

\begin{proof} We prove this lemma by induction on $n$. If $n=0$, using (\ref{3.19}) we have
\begin{equation*}
|g|^{2}_{\tilde{g}}=\tilde{g}^{ik}\tilde{g}^{j\ell}
g_{ij}g_{k\ell}\leq 4m
\end{equation*}
on $B_{\tilde{g}}(x_{0},r+\delta)\times[0,T]$. Since $|\tilde{\phi}|^{2}_{\tilde{g}}
\leq k_{1}$, it follows from (\ref{3.14}) that $|\phi|^{2}_{\tilde{g}}
\lesssim 1$ on $B_{\tilde{g}}(x_{0},r+\frac{\delta}{2})\times[0,T]$,
where $\lesssim$ depends only on $m, r, \delta, T, \tilde{g}, k_{1}$ and is
independent on $x_{0}$. Now we consider
the annulus $B_{\tilde{g}}(x_{0},r+\delta)\setminus B_{\tilde{g}}
(x_{0},r+\frac{\delta}{2})$. For any $x$ in this annulus, we can find a
small ball $B_{\tilde{g}}(x, \delta')\subset B_{\tilde{g}}
(x_{0},r+\delta)\setminus B_{\tilde{g}}(x_{0},r+\frac{\delta}{2})$. Using
(\ref{3.14}) again, we have $|\phi|^{2}_{\tilde{g}}
\lesssim 1$ on $B_{\tilde{g}}(x,\frac{\delta'}{2})\times[0,T]$, where $\lesssim$
depends only on $m, r, \delta, T, \tilde{g}, k_{1}$ and is independent on $x$. In
particular, $|\phi|^{2}_{\tilde{g}}(x)\lesssim 1$ on $[0,T]$. Hence $|\phi|^{2}_{\tilde{g}}
\lesssim 1$ on $B_{\tilde{g}}(x_{0},r+\delta)\times[0,T]$, where
$\lesssim$ depends only on $m, r, \delta, T, \tilde{g}, k_{1}$ and is independent
on $x_{0}$.

If $n=1$, the estimate (\ref{3.55}) was proved in Lemma \ref{l3.6}. We now suppose
that $n\geq2$ and
\begin{equation}
|\tilde{\nabla}^{k}g|^{2}_{\tilde{g}}\leq C_{k}, \ \ \
|\tilde{\nabla}^{k}\phi|^{2}_{\tilde{g}}\leq C_{k}, \ \ \ k=0,1,\cdots,n-1,
\end{equation}
on $B_{\tilde{g}}(x_{0},r+\frac{\delta}{k+1})\times[0,T]$. According to (3.39) in \cite{L05}, we have
\begin{eqnarray}
\partial_{t}\tilde{\nabla}^{n}g&=&g^{\alpha\beta}\tilde{\nabla}_{\alpha}
\tilde{\nabla}_{\beta}\tilde{\nabla}^{n}g\nonumber\\
&&+ \ \sum_{\sum^{n+2}_{s=1}k_{s}\leq n+2, \ 0\leq k_{s}\leq n+1}
\tilde{\nabla}^{k_{1}}g\ast\cdots\ast\tilde{\nabla}^{k_{n+2}}g
\ast P_{k_{1}\cdots k_{n+2}}\label{3.57}\\
&&+ \ \sum_{\ell_{1}+\ell_{2}=n+2, \ 1\leq \ell_{s}\leq n+1}
\tilde{\nabla}^{\ell_{1}}\phi\ast\tilde{\nabla}^{\ell_{2}}\phi,\nonumber
\end{eqnarray}
where $P_{k_{1}\cdots k_{n+2}}$ is a polynomial in $g, g^{-1},
\widetilde{{\rm Rm}},\cdots,\tilde{\nabla}^{n}\widetilde{{\rm Rm}}$.
Similarly, from (\ref{2.22}) we can show that
\begin{eqnarray}
\partial_{t}\tilde{\nabla}^{n}\phi
&=&g^{\alpha\beta}\tilde{\nabla}_{\alpha}\tilde{\nabla}_{\beta}
\tilde{\nabla}^{n}\phi\nonumber\\
&&+ \ \sum_{\sum^{n}_{s=1}k_{s}+\ell_{1}+\ell_{2}
\leq n+2, \ 0\leq k_{s}\leq n, \ 0\leq \ell_{1},\ell_{2}
\leq n+1}\tilde{\nabla}^{k_{1}}g\ast\cdots\ast\tilde{\nabla}^{k_{n}}
g\label{3.58}\\
&&\ast \ \tilde{\nabla}^{\ell_{1}}\phi\ast\tilde{\nabla}^{\ell_{2}}
\phi\ast P_{k_{1}\cdots k_{n}\ell_{1}\ell_{2}}.\nonumber
\end{eqnarray}
Using the notion $\boldsymbol{\Theta}$ defined in (\ref{3.28}), we conclude from (\ref{3.57}) and (\ref{3.58}) that
\begin{eqnarray}
\partial_{t}\tilde{\nabla}^{n}
\boldsymbol{\Theta}&=&g^{\alpha\beta}\tilde{\nabla}_{\alpha}
\tilde{\nabla}_{\beta}\tilde{\nabla}^{n}\boldsymbol{\Theta}\nonumber\\
&&+ \ \sum_{\sum^{n+2}_{s=1}k_{s}\leq n+2, \
0\leq k_{s}\leq n+1}\tilde{\nabla}^{k_{1}}\boldsymbol{\Theta}\ast\cdots\ast
\tilde{\nabla}^{k_{n+2}}\boldsymbol{\Theta}\ast P_{k_{1}\cdots k_{s}}.\label{3.59}
\end{eqnarray}
The above equation is exact the equation (3.41) in \cite{L05} or
the equation (69) in page 254 of \cite{S89}, and, following
the same argument, we obtain $|\tilde{\nabla}^{n}\boldsymbol{\Theta}|^{2}_{\tilde{g}}
\lesssim1$ on $B_{\tilde{g}}(x_{0},r+\frac{\delta}{n+1})\times[0,T]$,
where $\lesssim$ depends only on $n, m, r, \delta, T, \tilde{g}, k_{1}$.
\end{proof}

Fix a point $x_{0}\in M$ and choose a family of domains $(D_{k})_{k\in{\bf N}}$ on $M$ such that for each $k$, $\partial D_{k}$ is a compact smooth $(m-1)$-dimensional submanifold of $M$ and
\begin{equation*}
\bar{D}_{k}=D_{k}\cup\partial D_{k} \ \text{is a compact subset of} \ M, \ \ \
B_{\tilde{g}}(x_{0},k)\subset D_{k}.
\end{equation*}
By the same argument used in \cite{S89}, together with Theorem \ref{t3.5} and
Lemma \ref{l3.7}, we have

\begin{theorem}\label{t3.8} Suppose that $(M,\tilde{g})$ is a smooth complete
Riemannian manifold of dimension $m$ with $|\widetilde{{\rm Rm}}|^{2}_{\tilde{g}}
\leq k_{0}$ and $\tilde{\phi}$ be a smooth function satisfying $|\phi|^{2}
+|\tilde{\nabla}\phi|^{2}_{\tilde{g}}\leq k_{1}$ on $M$. There exists a positive constant
$T=T(m,k_{0},k_{1},\alpha_{1},\beta_{1},\beta_{2})$ such that the flow
\begin{eqnarray*}
\partial_{t}g_{ij}&=&-2R_{ij}+2\alpha_{1}\nabla_{i}\phi
\nabla_{j}\phi+\nabla_{i}V_{j}+\nabla_{j}V_{i}, \\
\partial_{t}\phi&=&\Delta\phi+\beta_{1}|\nabla\phi|^{2}_{g}
+\beta_{2}\phi+\langle V,\nabla\phi\rangle_{g},\\
(g(0),\phi(0))&=&(\tilde{g},\tilde{\phi}),
\end{eqnarray*}
has a smooth solution $(g(t),\phi(t))$ on $M\times[0,T]$ that satisfies the estimate
\begin{equation*}
\left(1-\frac{1}{80000(1+\alpha^{2}_{1}
+\beta^{2}_{1})m^{10}}\right)
\tilde{g}\leq g(t)\leq\left(1+\frac{1}{80000(1+\alpha^{2}_{1}
+\beta^{2}_{1})m^{10}}\right)\tilde{g}
\end{equation*}
on $M\times[0,T]$. Moreover $|\phi(t)|^{2}\lesssim1$ where $\lesssim$ depends
on $m, k_{0}, \alpha_{1}, \beta_{1}, \beta_{2}$.
\end{theorem}

\begin{proof} By the regularity of the flow and applying Corollary to $D_{k}$,
we have $|\nabla\phi|^{2}_{g}\lesssim1$ on $D_{k}$, where $\lesssim$ depends only on $k_{1}, \alpha_{1},
\beta_{2}$; then $|\tilde{\nabla}\phi|^{2}_{\tilde{g}}\lesssim1$ on $D_{k}$,
where $\lesssim$ depends only on $m, k_{1}, \alpha_{1}, \beta_{1},\beta_{2}$. Letting $k\to\infty$, we see that $|\tilde{\nabla}
\phi|^{2}_{\tilde{g}}\lesssim1$ on $M$, where $\lesssim$ depends only on $m,
k_{1}, \alpha_{1},\beta_{1},\beta_{2}$. In particular, $|\phi(t)|^{2}\lesssim1$ on
$M$, where $\lesssim$ depends only on $m, k_{0},\alpha_{1},\beta_{1},\beta_{2}$.
\end{proof}

\subsection{First order derivative estimates}\label{subsection3.3}

Let $\tilde{\phi}$ be a smooth function on a smooth complete Riemannian
manifold $(M,\tilde{g})$ of dimension $m$. Assume
\begin{equation}
|\widetilde{{\rm Rm}}|^{2}_{\tilde{g}}\leq k_{0}, \ \ \
|\tilde{\phi}|^{2}+|\tilde{\nabla}\tilde{\phi}|^{2}_{\tilde{g}}\leq k_{1}, \ \ \
|\tilde{\nabla}^{2}\tilde{\phi}|^{2}_{\tilde{g}}\leq k_{2}\label{3.60}
\end{equation}
on $M$. Let $(g(t), \phi(t)), T$ be obtained in Theorem \ref{t3.8} and
\begin{equation}
\delta:=\frac{1}{80000(1+\alpha^{2}_{1}+\beta^{2}_{1})m^{10}}.\label{3.61}
\end{equation}
Then
\begin{equation}
(1-\delta)\tilde{g}\leq g(t)\leq(1+\delta)\tilde{g}\label{3.62}
\end{equation}
on $M\times[0,T]$. As in \cite{S89}, define
\begin{equation}
h_{ij}:=g_{ij}-\tilde{g}, \ \ \ H_{ij}:=\frac{1}{\delta}h_{ij}.\label{3.63}
\end{equation}
Then
\begin{equation*}
\partial_{t}h_{ij}=
g^{\alpha\beta}\tilde{\nabla}_{\alpha}\tilde{\nabla}_{\beta}
h_{ij}+A_{ij},
\end{equation*}
where
\begin{eqnarray}
A_{ij}&=&g^{\alpha\beta}g_{ip}\tilde{g}^{pq}
\tilde{R}_{j\alpha q\beta}+g^{\alpha\beta}g_{jp}
\tilde{g}^{pq}\tilde{R}_{i\alpha q\beta}
+2\alpha_{1}\tilde{\nabla}_{i}\phi\tilde{\nabla}_{j}\phi\nonumber\\
&&+ \ \frac{1}{2}g^{\alpha\beta}g^{pq}
\bigg(\tilde{\nabla}_{i}h_{p\alpha}+2\tilde{\nabla}_{\alpha}h_{jp}
\tilde{\nabla}_{q}h_{i\beta}-2\tilde{\nabla}_{\alpha}
h_{jp}\tilde{\nabla}_{\beta}h_{iq}\label{3.64}\\
&&+ \ -2\tilde{\nabla}_{j}h_{p\alpha}\tilde{\nabla}_{\beta}
h_{iq}-2\tilde{\nabla}_{i}h_{p\alpha}\tilde{\nabla}_{\beta}
h_{jq}\bigg).\nonumber
\end{eqnarray}
From $\delta<1/2$ and (\ref{3.60}) we have from (\ref{3.64}) that
\begin{equation}
-\bigg(8m\sqrt{k_{0}}+20|\tilde{\nabla}h|^{2}_{\tilde{g}}
\bigg)\tilde{g}\leq A_{ij}\leq\bigg(8m\sqrt{k_{0}}+20|\tilde{\nabla}h
|^{2}_{\tilde{g}}\bigg)\tilde{g}\label{3.65}
\end{equation}
on $M\times[0,T]$. Therefore
\begin{equation}
\partial_{t}H_{ij}=g^{\alpha\beta}\tilde{\nabla}_{\alpha}
\tilde{\nabla}_{\beta}H_{ij}+B_{ij}, \ \ \ H(0)=0, \label{3.66}
\end{equation}
where $B_{ij}:=A_{ij}/\delta$ satisfying
\begin{equation}
-\bigg(\frac{8m\sqrt{k_{0}}}{\delta}
+20\delta|\tilde{\nabla}H|^{2}_{\tilde{g}}
\bigg)\tilde{g}\leq B_{ij}\leq\bigg(\frac{8m\sqrt{k_{0}}}{\delta}
+20\delta|\tilde{\nabla}H|^{2}_{\tilde{g}}\bigg)\tilde{g}\label{3.67}
\end{equation}
on $M\times[0,T]$. As in \cite{L05}, define
\begin{equation}
\psi:=\phi-\tilde{\phi}, \ \ \ \Psi:=\delta\psi.\label{3.68}
\end{equation}
Then
\begin{equation*}
\partial_{t}\psi=g^{\alpha\beta}\tilde{\nabla}_{\alpha}
\tilde{\nabla}_{\beta}\psi+C,
\end{equation*}
where
\begin{equation}
C:=g^{\alpha\beta}\tilde{\nabla}_{\alpha}\tilde{\nabla}_{\beta}
\tilde{\phi}+\beta_{1}|\tilde{\nabla}\tilde{\phi}|^{2}_{g}
+\beta_{2}\tilde{\phi}
+\beta_{1}|\tilde{\nabla}\psi|^{2}_{g}
+\beta_{2}\psi+2\beta_{1}\langle\tilde{\nabla}\psi,\tilde{\nabla}
\tilde{\phi}\rangle_{g}.\label{3.69}
\end{equation}
From $\delta<1/2$, (\ref{3.60}) and the proof of Theorem \ref{t3.8}, we have from (\ref{3.69}) that
\begin{equation*}
|C|\leq 2m\sqrt{k_{0}}+2\beta_{1}k_{1}
+\beta_{2}\sqrt{k_{1}}+\beta_{2}|\psi|+2\beta_{1}|\tilde{\nabla}
\psi|^{2}_{\tilde{g}}\lesssim1
\end{equation*}
where $\lesssim$ depends on $m,k_{0}, k_{1}, \alpha_{1}, \beta_{1},\beta_{2}$. Consequently,
\begin{eqnarray*}
\partial_{t}H_{ij}&=&g^{\alpha\beta}\tilde{\nabla}_{\alpha}
\tilde{\nabla}_{\beta}H_{ij}+B_{ij}, \ \ \ H(0)=0,\\
\partial_{t}\Psi&=&g^{\alpha\beta}\tilde{\nabla}_{\alpha}
\tilde{\nabla}_{\beta}\Psi+D, \ \ \ \Psi(0)=0.
\end{eqnarray*}
Since
\begin{eqnarray*}
-\tilde{g}&\leq& H(t) \ \ \leq \ \ \tilde{g},\\
\frac{1}{1+\delta}\tilde{g}^{-1}&\leq&g^{-1} \ \
\leq \ \ \frac{1}{1-\delta}\tilde{g}^{-1},\\
|\tilde{\nabla}g^{-1}(t)|^{2}_{\tilde{g}}&\leq&\frac{\delta^{2}}{(1-\delta)^{4}}
|\tilde{\nabla}H(t)|^{2}_{\tilde{g}},\\
|\Psi(t)|^{2}&\leq&1
\end{eqnarray*}
By the argument in \cite{S89}, we arrive at
\begin{equation*}
\sup_{M\times[0,T]}
\bigg(|\tilde{\nabla}H(t)|^{2}_{\tilde{g}}
+|\tilde{\nabla}\Psi(t)|^{2}_{\tilde{g}}\bigg)\leq1
\end{equation*}
where $\lesssim$ depends only on $m, k_{0}, k_{1},k_{1}, \alpha_{1},\beta_{1},
\beta_{2}$.

\begin{theorem}\label{t3.9} There exists a positive constant $T=T(m,k_{0},k_{1},\alpha_{1},
\beta_{1},\beta_{2})$ depending only on $m, k_{0}, k_{1}, \alpha_{1},\beta_{1},
\beta_{2}$ such that
\begin{equation*}
\sup_{M\times[0,T]}|\tilde{\nabla}g(t)|^{2}_{\tilde{g}}
\leq C, \ \ \ \sup_{M\times[0,T]}|\tilde{\nabla}
\phi|^{2}_{\tilde{g}}\leq C
\end{equation*}
where $C$ is a positive constant depending only on $
m, k_{0}, k_{1}, k_{2}, \alpha_{1}, \beta_{1},\beta_{2}$.
\end{theorem}

\subsection{Second order derivative estimates}\label{subsection3.4}

In this subsection we derive the second order derivative estimates. Let
$(M,\tilde{g})$ be a smooth complete Riemannian manifold of dimension $m$ and
$\tilde{\phi}$ a smooth function on $M$, satisfying
\begin{equation}
|\widetilde{{\rm Rm}}|_{\tilde{g}}\leq k_{0}, \ \ \ |\tilde{\phi}|^{2}
+|\tilde{\nabla}\tilde{\phi}|^{2}_{\tilde{g}}\leq k_{1}, \ \ \
|\tilde{\nabla}^{2}\tilde{\phi}|^{2}_{\tilde{g}}\leq k_{2}.\label{3.70}
\end{equation}
From Theorem \ref{t3.8} and Theorem \ref{t3.9}, there exists a positive
constant $T$ depending on $m, k_{0}, k_{1}, \alpha_{1}, \beta_{1}, \beta_{2}$ such that the
flow
\begin{eqnarray*}
\partial_{t}g_{ij}&=&-2R_{ij}+2\alpha_{1}\nabla_{i}\phi\nabla_{j}\phi
+\nabla_{i}V_{j}+\nabla_{j}V_{i},\\
\partial_{t}\phi&=&\Delta\phi+\beta_{1}|\nabla\phi|^{2}_{g}
+\beta_{2}\phi+\langle V,\nabla\phi\rangle_{g},\\
(g(0),\phi(0))&=&(\tilde{g},\tilde{\phi}),
\end{eqnarray*}
has a smooth solution $(g(t),\phi(t))$ on $M\times[0,T]$ that satisfies
the estimates
\begin{equation}
\frac{1}{2}\tilde{g}\leq g(t)\leq2\tilde{g}, \ \ \
|\tilde{\phi}|^{2}\lesssim1, \ \ \ |\tilde{\nabla}g|^{2}_{\tilde{g}}
\lesssim1, \ \ \
|\tilde{\nabla}\phi|^{2}_{\tilde{g}}\lesssim 1,\label{3.71}
\end{equation}
where $\lesssim$ (used throughout this subsection) depend only on $m, k_{0}, k_{1}, k_{2}, \alpha_{1},
\beta_{1}, \beta_{2}$.

According to Lemma \ref{l2.5} and the proof of equation (34) in page 266 of
\cite{S89}, we have
\begin{equation}
\partial_{t}{\rm Rm}=\Delta{\rm Rm}+g^{\ast-2}\ast{\rm Rm}^{\ast 2}
+g^{-1}\ast V\ast\nabla{\rm Rm}+g^{-1}\ast{\rm Rm}\ast\nabla V
+\nabla^{2}\phi\ast\nabla^{2}\phi.\label{3.72}
\end{equation}
Recall the formula (28) in page 266 of \cite{S89},
\begin{equation*}
\partial_{t}V=\Delta V+g^{\ast-2}\ast{\rm Rm}\ast V
+g^{-1}\ast V+\ast\partial_{t}g+g^{-1}\ast\tilde{\nabla}g
\ast\partial_{t}g^{-1}\ast g.
\end{equation*}
Since
\begin{equation*}
\partial_{t}g=g^{-1}\ast{\rm Rm}+\nabla V+\nabla\phi\ast\nabla\phi, \ \ \
\partial_{t}g^{-1}=g^{-1}\ast g^{-1}\ast\partial_{t}g, \ \ \
V=g^{-1}\ast\tilde{\nabla}g,
\end{equation*}
it follows that
\begin{eqnarray}
\partial_{t}V&=&\Delta V+g^{\ast-3}\ast\tilde{\nabla}g\ast{\rm Rm}
+g^{\ast-2}\ast\tilde{\nabla}g\ast\nabla V
+g^{\ast-4}\ast g\ast\tilde{\nabla}g\ast{\rm Rm}\nonumber\\
&&+ \ g^{\ast-3}\ast g\ast\tilde{\nabla}g\ast\nabla V
+g^{\ast-2}\ast\tilde{\nabla}g\ast\nabla\phi\ast\nabla\phi
+g^{\ast-3}\ast g\ast\tilde{\nabla}g\ast\nabla\phi\ast\nabla\phi.\label{3.73}
\end{eqnarray}
Since
\begin{equation*}
\partial_{t}\nabla V=\nabla(\partial_{t}V)+V\ast\partial_{t}\Gamma, \ \ \
\partial_{t}\Gamma=g^{-1}\ast\nabla(\partial_{t}g),
\end{equation*}
it follows that (as the proof of the equation (99) in page 278 of \cite{S89})
\begin{eqnarray}
\partial_{t}\nabla V&=&\Delta\nabla V+g^{\ast-3}\ast\nabla\tilde{\nabla}g\ast{\rm Rm}
+g^{\ast-3}\ast\tilde{\nabla}g\ast\nabla{\rm Rm}
+g^{\ast-2}\ast\nabla\tilde{\nabla}g\ast\nabla V\nonumber\\
&&+ \ g^{\ast-2}\ast\tilde{\nabla}g\ast\nabla^{2}V
+g^{\ast-4}\ast g\ast\nabla\tilde{\nabla}g\ast{\rm Rm}
+g^{\ast-4}\ast g\ast\tilde{\nabla}g\ast\nabla{\rm Rm}\nonumber\\
&&+ \ g^{\ast-3}\ast g\ast\nabla\tilde{\nabla}g\ast\nabla V
+g^{\ast-3}\ast g\ast\tilde{\nabla}g\ast\nabla^{2}V\label{3.74}\\
&&+ \ g^{\ast-2}\ast\nabla\tilde{\nabla}g\ast\nabla\phi\ast\nabla\phi
+g^{\ast-2}\ast\tilde{\nabla}g\ast\nabla\phi\ast\nabla^{2}\phi\nonumber\\
&&+ \ g^{\ast-3}\ast g\ast\nabla\tilde{\nabla}g
\ast\nabla\phi\ast\nabla\phi+g^{\ast-3}
\ast g\ast\tilde{\nabla}g\ast\nabla\phi\ast\nabla^{2}\phi.\nonumber
\end{eqnarray}

By the diffeomorphisms $(\Psi_{t})_{t\in[0,T]}$ defined by (\ref{2.17}), we have
\begin{eqnarray*}
\partial_{t}\hat{g}_{ij}(x,t)&=&-2\hat{R}_{ij}(x,t)+2\alpha_{1}
\hat{\nabla}_{i}\hat{\phi}(x,t)\hat{\nabla}_{j}\hat{\phi}(x,t),\\
\partial_{t}\hat{\phi}(x,t)&=&\hat{\Delta}\hat{\phi}(x,t)
+\beta_{1}|\hat{\nabla}\hat{\phi}|^{2}_{\hat{g}}(x,t)+\beta_{2}
\hat{\phi}(x,t),
\end{eqnarray*}
where $\hat{g}(x,t)$ and $\hat{\phi}(x,t)$ are defined by (\ref{2.16}). Then
\begin{equation*}
\nabla_{i}\nabla_{j}\phi(x,t)=y^{\alpha}{}_{,i}y^{\beta}{}_{,j}
\hat{\nabla}_{\alpha}\hat{\nabla}_{\beta}\hat{\phi}(y,t), \ \ \ y^{\alpha}{}_{,i}:=
\frac{\partial}{\partial x^{i}}y^{\alpha},
\end{equation*}
and hence
\begin{eqnarray*}
\partial_{t}\nabla_{i}\nabla_{j}\phi(x,t)&=&y^{\alpha}{}_{,i}y^{\beta}{}_{,j}
\partial_{t}\hat{\nabla}_{\alpha}\hat{\nabla}_{\beta}\hat{\phi}(y,t)
+y^{\alpha}{}_{,i}y^{\beta}{}_{,j}\partial_{t}y^{\gamma}\frac{\partial}{\partial y^{\gamma}}
\hat{\nabla}_{\alpha}\hat{\nabla}_{\beta}\hat{\phi}(y,t)\\
&&+ \ \partial_{t}\left(y^{\alpha}{}_{,i}y^{\beta}{}_{,j}\right)
\hat{\nabla}_{\alpha}\hat{\nabla}_{\beta}\hat{\phi}(y,t).
\end{eqnarray*}
By Lemma \ref{l2.7}, we have
\begin{eqnarray*}
y^{\alpha}{}_{,i}y^{\beta}{}_{,j}
\partial_{t}\hat{\nabla}_{\alpha}\hat{\nabla}_{\beta}\hat{\phi}(y,t)
&=&\Delta\nabla_{i}\nabla_{j}\phi+g^{\ast-2}\ast{\rm Rm}\ast\nabla^{2}\phi
+\beta_{2}
\nabla^{2}\phi\\
&&+ \ g^{-1}\ast\nabla\phi\ast\nabla\phi\ast\nabla^{2}\phi
+g^{-1}\ast\nabla\phi\ast\nabla^{3}\phi\\
&&+ \ g^{-1}\ast\nabla^{2}\phi
\ast\nabla^{2}\phi+g^{\ast-2}\ast{\rm Rm}\ast\nabla\phi\ast\nabla\phi.
\end{eqnarray*}
Using (17) and (18) in page 263 of \cite{S89}, we can conclude that
\begin{eqnarray*}
I&=&y^{\alpha}{}_{,i}y^{\beta}{}_{,j}\partial_{t}y^{\gamma}
\frac{\partial}{\partial y^{\gamma}}\hat{\nabla}_{\alpha}
\hat{\nabla}_{\beta}\hat{\phi}(y,t) \ \ = \ \ y^{\alpha}{}_{,i}y^{\beta}{}_{,j}\partial_{t}y^{\gamma}\frac{\partial
x^{p}}{\partial y^{\gamma}}\frac{\partial}{\partial x^{p}}
\hat{\nabla}_{\alpha}\hat{\nabla}_{\beta}\hat{\phi}(y,t)\\
&=&\partial_{t}y^{\gamma}\frac{\partial x^{p}}{\partial y^{\gamma}}
\frac{\partial}{\partial x^{p}}\left(y^{\alpha}{}_{,i}y^{\beta}{}_{,j}
\hat{\nabla}_{\alpha}\hat{\nabla}_{\beta}\hat{\phi}\right)-\partial_{t}y^{\gamma}\frac{\partial x^{p}}{\partial y^{\gamma}}
(y^{\alpha}{}_{,i}y^{\beta}{}_{,j})_{,p}\hat{\nabla}_{\alpha}
\hat{\nabla}_{\beta}\hat{\phi}\\
&=&g^{pq}V_{q}\frac{\partial}{\partial x^{p}}\nabla_{i}\nabla_{j}\phi
-g^{pq}V_{q}\left(y^{\alpha}{}_{,ip}
\frac{\partial x^{s}}{\partial y^{\alpha}}\nabla_{s}\nabla_{j}\phi
+y^{\beta}{}_{,jp}\frac{\partial x^{s}}{\partial y^{\beta}}\nabla_{i}\nabla_{s}
\phi\right),\\
J&=&\partial_{t}(y^{\alpha}{}_{,i}y^{\beta}{}_{,j})\hat{\nabla}_{\alpha}
\hat{\nabla}_{\beta}\hat{\phi}(y,t) \ \ = \ \ (\partial_{t}y^{\alpha})_{,i}
\frac{\partial x^{s}}{\partial y^{\alpha}}\nabla_{s}\nabla_{j}\phi
+(\partial_{t}y^{\beta})_{,j}\frac{\partial x^{s}}{\partial y^{\beta}}
\nabla_{i}\nabla_{s}\phi\\
&=&g^{pq}V_{q}\left(y^{\alpha}{}_{,pi}\frac{\partial x^{s}}{\partial y^{\beta}}
\nabla_{s}\nabla_{j}\phi+y^{\beta}{}_{,pj}\frac{\partial x^{s}}{\partial y^{\alpha}}
\nabla_{i}\nabla_{s}\phi\right)\\
&&+ \ (g^{pq}V_{q})_{,i}\nabla_{p}\nabla_{j}\phi
+(g^{pq}V_{q})_{,j}\nabla_{i}\nabla_{p}\phi.
\end{eqnarray*}
Consequently,
\begin{eqnarray*}
I+J&=&g^{pq}V_{q}\nabla_{p}\nabla_{i}\nabla_{j}\phi
+g^{pq}\left(\nabla_{i}V_{q}\nabla_{p}\nabla_{j}\phi
+\nabla_{j}V_{q}\nabla_{i}\nabla_{q}\phi\right)\\
&=&g^{-1}\ast V\ast\nabla^{3}\phi+g^{-1}\ast\nabla V\ast\nabla^{2}\phi.
\end{eqnarray*}
Combining those identities yields
\begin{eqnarray}
\partial_{t}\nabla^{2}\phi&=&\Delta\nabla^{2}\phi
+g^{\ast-2}\ast{\rm Rm}\ast\nabla^{2}\phi+\beta_{2}\nabla^{2}\phi
+g^{-1}\ast\nabla\phi\ast\nabla\phi\ast\nabla^{2}\phi\nonumber\\
&&+ \ g^{-1}\ast\nabla
\phi\ast\nabla^{3}\phi+g^{-1}\ast\nabla^{2}\phi\ast\nabla^{2}\phi+g^{\ast-2}
\ast{\rm Rm}\ast\nabla\phi\ast\nabla\phi\label{3.75}\\
&&+ \ g^{-1}\ast V\ast\nabla^{3}\phi+g^{-1}\ast\nabla V
\ast\nabla^{2}\phi.\nonumber
\end{eqnarray}

The volume form
\begin{equation*}
dV:=dV_{g(t)}=\sqrt{\det(g_{ij})}dx^{1}\wedge\cdots\wedge dx^{m}
\end{equation*}
evolves by
\begin{equation}
\partial_{t}dV=\frac{1}{2}g^{ij}\partial_{t}g_{ij}\!\ dV
=\left(-R+\alpha_{1}|\nabla\phi|^{2}_{g}+{\rm div}_{g}V\right)dV.\label{3.76}
\end{equation}
In particular, $dV=d\tilde{V}$. For any point $x_{0}\in M$ and any $r>0$ we denote by
$B_{\tilde{g}}(x_{0},r)$ the metric ball with respect to $\tilde{g}$. Recall the
definition
\begin{equation*}
\boldsymbol{\Theta}=(g,\phi).
\end{equation*}

\begin{lemma}\label{l3.10} We have
\begin{equation*}
\int^{T}_{0}\left(\int_{B_{\tilde{g}}(x_{0},r)}
|\tilde{\nabla}^{2}\boldsymbol{\Theta}|^{2}_{\tilde{g}}d\tilde{V}\right)dt\lesssim1
\end{equation*}
where $\lesssim$ depends on $m, r, k_{0}, k_{1}, \alpha_{1}, \beta_{1}, \beta_{2}$.
\end{lemma}

\begin{proof} As in the proof of Lemma 6.2 in \cite{S89}, we chose a cutoff
function $\xi(x)\in C^{\infty}_{0}(M)$ such that $|\tilde{\nabla}\xi|_{\tilde{g}}
\leq 8$ in $M$ and
\begin{eqnarray}
\xi&=&1 \ \ \ \text{in} \ B_{\tilde{g}}(x_{0},r),\nonumber\\
\xi&=&0 \ \ \ \text{in} \ M\setminus B_{\tilde{g}}\left(x_{0},
r+\frac{1}{2}\right),\label{3.77}\\
0&\leq&\xi \ \ \leq \ \ 1 \ \ \ \text{in} \ M.\nonumber
\end{eqnarray}
Since $m=g^{ij}g_{ij}$, it follows that the constant $1$ can
be replaced by $g^{-1}\ast g$. From (\ref{3.16}) we have
\begin{eqnarray*}
\partial_{t}\tilde{\nabla}g&=&g^{\alpha\beta}\tilde{\nabla}_{\alpha}
\tilde{\nabla}_{\beta}\tilde{\nabla}g+g^{-1}\ast g\ast\tilde{\nabla}
\widetilde{{\rm Rm}}+g^{\ast-2}\ast g\ast\tilde{\nabla}g\ast
\widetilde{{\rm Rm}}\\
&&+ \ g^{\ast-2}\ast\tilde{\nabla}g\ast\tilde{\nabla}^{2}g
+g^{\ast-3}\ast(\tilde{\nabla}g)^{\ast3}+\tilde{\nabla}\phi
\ast\tilde{\nabla}^{2}\phi.
\end{eqnarray*}
Therefore integrating $|\tilde{\nabla}g|^{2}_{\tilde{g}}\xi^{2}$ over the
ball $B_{\tilde{g}}(x_{0},r+1)$ implies
\begin{eqnarray*}
I&:=&\frac{d}{dt}\int_{B_{\tilde{g}}(x_{0},r+1)}
|\tilde{\nabla}g|^{2}_{\tilde{g}}\xi^{2}\!\ d\tilde{V} \ \
= \ \ 2\int_{B_{\tilde{g}}(x_{0},r+1)}
\langle\tilde{\nabla} g,\partial_{t}\tilde{\nabla}g\rangle_{\tilde{g}}
\xi^{2}\!\ d\tilde{V}\\
&=&2\int_{B_{\tilde{g}}(x_{0},r+1)}
\langle\tilde{\nabla}g,g^{\alpha\beta}\tilde{\nabla}_{\alpha}
\tilde{\nabla}_{\beta}\tilde{\nabla}g\rangle_{\tilde{g}}\xi^{2}\!\ d\tilde{V}\\
&&+ \ \int_{B_{\tilde{g}}(x_{0},r+1)}
g^{-1}\ast g\ast\tilde{\nabla}g\ast\tilde{\nabla}\widetilde{{\rm Rm}}\!\ \xi^{2}\!\
d\tilde{V}+\int_{B_{\tilde{g}}(x_{0},r+1)}
\bigg[g^{\ast-2}\ast g\ast\tilde{\nabla}g
\ast\widetilde{{\rm Rm}}\\
&&+ \ g^{\ast-2}\ast\tilde{\nabla}g
\ast\tilde{\nabla}^{2}g+g^{\ast-3}\ast(\tilde{\nabla}g)^{\ast3}
+\tilde{\nabla}\phi\ast\tilde{\nabla}^{2}\phi\bigg]\ast
\tilde{\nabla}g\!\ \xi^{2}\!\ d\tilde{V}\\
&:=&I_{1}+I_{2}+I_{3}.
\end{eqnarray*}
Now the following computations are similar to that given in \cite{L05}. For convenience, we
give a self-contained proof. Using (\ref{3.70}) and (\ref{3.71}) we get
\begin{eqnarray*}
&&\tilde{\nabla}g
\ast\bigg[g^{\ast-2}\ast g\ast\tilde{\nabla}g
\ast\widetilde{{\rm Rm}}+g^{\ast-2}
\ast\tilde{\nabla}g\ast\tilde{\nabla}^{2}g+g^{\ast-3}
\ast(\tilde{\nabla}g)^{\ast3}+\tilde{\nabla}\phi\ast\tilde{\nabla}^{2}
\phi\bigg]\\
&\lesssim&|\tilde{\nabla}g|^{2}_{\tilde{g}}
+|\tilde{\nabla}g|^{2}_{\tilde{g}}|\tilde{\nabla}^{2}g|_{\tilde{g}}
+|\tilde{\nabla}g|^{4}_{\tilde{g}}+|\tilde{\nabla}g|_{\tilde{g}}
|\tilde{\nabla}\phi|_{\tilde{g}}|\tilde{\nabla}^{2}\phi|_{\tilde{g}}\\
&\lesssim&1+|\tilde{\nabla}^{2}g|_{\tilde{g}}
+|\tilde{\nabla}^{2}\phi|_{\tilde{g}},
\end{eqnarray*}
and then
\begin{equation*}
I_{3}\lesssim\int_{B_{\tilde{g}}(x_{0},r+1)}
\left(1+|\tilde{\nabla}^{2}g|_{\tilde{g}}
+|\tilde{\nabla}^{2}\phi|_{\tilde{g}}\right)\xi^{2}\!\ d\tilde{V}.
\end{equation*}
By the Bishop-Gromov volume comparison (see \cite{CLN06}), we have
\begin{equation}
\int_{B_{\tilde{g}}(x_{0},r+1)}d\tilde{V}\lesssim1.\label{3.78}
\end{equation}
By the estimate (\ref{3.78}), the $I_{3}$-term can be rewritten as
\begin{equation}
I_{3}\lesssim1+\int_{B_{\tilde{g}}(x_{0},r+1)}
\left(|\tilde{\nabla}^{2}g|_{\tilde{g}}
+|\tilde{\nabla}^{2}\phi|_{\tilde{g}}\right)\xi^{2}\!\ d\tilde{V}.\label{3.79}
\end{equation}
The $I_{1}$-term and $I_{2}$-term were computed in \cite{S89} (see (54) and (58) in
page 270)
\begin{eqnarray}
I_{1}&\leq&-\frac{1}{2}\int_{B_{\tilde{g}}(x_{0},r+1)}
|\tilde{\nabla}^{2}g|^{2}_{\tilde{g}}\xi^{2}\!\ d\tilde{V}+C_{1},\label{3.80}\\
I_{2}&\lesssim&1+\int_{B_{\tilde{g}}(x_{0},r+1)}
|\tilde{\nabla}^{2}g|_{\tilde{g}}\xi\!\ d\tilde{V}.\label{3.81}
\end{eqnarray}
From (\ref{3.79}), (\ref{3.80}), and (\ref{3.81}), we arrive at
\begin{equation}
I\leq-\frac{1}{4}\int_{B_{\tilde{g}}(x_{0},r+1)}
|\tilde{\nabla}^{2}g|^{2}_{\tilde{g}}\xi^{2}\!\ d\tilde{V}
+C_{2}\int_{B_{\tilde{g}}(x_{0},r+1)}
|\tilde{\nabla}^{2}\phi|_{\tilde{g}}\xi\!\ d\tilde{V}+C_{2}.\label{3.82}
\end{equation}

From (\ref{3.24}) we have
\begin{eqnarray*}
\partial_{t}\tilde{\nabla}\phi&=&g^{\alpha\beta}
\tilde{\nabla}_{\alpha}\tilde{\nabla}_{\beta}\tilde{\nabla}\phi
+g^{\ast-2}\ast\widetilde{{\rm Rm}}\ast\tilde{\nabla}
\phi+g^{\ast-2}\ast\tilde{\nabla}g\ast\tilde{\nabla}^{2}\phi\\
&&+ \ g^{\ast-2}\ast\tilde{\nabla}g\ast(\tilde{\nabla}\phi)^{\ast2}
+g^{-1}\ast\tilde{\nabla}\phi\ast\tilde{\nabla}^{2}\phi
+\beta_{2}\tilde{\nabla}\phi.
\end{eqnarray*}
Therefore integrating $|\tilde{\nabla}\phi|^{2}_{\tilde{g}}\xi^{2}$ over the
ball $B_{\tilde{g}}(x_{0},r+1)$ implies
\begin{eqnarray*}
J&:=&\frac{d}{dt}\int_{B_{\tilde{g}}(x_{0},r+1)}
|\tilde{\nabla}\phi|^{2}_{\tilde{g}}\xi^{2}\!\ d\tilde{V} \ \
= \ \ 2\int_{B_{\tilde{g}}(x_{0},r+1)}
\langle\tilde{\nabla}\phi,\partial_{t}\tilde{\nabla}\phi\rangle_{\tilde{g}}
\xi^{2}\!\ d\tilde{V}\\
&=&2\int_{B_{\tilde{g}}(x_{0},r+1)}
\langle\tilde{\nabla}\phi,g^{\alpha\beta}\tilde{\nabla}_{\alpha}
\tilde{\nabla}_{\beta}\tilde{\nabla}\phi\rangle_{\tilde{g}}\xi^{2}\!\
d\tilde{V}+\int_{B_{\tilde{g}}(x_{0},r+1)}
\xi^{2}\bigg[g^{\ast-2}\ast
\widetilde{{\rm Rm}}\ast\tilde{\nabla}\phi\\
&&+ \ g^{\ast-2}\ast\tilde{\nabla}g
\ast\tilde{\nabla}^{2}\phi+g^{\ast-2}\ast\tilde{\nabla}g
\ast(\tilde{\nabla}\phi)^{\ast2}+g^{-1}\ast\tilde{\nabla}
\phi\ast\tilde{\nabla}^{2}\phi+\tilde{\nabla}\phi\bigg]\tilde{\nabla}
\phi\!\ d\tilde{V}\\
&:=&J_{1}+J_{2}.
\end{eqnarray*}
Using (\ref{3.70}) and (\ref{3.71}), we get
\begin{eqnarray*}
&&\bigg[g^{\ast-2}\ast
\widetilde{{\rm Rm}}\ast\tilde{\nabla}\phi+g^{\ast-2}\ast\tilde{\nabla}g
\ast\tilde{\nabla}^{2}\phi+g^{\ast-2}\ast\tilde{\nabla}g
\ast(\tilde{\nabla}\phi)^{\ast2}\\
&&+ \ g^{-1}\ast\tilde{\nabla}
\phi\ast\tilde{\nabla}^{2}\phi+\tilde{\nabla}\phi\bigg]\ast\tilde{\nabla}
\phi\\
&\lesssim&|\tilde{\nabla}\phi|^{2}_{\tilde{g}}
+|\tilde{\nabla}g|_{\tilde{g}}|\tilde{\nabla}\phi|_{\tilde{g}}
|\tilde{\nabla}^{2}\phi|_{\tilde{g}}+|\tilde{\nabla}g|_{\tilde{g}}
|\tilde{\nabla}\phi|^{3}_{\tilde{g}}+|\tilde{\nabla}\phi|^{2}_{\tilde{g}}
|\tilde{\nabla}^{2}\phi|_{\tilde{g}}+|\tilde{\nabla}\phi|^{2}_{\tilde{g}}\\
&\lesssim&1+|\tilde{\nabla}^{2}\phi|_{\tilde{g}}.
\end{eqnarray*}
By the estimate (\ref{3.78}), the $J_{2}$-term can be bounded by
\begin{equation}
J_{2}\lesssim 1+\int_{B_{\tilde{g}}(x_{0},r+1)}
|\tilde{\nabla}^{2}\phi|_{\tilde{g}}\xi\!\ d\tilde{V}.\label{3.83}
\end{equation}
From the integration by parts, we obtain
\begin{eqnarray*}
J_{1}&=&-2\int_{B_{\tilde{g}}(x_{0},r+1)}
g^{\alpha\beta}\langle\tilde{\nabla}_{\beta}\tilde{\nabla}\phi,
\tilde{\nabla}_{\alpha}\tilde{\nabla}\phi\rangle_{\tilde{g}}\xi^{2}\!\ d\tilde{V}\\
&&+ \ \int_{B_{\tilde{g}}(x_{0},r+1)}
\tilde{\nabla}^{2}\phi\ast g^{\ast-2}\ast\tilde{\nabla}g
\ast\tilde{\nabla}\phi\!\ \xi^{2}\!\ d\tilde{V}\\
&&+ \ \int_{B_{\tilde{g}}(x_{0},r+1)}
g^{-1}\ast\tilde{\nabla}^{2}\phi\ast\tilde{\nabla}\phi
\ast\xi\ast\tilde{\nabla}\xi\!\ d\tilde{V}.
\end{eqnarray*}
Since
\begin{equation*}
g^{\alpha\beta}\langle\tilde{\nabla}_{\beta}\tilde{\nabla}
\phi,\tilde{\nabla}_{\alpha}\tilde{\nabla}\phi\rangle_{\tilde{g}}
\geq\frac{1}{2}|\tilde{\nabla}^{2}\phi|^{2}_{\tilde{g}}
\end{equation*}
by (\ref{3.71}), and
\begin{eqnarray*}
\int_{B_{\tilde{g}}(x_{0},r+1)}
\tilde{\nabla}^{2}\phi\ast g^{\ast-2}\ast\tilde{\nabla}g
\ast\tilde{\nabla}\phi\!\ \xi^{2}\!\ d\tilde{V}&\lesssim&
\int_{B_{\tilde{g}}(x_{0},r+1)}|\tilde{\nabla}^{2}\phi|_{\tilde{g}}
\xi\!\ d\tilde{V},\\
\int_{B_{\tilde{g}}(x_{0},r+1)}
g^{-1}\ast\tilde{\nabla}^{2}\phi\ast\tilde{\nabla}\phi
\ast\xi\ast\tilde{\nabla}\xi\!\ d\tilde{V}&\lesssim&
\int_{B_{\tilde{g}}(x_{0},r+1)}|\tilde{\nabla}^{2}\phi|_{\tilde{g}}
\xi\!\ d\tilde{V},
\end{eqnarray*}
it follows that
\begin{equation}
J_{1}\leq-\int_{B_{\tilde{g}}(x_{0},r+1)}
|\tilde{\nabla}^{2}\phi|^{2}_{\tilde{g}}\xi^{2}\!\ d\tilde{V}
+C_{3}\int_{B_{\tilde{g}}(x_{0},r+1)}|\tilde{\nabla}^{2}\phi|_{\tilde{g}}
\xi\!\ d\tilde{V}.\label{3.84}
\end{equation}
From (\ref{3.83}) and (\ref{3.84}),
\begin{equation}
J\leq-\frac{1}{2}\int_{B_{\tilde{g}}(x_{0},r+1)}|\tilde{\nabla}^{2}\phi|^{2}_{\tilde{g}}
\xi^{2}\!\ d\tilde{V}+C_{4}.\label{3.85}
\end{equation}
Together with (\ref{3.82}), we arrive at
\begin{equation*}
I+J\leq-\frac{1}{4}\int_{B_{\tilde{g}}(x_{0},r+1)}
\left(|\tilde{\nabla}^{2}g|^{2}_{\tilde{g}}
+|\tilde{\nabla}^{2}\phi|^{2}_{\tilde{g}}\right)\xi^{2}\!\ d\tilde{V}+C_{5}.
\end{equation*}
Thus
\begin{equation}
\frac{d}{dt}\int_{B_{\tilde{g}}(x_{0},r+1)}
|\tilde{\nabla}\boldsymbol{\Theta}|^{2}_{\tilde{g}}
\xi^{2}\!\ d\tilde{V}\leq-\frac{1}{4}\int_{B_{\tilde{g}}(x_{0},
r+1)}|\tilde{\nabla}^{2}\boldsymbol{\Theta}|^{2}_{\tilde{g}}
\xi^{2}\!\ d\tilde{V}+C_{5}.\label{3.86}
\end{equation}
Integrating (\ref{3.86}) over $[0,T]$ implies
\begin{equation*}
\int^{T}_{0}\left(\int_{B_{\tilde{g}}(x_{0},r+1)}
|\tilde{\nabla}^{2}\boldsymbol{\Theta}|^{2}_{\tilde{g}}
\xi^{2}\!\ d\tilde{V}\right)dt\lesssim1
\end{equation*}
Since $\xi=1$ on $B_{\tilde{g}}(x_{0},r)$, the above estimate yields the
desired inequality.
\end{proof}

Using (\ref{3.71}) we have
\begin{equation*}
|\tilde{\nabla}^{2}g|^{2}_{g}\leq 16|\tilde{\nabla}^{2}g|^{2}_{\tilde{g}}, \ \ \
|\tilde{\nabla}^{2}\phi|^{2}_{g}\leq 4|\tilde{\nabla}^{2}\phi|^{2}_{\tilde{g}}, \ \ \
dV\leq 2^{m/2}d\tilde{V},
\end{equation*}
on $M\times[0,T]$ and then
\begin{equation*}
\int^{T}_{0}\left(\int_{B_{\tilde{g}}(x_{0},r)}
|\tilde{\nabla}^{2}\boldsymbol{\Theta}|^{2}_{g}
dV\right)dt\lesssim\int^{T}_{0}\left(\int_{B_{\tilde{g}}(x_{0},r)}
|\tilde{\nabla}^{2}\boldsymbol{\Theta}|^{2}_{\tilde{g}}d\tilde{V}
\right)dt.
\end{equation*}
By (66) in page 272 of \cite{S89}, we have
\begin{equation*}
|\nabla\tilde{\nabla}g|^{2}_{g}\leq2|\tilde{\nabla}^{2}g|^{2}_{g}
+C_{6};
\end{equation*}
on the other hand, by $\nabla\tilde{\nabla}\phi
=\tilde{\nabla}^{2}\phi+g^{-1}\ast\tilde{\nabla}g\ast\tilde{\nabla}
\phi$, we get
\begin{equation*}
|\nabla\tilde{\nabla}\phi|^{2}_{g}\leq 2|\tilde{\nabla}^{2}
\phi|^{2}_{g}+C_{7}.
\end{equation*}
Thus
\begin{equation*}
\int^{T}_{0}\left(\int_{B_{\tilde{g}}(x_{0},r)}
|\nabla\tilde{\nabla}\boldsymbol{\Theta}|^{2}_{g}
dV\right)dt\lesssim1+\int^{T}_{0}\left(\int_{B_{\tilde{g}}(x_{0},r)}
|\tilde{\nabla}^{2}\boldsymbol{\Theta}|^{2}_{g}dV\right)dt.
\end{equation*}

Therefore,

\begin{lemma}\label{l3.11} We have
\begin{equation*}
\int^{T}_{0}\left(\int_{B_{\tilde{g}}(x_{0},r)}
\left(|\tilde{\nabla}^{2}\boldsymbol{\Theta}|^{2}_{g}+|
\nabla\tilde{\nabla}\boldsymbol{\Theta}|^{2}_{g}\right)dV
\right)dt\lesssim1.
\end{equation*}
where $\lesssim$ depends on $m, r, k_{0}, k_{1}, \alpha_{1}, \beta_{1}, \beta_{2}$.
\end{lemma}

We now prove the integral estimates for ${\rm Rm}$, $\nabla^{2}\phi$,
and $\nabla V$. The similar results were proved by Shi \cite{S89} for the
Ricci flow and List \cite{L05} for the Ricci flow coupled with the heat
flow.

\begin{lemma}\label{l3.12} We have
\begin{equation*}
\int_{B_{\tilde{g}}(x_{0},r)}
\left(|{\rm Rm}|^{2}_{g}+|\nabla^{2}\phi|^{2}_{g}
+|\nabla V|^{2}_{g}\right)dV\lesssim1
\end{equation*}
where $\lesssim$ depends on $m,r, k_{0}, k_{1}, k_{2}, \alpha_{1}, \beta_{1},
\beta_{2}$.
\end{lemma}

\begin{proof} Keep to use the same cutoff function $\xi(x)$ introduced in the
proof of Lemma \ref{l3.10}. From $|{\rm Rm}|^{2}_{g}
=g^{i\alpha}g^{j\beta}g^{k\gamma}g^{\ell\delta}
R_{ijk\ell}R_{\alpha\beta\gamma\delta}$, we get
\begin{equation*}
\partial_{t}|{\rm Rm}|^{2}_{g}=2\langle{\rm Rm},\partial_{t}{\rm Rm}\rangle_{g}
+{\rm Rm}^{\ast2}\ast g^{\ast-3}\ast\partial_{t}g^{-1}
\end{equation*}
and
\begin{eqnarray}
&&\int_{B_{\tilde{g}}(x_{0},r+1)}
|{\rm Rm}|^{2}_{g}\xi^{2}\!\ dV\nonumber\\
&=&\int_{B_{\tilde{g}}(x_{0},r+1)}
|\widetilde{{\rm Rm}}|^{2}_{\tilde{g}}\xi^{2}\!\ d\tilde{V}+\int^{t}_{0}\bigg(\frac{d}{dt}\int_{B_{\tilde{g}}(x_{0},r+1)}
|{\rm Rm}|^{2}_{g}\xi^{2}\!\ dV\bigg)dt\nonumber\\
&=&\int_{B_{\tilde{g}}(x_{0},r+1)}
|\widetilde{{\rm Rm}}|^{2}_{\tilde{g}}\xi^{2}\!\ d\tilde{V}
+\int^{t}_{0}\bigg(\int_{B_{\tilde{g}}(x_{0},r+1)}
|{\rm Rm}|^{2}_{g}\xi^{2}\partial_{t}dV\bigg)dt\label{3.87}\\
&&+ \ \int^{t}_{0}\bigg[\int_{B_{\tilde{g}}(x_{0},r+1)}
\bigg(2\langle{\rm Rm},\partial_{t}{\rm Rm}\rangle_{g}
+g^{\ast-3}\ast{\rm Rm}^{\ast2}\ast\partial_{t}g^{-1}\bigg)\xi^{2}\!\ dV\bigg]
dt.\nonumber
\end{eqnarray}
By the estimate (\ref{3.78}) we have
\begin{equation}
\int_{B_{\tilde{g}}(x_{0},r+1)}|\widetilde{{\rm Rm}}|^{2}_{\tilde{g}}\xi^{2}\!\ d\tilde{V}
\lesssim1.\label{3.88}
\end{equation}
Using (\ref{3.76}) implies
\begin{eqnarray}
\int^{t}_{0}\int_{B_{\tilde{g}}(x_{0},r+1)}
|{\rm Rm}|^{2}_{g}\xi^{2}\partial_{t}dVdt&=&\int^{t}_{0}\int_{B_{\tilde{g}}(x_{0},r+1)}g^{\ast-4}
\ast{\rm Rm}^{\ast2}\ast\bigg(g^{\ast-2}
\ast{\rm Rm}\nonumber\\
&&+ \ g^{-1}\ast(\nabla\phi)^{\ast2}+g^{-1}\ast\nabla V\bigg)
\xi^{2}\!\ dV dt\nonumber\\
&=&\int^{t}_{0}\int_{B_{\tilde{g}}(x_{0},r+1)}
\bigg(g^{\ast-5}\ast{\rm Rm}^{\ast2}
\ast(\nabla\phi)^{\ast2}\label{3.89}\\
&&+ \ g^{\ast-6}\ast{\rm Rm}^{\ast3}+g^{\ast-5}\ast{\rm Rm}^{\ast2}\ast\nabla V\bigg)
\xi^{2}\!\ dV dt.\nonumber
\end{eqnarray}
By the evolution
\begin{equation*}
\partial_{t}g^{-1}=g^{\ast-3}\ast{\rm Rm}+g^{\ast-2}
\ast\nabla V+g^{\ast-2}\ast(\nabla\phi)^{\ast2},
\end{equation*}
above (\ref{3.73}), we get
\begin{eqnarray}
&&\int^{t}_{0}\int_{B_{\tilde{g}}(x_{0},r+1)}
g^{\ast-3}\ast{\rm Rm}^{\ast2}\ast\partial_{t}g^{-1}\xi^{2}\!\ dVdt\label{3.90}\\
&=&\int^{t}_{0}\int_{B_{\tilde{g}}(x_{0},r+1)}
g^{\ast-5}\ast{\rm Rm}^{\ast2}\ast\left(g^{-1}
\ast{\rm Rm}+\nabla V+(\nabla\phi)^{\ast2}\right)
\xi^{2}\!\ dV dt.\nonumber
\end{eqnarray}
According to (\ref{3.72}),
\begin{eqnarray}
&&2\int^{t}_{0}\int_{B_{\tilde{g}}(x_{0},r+1)}
\langle{\rm Rm},\partial_{t}{\rm Rm}\rangle_{g}\xi^{2}\!\ dV_{t}dt\label{3.91}\\
&=&2\int^{t}_{0}\int_{B_{\tilde{g}}(x_{0},r+1)}
\langle{\rm Rm},\Delta{\rm Rm}\rangle_{g}\xi^{2}\!\ dVdt
+\int^{t}_{0}\int_{B_{\tilde{g}}(x_{0},r+1)}
\bigg(g^{\ast-4}\ast{\rm Rm}
\ast(\nabla^{2}\phi)^{\ast2}\nonumber\\
&&+ \ g^{\ast-5}\ast{\rm Rm}^{\ast2}\ast\nabla V
+g^{\ast-5}\ast V\ast{\rm Rm}\ast\nabla {\rm Rm}
+g^{\ast-6}\ast{\rm Rm}^{\ast3}\bigg)\xi^{2}\!\ dVdt.\nonumber
\end{eqnarray}
Substituting (\ref{3.88}), (\ref{3.89}), (\ref{3.90}), and (\ref{3.91}), into (\ref{3.87}), we arrive at
\begin{eqnarray}
&&\int_{B_{\tilde{g}}(x_{0},r+1)}
|{\rm Rm}|^{2}_{g}\xi^{2}\!\ dV\nonumber\\
&\leq&C_{1}+2\int^{t}_{0}\int_{B_{\tilde{g}}(x_{0},r+1)}
\langle{\rm Rm},\Delta{\rm Rm}\rangle_{g}\xi^{2}\!\ dVdt
+\int^{t}_{0}\int_{B_{\tilde{g}}(x_{0},r+1)}
\bigg(g^{\ast-6}\ast{\rm Rm}^{\ast3}\nonumber\\
&&+ \ g^{\ast-5}\ast{\rm Rm}^{\ast2}\ast\nabla V
+g^{\ast-5}\ast V\ast{\rm Rm}\ast\nabla{\rm Rm}\label{3.92}\\
&&+ \ g^{\ast-5}\ast{\rm Rm}^{\ast2}\ast(\nabla\phi)^{\ast2}
+g^{\ast-4}\ast{\rm Rm}\ast(\nabla^{2}\phi)^{\ast2}\bigg)\xi^{2}\!\ dVdt,\nonumber
\end{eqnarray}
where $C_{1}$ is a uniform positive
constant independent of $t$ and $x_{0}$. The second term on the right-hand side
of (\ref{3.92}) was estimated in \cite{S89}(equation (78) in page 274):
\begin{eqnarray}
&&2\int^{t}_{0}\int_{B_{\tilde{g}}(x_{0},r+1)}
\langle{\rm Rm},\Delta{\rm Rm}\rangle_{g}\xi^{2}\!\ dVdt\nonumber\\
&\leq&-\frac{3}{2}\int^{t}_{0}\int_{B_{\tilde{g}}(x_{0},r+1)}
|\nabla{\rm Rm}|^{2}_{g}\xi^{2}\!\ dV_{t}dt
+C_{2}\int^{t}_{0}\int_{B_{\tilde{g}}(x_{0},r+1)}|{\rm Rm}|^{2}_{g}dVdt.
\label{3.93}\end{eqnarray}
Using $V=g^{-1}\ast\tilde{\nabla}g$ and (\ref{3.71}), as showed in \cite{S89}
(equations (80), (81), (88), pages 275--277), we obtain
\begin{eqnarray}
&&\int^{t}_{0}\int_{B_{\tilde{g}}(x_{0},r+1)}
g^{\ast-5}\ast V\ast{\rm Rm}\ast\nabla{\rm Rm}\!\ \xi^{2}\!\ dVdt\nonumber\\
&\leq&\frac{1}{4}\int^{t}_{0}\int_{B_{\tilde{g}}(x_{0},r+1)}
|\nabla{\rm Rm}|^{2}_{g}\xi^{2}\!\ dV_{t}dt
+C_{3}\int^{k}_{0}\int_{B_{\tilde{g}}(x_{0},r+1)}
|{\rm Rm}|^{2}_{g}\!\ dVdt,\label{3.94}\\
&&\int^{t}_{0}\int_{B_{\tilde{g}}(x_{0},r+1)}
g^{\ast-5}\ast{\rm Rm}^{\ast2}\ast\nabla V\!\ \xi^{2}\!\ dVdt\nonumber\\
&\leq&\frac{1}{8}\int^{t}_{0}
\int_{B_{\tilde{g}}(x_{0},r+1)}|\nabla{\rm Rm}|^{2}_{g}\xi^{2}\!\ dVdt
+C_{4}\int^{t}_{0}\int_{B_{\tilde{g}}(x_{0},r+1)}|{\rm Rm}|^{2}_{g}dVdt,
\label{3.95}\\
&&\int^{t}_{0}\int_{B_{\tilde{g}}(x_{0},r+1)}
g^{\ast-6}\ast{\rm Rm}^{\ast3}\!\ \xi^{2}\!\ dVdt\nonumber\\
&\leq&\frac{1}{8}\int^{t}_{0}\int_{B_{\tilde{g}}(x_{0},r+1)}
|\nabla{\rm Rm}|^{2}_{g}\xi^{2}\!\ dVdt+C_{5}\int^{t}_{0}
\int_{B_{\tilde{g}}(x_{0},r+1)}|{\rm Rm}|^{2}_{g} dVdt.\label{3.96}
\end{eqnarray}
Plugging (\ref{3.93}), (\ref{3.94}), (\ref{3.95}), (\ref{3.96}) into
(\ref{3.92}), yields
\begin{eqnarray}
&&\int_{B_{\tilde{g}}(x_{0},r+1)}|{\rm Rm}|^{2}_{g}\xi^{2}\!\ dV\nonumber\\
&\leq&-\int^{t}_{0}\int_{B_{\tilde{g}}(x_{0},r+1)}
|\nabla{\rm Rm}|^{2}_{g}\xi^{2}\!\ dVdt+C_{6}\int^{t}_{0}\int_{B_{\tilde{g}}(x_{0},r+1)}
|{\rm Rm}|^{2}_{g}dVdt+C_{7}\nonumber\\
&&+ \ \int^{t}_{0}\int_{B_{\tilde{g}}(x_{0},r+1)}
\bigg(g^{\ast-5}\ast{\rm Rm}^{\ast2}\ast(\nabla\phi)^{\ast2}
+g^{\ast-4}\ast{\rm Rm}\ast(\nabla^{2}\phi)^{\ast2}\bigg)\xi^{2}\!\ dVdt.
\label{3.97}
\end{eqnarray}
Using (\ref{3.71}) and the Cauchy-Schwarz inequality, we can conclude that
\begin{equation*}
g^{\ast-5}\ast{\rm Rm}^{\ast2}\ast(\nabla\phi)^{\ast2}
\lesssim|{\rm Rm}|^{2}_{g}|\nabla\phi|^{2}_{g};
\end{equation*}
but $|\nabla\phi|^{2}_{g}=g^{ij}\tilde{\nabla}_{i}\phi
\tilde{\nabla}_{j}\phi\leq 2|\tilde{\nabla}\phi|^{2}_{\tilde{g}}\lesssim1$, the
above quantity $g^{\ast-5}\ast{\rm Rm}^{\ast2}\ast(\nabla\phi)^{\ast2}$
is bounded from above by $|{\rm Rm}|^{2}_{g}$. From the equation (90) in page 277 of \cite{S89},
\begin{equation*}
g^{\ast-5}\ast{\rm Rm}^{\ast2}\ast(\nabla\phi)^{\ast2}
\lesssim1+|\nabla\tilde{\nabla}g|^{2}_{\tilde{g}}\lesssim
1+|\nabla\tilde{\nabla}g|^{2}_{g}\lesssim1+|\nabla\tilde{\nabla}
\boldsymbol{\Theta}|^{2}_{g};
\end{equation*}
this estimate together with Lemma \ref{l3.11} gives us
\begin{equation}
\int^{t}_{0}\int_{B_{\tilde{g}}(x_{0},r+1)}
g^{\ast-5}\ast{\rm Rm}^{\ast2}\ast(\nabla\phi)^{\ast2}\!\ \xi^{2}\!\ dVdt
\lesssim1.\label{3.98}
\end{equation}
To deal with the last term on the right-hand side of (\ref{3.97}), we perform
the integration by parts to obtain
\begin{eqnarray*}
&&\int_{B_{\tilde{g}}(x_{0},r+1)}
g^{\ast-4}\ast{\rm Rm}\ast(\nabla^{2}\phi)^{\ast2}\!\ \xi^{2}\!\ dV\\
&=&\int_{B_{\tilde{g}}(x_{0},r+1)}
g^{\ast-4}\ast\nabla\phi\ast
\bigg(\nabla{\rm Rm}\ast\nabla^{2}\phi
+{\rm Rm}\ast\nabla^{3}\phi+{\rm Rm}
\ast\nabla^{2}\phi\ast\frac{\nabla\xi}{\xi}\bigg)\xi^{2}\!\ dV\\
&\leq&\frac{1}{2}\int_{B_{\tilde{g}}(x_{0},r+1)}|\nabla{\rm Rm}|^{2}_{g}
\xi^{2}\!\ dV+C_{8}\int_{B_{\tilde{g}}(x_{0},r+1)}|\nabla^{2}\phi|^{2}_{g}dV
\\
&&+ \ \epsilon\int_{B_{\tilde{g}}(x_{0},r+1)}|\nabla^{3}\phi|^{2}_{g}
\xi^{2}dV+C_{\epsilon}\int_{B_{\tilde{g}}(x_{0},r+1)}|{\rm Rm}|^{2}_{g}dV\\
&&+ \ C_{9}\int_{B_{\tilde{g}}(x_{0},r+1)}|{\rm Rm}|^{2}_{g}
dV+C_{10}\int_{B_{\tilde{g}}(x_{0},r+1)}|\nabla^{2}\phi|^{2}_{g}dV\\
&\leq&\frac{1}{2}\int_{B_{\tilde{g}}(x_{0},r+1)}|\nabla{\rm Rm}|^{2}_{g}
\xi^{2}\!\ dV+\epsilon\int_{B_{\tilde{g}}(x_{0},r+1)}|\nabla^{3}\phi|^{2}_{g}
\xi^{2}\!\ dV\\
&&+ \ C_{11}\int_{B_{\tilde{g}}(x_{0},r+1)}|\nabla\tilde{\nabla}
\boldsymbol{\Theta}|^{2}_{g}dV
\end{eqnarray*}
using (\ref{3.71}), $\nabla^{2}\phi=\nabla\tilde{\nabla}\phi$, and
$|{\rm Rm}|^{2}_{g}\lesssim1+|\nabla\tilde{\nabla}g|^{2}_{g}$. Note also that
the second term on the right-hand side of (\ref{3.97}) is uniformly bounded
by the same estimate for $|{\rm Rm}|^{2}_{g}$ and Lemma \ref{l3.11}. Consequently,
\begin{eqnarray}
\int_{B_{\tilde{g}}(x_{0},r+1)}
|{\rm Rm}|^{2}_{g}\xi^{2}\!\ dV&\leq&-\frac{1}{2}\int^{t}_{0}
\int_{B_{\tilde{g}}(x_{0},r+1)}|\nabla{\rm Rm}|^{2}_{g}\xi^{2}\!\ dV dt\nonumber\\
&&+ \ \epsilon\int^{t}_{0}\int_{B_{\tilde{g}}(x_{0},r+1)}
|\nabla^{3}\phi|^{2}_{g}\xi^{2}\!\ dVdt+C_{12}.\label{3.99}
\end{eqnarray}

We next establish the similar inequality for $|\nabla^{2}\phi|^{2}_{g}=g^{ik}g^{j\ell}\nabla_{i}\nabla_{j}\phi
\nabla_{k}\nabla_{\ell}\phi$. Calculate $\partial_{t}|\nabla^{2}\phi|^{2}_{g}=2\langle\nabla^{2}\phi,
\partial_{t}\nabla^{2}\phi\rangle_{g}+(\nabla^{2}\phi)^{\ast2}\ast g^{-1}
\ast\partial_{t}g^{-1}$ and
\begin{eqnarray}
&&\int_{B_{\tilde{g}}(x_{0},r+1)}|\nabla^{2}\phi|^{2}_{g}\xi^{2}\!\ dV
\label{3.100}\\
&=&\int_{B_{\tilde{g}}(x_{0},r+1)}|\tilde{\nabla}^{2}\tilde{\phi}|^{2}_{\tilde{g}}
\xi^{2} d\tilde{V}+\int^{t}_{0}\bigg(\frac{d}{dt}
\int_{B_{\tilde{g}}(x_{0},r+1)}|\nabla^{2}\phi|^{2}_{g}\xi^{2}\!\ dV
\bigg)dt\nonumber\\
&=&\int_{B_{\tilde{g}}(x_{0},r+1)}
|\tilde{\nabla}^{2}\tilde{\phi}|^{2}_{\tilde{g}}\xi^{2}\!\ d\tilde{V}
+\int^{t}_{0}\int_{B_{\tilde{g}}(x_{0},r+1)}|\nabla^{2}\phi|^{2}_{g}
\xi^{2}\partial_{t}dV dt\nonumber\\
&&+ \ \int^{t}_{0}\int_{B_{\tilde{g}}(x_{0},r+1)}
\bigg(2\langle\nabla^{2}\phi,\partial_{t}\nabla^{2}\phi\rangle_{g}
+(\nabla^{2}\phi)^{\ast2}\ast g^{-1}\ast\partial_{t}g^{-1}\bigg)
\xi^{2}\!\ dVdt.\nonumber
\end{eqnarray}
Since $|\tilde{\nabla}^{2}\tilde{\phi}|^{2}_{\tilde{g}}\leq k_{2}$ by the
assumption (\ref{3.70}), we have
\begin{equation}
\int_{B_{\tilde{g}}(x_{0},r+1)}|\tilde{\nabla}^{2}\tilde{\phi}|^{2}_{\tilde{g}}
\xi^{2}\!\ d\tilde{V}\lesssim1.\label{3.101}
\end{equation}
Using (\ref{3.76}) implies
\begin{eqnarray}
\int^{t}_{0}\int_{B_{\tilde{g}}(x_{0},r+1)}|\nabla^{2}\phi|^{2}_{g}
\xi^{2}\partial_{t}dV dt&=&
\int^{t}_{0}\int_{B_{\tilde{g}}(x_{0},r+1)}g^{\ast-2}\ast(\nabla^{2}\phi)^{\ast2}
\ast\bigg(g^{\ast-2}\ast{\rm Rm}\nonumber\\
&&+ \ g^{-1}\ast(\nabla\phi)^{\ast2}
+g^{-1}\ast\nabla V\bigg)\xi^{2}\!\ dVdt\nonumber\\
&=&\int^{t}_{0}\int_{B_{\tilde{g}}(x_{0},r+1)}
\bigg(g^{\ast-4}\ast{\rm Rm}\ast(\nabla^{2}\phi)^{\ast2}\label{3.102}\\
&&+ \ g^{\ast-3}\ast(\nabla\phi)^{\ast2}\ast
(\nabla^{2}\phi)^{\ast2}\nonumber\\
&&+ \ g^{\ast-3}\ast\nabla V\ast(\nabla^{2}
\phi)^{\ast2}\bigg)\xi^{2}\!\ dVdt.\nonumber
\end{eqnarray}
By the evolution equation of $\partial_{t}g^{-1}$ above (\ref{3.90}), we get
\begin{eqnarray}
&&\int^{t}_{0}\int_{B_{\tilde{g}}(x_{0},r+1)}
g^{-1}\ast(\nabla^{2}\phi)^{\ast2}\ast\partial_{t}g^{-1}\!\ \xi^{2}\!\ dVdt\label{3.103}\\
&=&\int^{t}_{0}\int_{B_{\tilde{g}}(x_{0},r+1)}
g^{\ast-3}\ast(\nabla^{2}\phi)^{\ast2}\ast\bigg(g^{-1}\ast{\rm Rm}+\nabla V
+(\nabla\phi)^{\ast2}\bigg)\xi^{2}\!\ dVdt.\nonumber
\end{eqnarray}
By (\ref{3.75}) the third term on the right-hand side of (\ref{3.100}) can be written as
\begin{eqnarray}
&&\int^{t}_{0}\int_{B_{\tilde{g}}(x_{0},r+1)}2\langle\nabla^{2}\phi,
\partial_{t}\nabla^{2}\phi\rangle_{g}\xi^{2}\!\ dVdt\nonumber\\
&=&2\int^{t}_{0}\int_{B_{\tilde{g}}(x_{0},r+1)}
\langle\nabla^{2}\phi,\Delta\nabla^{2}\phi\rangle_{g}\xi^{2}\!\ dVdt
+\int^{t}_{0}\int_{B_{\tilde{g}}(x_{0},r+1)}
\bigg(g^{\ast-2}\ast(\nabla^{2}\phi)^{\ast2}
\nonumber\\
&&+ \ g^{\ast-4}\ast{\rm Rm}\ast(\nabla^{\ast2}\phi)^{2}+g^{\ast-3}\ast(\nabla\phi)^{\ast2}\ast(\nabla^{2}\phi)^{\ast2}
+g^{\ast-3}\ast(\nabla^{2}\phi)^{\ast3}
\label{3.104}\\
&&+ \ g^{\ast-3}\ast\nabla\phi\ast\nabla^{\ast2}\phi\ast\nabla^{3}\phi+g^{\ast-4}\ast{\rm Rm}\ast(\nabla\phi)^{\ast2}\ast\nabla^{2}\phi\nonumber\\
&&+ \ g^{\ast-3}\ast V\ast\nabla^{2}\phi\ast\nabla^{3}\phi+g^{\ast-3}\ast\nabla V\ast(\nabla^{2}\phi)^{\ast2}\bigg)\xi^{2}\!\ dVdt.
\nonumber
\end{eqnarray}
Substituting (\ref{3.101}), (\ref{3.102}), (\ref{3.103}), and (\ref{3.104}) into (\ref{3.100}), we arrive at
\begin{eqnarray}
&&\int_{B_{\tilde{g}}(x_{0},r+1)}|\nabla^{2}\phi|^{2}_{g}\xi^{2}\!\ dV\nonumber\\
&\leq&C_{13}+2\int^{t}_{0}\int_{B_{\tilde{g}}(x_{0},r+1)}
\langle\nabla^{2}\phi,\Delta\nabla^{2}\phi\rangle_{g}\xi^{2}\!\ dVdt\nonumber\\
&&+ \ \int^{t}_{0}\int_{B_{\tilde{g}}(x_{0},r+1)}
\bigg(g^{\ast-4}\ast{\rm Rm}\ast(\nabla^{2}\phi)^{\ast2}
+g^{\ast-3}\ast(\nabla\phi)^{\ast2}\ast(\nabla^{2}\phi)^{\ast2}\label{3.105}\\
&&+ \ g^{\ast-3}\ast(\nabla^{2}\phi)^{\ast3}+g^{\ast-3}
\ast\nabla\phi\ast\nabla^{2}\phi\ast\nabla^{3}\phi
+g^{\ast-4}\ast{\rm Rm}\ast(\nabla\phi)^{\ast2}\ast\nabla^{2}\phi\nonumber\\
&&+ \ g^{\ast-3}\ast V\ast\nabla^{2}\phi\ast\nabla^{3}\phi
+g^{\ast-3}\ast\nabla V\ast(\nabla^{2}\phi)^{\ast2}+
g^{\ast-2}\ast(\nabla^{2}\phi)^{\ast2}\bigg)\xi^{2}\!\ dVdt.\nonumber
\end{eqnarray}
By integration by parts, we find that the first term on the right-hand side of
(\ref{3.105}) equals
\begin{eqnarray*}
&&2\int^{t}_{0}B_{\tilde{g}(x_{0},r+1)}
\langle\xi^{2}\nabla^{2}\phi,g^{\alpha\beta}\nabla_{\alpha}\nabla_{\beta}
\nabla^{2}\phi\rangle_{g}dVdt\\
&=&-2\int^{t}_{0}\int_{B_{\tilde{g}}(x_{0},r+1)}
g^{\alpha\beta}\langle\nabla_{\alpha}(\xi^{2}\nabla^{2}\phi),
\nabla_{\beta}\nabla^{2}\phi\rangle_{g}dVdt\\
&=&-2\int^{t}_{0}\int_{B_{\tilde{g}}(x_{0},r+1)}
g^{\alpha\beta}\langle\nabla_{\alpha}\nabla^{2}\phi,\nabla_{\beta}\nabla^{2}
\phi\rangle_{g}\xi^{2}\!\ dVdt\\
&&+ \ \int^{t}_{0}\int_{B_{\tilde{g}}(x_{0},r+1)}
g^{\ast-4}\ast\nabla^{2}\phi\ast\nabla^{3}\phi\ast\nabla\xi\!\ \xi\!\ dVdt\\
&\leq&-\frac{3}{2}\int^{t}_{0}\int_{B_{\tilde{g}}(x_{0},r+1)}
|\nabla^{3}\phi|^{2}_{g}\xi^{2}\!\ dVdt+C_{14}\int^{t}_{0}
\int_{B_{\tilde{g}}(x_{0},r+1)}|\nabla^{2}\phi|^{2}_{g}dVdt,
\end{eqnarray*}
by (\ref{3.71}) and the fact that $|\nabla\xi|_{g}=|\tilde{\nabla}\xi|_{g}
\leq\sqrt{2}|\tilde{\nabla}\xi|_{\tilde{g}}\leq8\sqrt{2}$. We now estimate the
rest terms on the right-hand side of (\ref{3.105}). By the estimate below
(\ref{3.98}), we have
\begin{eqnarray*}
&&\int_{B_{\tilde{g}}(x_{0},r+1)}
g^{\ast-4}\ast{\rm Rm}\ast(\nabla^{2}\phi)^{\ast2}\xi^{2}dV\nonumber\\
&\leq&\frac{1}{4}\int_{B_{\tilde{g}}(x_{0},r+1)}
|\nabla{\rm Rm}|^{2}_{g}\xi^{2}\!\ dV+\epsilon\int_{B_{\tilde{g}}(x_{0},r+1)}
|\nabla^{3}\phi|^{2}_{g}\xi^{2}\!\ dV\nonumber\\
&&+ \ C_{15}\int_{B_{\tilde{g}}(x_{0},r+1)}
|\nabla\tilde{\nabla}\boldsymbol{\Theta}|^{2}_{g}dV,\nonumber
\end{eqnarray*}
where we replaced the coefficients $1/2$ by $1/4$ (however the proof is the same). Using (\ref{3.71}), $
|{\rm Rm}|^{2}_{g}\lesssim 1+|\nabla\tilde{\nabla}g|^{2}_{g}$ (see the
equation (90) in page 277 of \cite{S89}), and Lemma \ref{l3.11}, the integral of
$g^{\ast-3}\ast(\nabla\phi)^{\ast2}\ast(\nabla^{2}\phi)^{\ast2}
+g^{\ast-4}\ast{\rm Rm}\ast(\nabla\phi)^{\ast2}
\ast\nabla^{2}\phi+g^{\ast-2}\ast(\nabla^{2}\phi)^{\ast2}$ is bounded by
\begin{equation*}
\int^{t}_{0}\int_{B_{\tilde{g}}(x_{0},r+1)}
\left[|\nabla^{2}\phi|^{2}_{g}+|{\rm Rm}|^{2}_{g}\right]dVdt
\lesssim\int^{t}_{0}\int_{B_{\tilde{g}}(x_{0},r+1)}
\left[1+|\nabla\tilde{\nabla}\boldsymbol{\Theta}|^{2}_{g}\right]dVdt
\lesssim1
\end{equation*}
where we used the fact that $T$ depends on the given constants and the volume
estimate (\ref{3.78}), from the second step to the third step. By (\ref{3.71}), we get
\begin{eqnarray*}
&&\int^{t}_{0}\int_{B_{\tilde{g}}(x_{0},r+1)}
g^{\ast-3}\ast(\nabla^{2}\phi)^{\ast3}\!\ \xi^{2}\!\ dVdt\\
&=&\int^{t}_{0}\int_{B_{\tilde{g}}(x_{0},r+1)}\nabla\phi\ast\nabla\left(\xi^{2}g^{\ast-3}
\ast(\nabla^{2}\phi)^{\ast2}\right)dVdt\\
&=&\int^{t}_{0}\int_{B_{\tilde{g}}(x_{0},r+1)}
\nabla\phi\ast\bigg(\xi\nabla\xi\ast g^{\ast-3}
\ast(\nabla^{2}\phi)^{\ast2}+\xi^{2}g^{\ast-4}\ast\nabla g\ast(\nabla^{2}\phi)^{\ast2}
\\
&&+ \ \xi^{2}g^{\ast-3}\ast\nabla^{2}\phi\ast\nabla^{3}\phi\bigg)dVdt\\
&\leq&\epsilon\int^{t}_{0}\int_{B_{\tilde{g}}(x_{0},r+1)}
|\nabla^{3}\phi|^{2}_{g}\xi^{2}\!\ dVdt+C_{16}\int^{t}_{0}
\int_{B_{\tilde{g}}(x_{0},r+1)}|\nabla^{2}\phi|^{2}_{g}dVdt;
\end{eqnarray*}
similarly, according to the definition $V=g^{-1}\ast\tilde{\nabla}g$,
\begin{eqnarray*}
&&\int^{t}_{0}\int_{B_{\tilde{g}}(x_{0},r+1)}
g^{\ast-3}\ast\nabla\phi\ast\nabla^{2}\phi\ast\nabla^{3}\phi\!\ \xi^{2}\!\
dVdt\\
&\leq&\epsilon\int^{t}_{0}\int_{B_{\tilde{g}}(x_{0},r+1)}
|\nabla^{3}g|^{2}_{g}\xi^{2}\!\ dVdt+C_{17}\int^{t}_{0}
\int_{B_{\tilde{g}}(x_{0},r+1)}|\nabla^{2}g|^{2}_{g}dVdt,\\
&&\int^{t}_{0}\int_{B_{\tilde{g}}(x_{0},r+1)}
g^{\ast-3}\ast V\ast\nabla^{2}\phi\ast\nabla^{3}\phi\!\ \xi^{2}dVdt\\
&\leq&\epsilon\int^{t}_{0}\int_{B_{\tilde{g}}(x_{0},r+1)}
|\nabla^{3}\phi|^{2}_{g}\xi^{2}\!\ dVdt+C_{18}\int^{t}_{0}
\int_{B_{\tilde{g}}(x_{0},r+1)}|\nabla^{2}\phi|^{2}_{g}dVdt.
\end{eqnarray*}
Taking the integration by parts on $\nabla V$ yields
\begin{eqnarray*}
&&\int^{t}_{0}\int_{B_{\tilde{g}}(x_{0},r+1)}
g^{\ast-3}\ast\nabla V\ast(\nabla^{2}\phi)^{\ast2}\xi^{2}\!\ dVdt\\
&=&\int^{t}_{0}\int_{B_{\tilde{g}}(x_{0},r+1)}
V\ast\nabla\left(\xi^{2}g^{\ast-3}\ast(\nabla^{2}\phi)^{\ast2}\right)dVdt\\
&=&\int^{t}_{0}\int_{B_{\tilde{g}}(x_{0},r+1)}V\ast\bigg(\xi^{2}
g^{\ast-4}\ast\nabla g\ast(\nabla^{2}\phi)^{\ast2}
+\xi^{2}g^{\ast-3}\ast\nabla^{2}\phi\ast\nabla^{3}\phi\\
&&+ \ g^{\ast-3}\ast(\nabla^{2}\phi)^{\ast2}\ast\nabla\xi\!\ \xi\bigg)
dVdt\\
&\leq&\epsilon\int^{t}_{0}\int_{B_{\tilde{g}}(x_{0},r+1)}
|\nabla^{3}\phi|^{2}_{g}\xi^{2}\!\ dVdt
+C_{19}\int^{t}_{0}\int_{B_{\tilde{g}}(x_{0},r+1)}
|\nabla^{2}\phi|^{2}_{g}dVdt.
\end{eqnarray*}
Substituting the above estimates into (\ref{3.105}) and using Lemma \ref{l3.11}, we have
\begin{eqnarray}
\int_{B_{\tilde{g}}(x_{0},r+1)}
|\nabla^{2}\phi|^{2}_{g}\xi^{2}\!\ dV
&\leq&C_{20}+\frac{1}{4}\int^{t}_{0}\int_{\Omega}|\nabla{\rm Rm}|^{2}_{g}\xi^{2}\!\
dVdt\nonumber\\
&&- \ \left(\frac{3}{2}-5\epsilon\right)\int^{t}_{0}
\int_{B_{\tilde{g}}(x_{0},r+1)}|\nabla^{3}\phi|^{2}_{g}\xi^{2}\!\ dVdt,\label{3.106}
\end{eqnarray}
where $\epsilon$ is a sufficiently small positive number that shall be
determined later. Combining (\ref{3.100}) with (\ref{3.106}), we arrive at
\begin{eqnarray}
\int_{B_{\tilde{g}}(x_{0},r+1)}\left[|{\rm Rm}|^{2}_{g}+|\nabla^{2}
\phi|^{2}_{g}\right]\xi^{2}dV&\leq&C_{21}
-\frac{1}{4}\int^{t}_{0}\int_{B_{\tilde{g}}(x_{0},r+1)}
|\nabla{\rm Rm}|^{2}_{g}\xi^{2}dVdt\nonumber\\
&&-\left(\frac{3}{2}-5\epsilon\right)
\int^{t}_{0}\int_{B_{\tilde{g}}(x_{0},r+1)}|\nabla^{3}\phi|^{2}_{g}\xi^{2}dVdt.\label{3.107}
\end{eqnarray}
As a consequence of the above estimate (\ref{3.107}), we can conclude that the integral of
$|{\rm Rm}|^{2}_{g}+|\nabla^{2}\phi|^{2}_{g}$ over the metric ball $B_{\tilde{g}}
(x_{0},r)$ is uniformly bounded from above. We here keep the minus terms on the
right-hand side of (\ref{3.107}) to deal with the integral of
$|\nabla V|^{2}_{g}\xi^{2}$
over $B_{\tilde{g}}(x_{0},r+1)$, and therefore, we prove Lemma \ref{l3.12}.

Since the metric $g$ is equivalent to $\tilde{g}$, we may write $g^{\ast-k}
\ast g^{\ell}=g^{\ast(\ell-k)}$. Under this convenience, the equation
(\ref{3.74}) can be written as
\begin{eqnarray}
\partial_{t}\nabla V&=&\Delta\nabla V+g^{\ast-3}\ast\nabla\tilde{\nabla}g
\ast{\rm Rm}+g^{\ast-3}\ast\tilde{\nabla}g\ast\nabla{\rm Rm}
+g^{\ast-2}\ast\nabla\tilde{\nabla}g\ast\nabla V\nonumber\\
&&+ \ g^{\ast-2}\ast\tilde{\nabla}g\ast\nabla^{2}V
+g^{\ast-2}\ast\nabla\tilde{\nabla}g\ast(\nabla\phi)^{\ast2}
+g^{\ast-2}\ast\tilde{\nabla}g\ast\nabla\phi\ast\nabla^{2}\phi.\label{3.108}
\end{eqnarray}
From $|\nabla V|^{2}_{g}=g^{ik}g^{j\ell}\nabla_{i}V_{j}
\nabla_{k}V_{\ell}$ we obtain
\begin{equation*}
\partial_{t}|\nabla V|^{2}_{g}=2\langle\nabla V,\partial_{t}\nabla V\rangle_{g}
+g^{\ast-4}\ast{\rm Rm}\ast(\nabla V)^{\ast2}+g^{\ast-3}\ast(\nabla V)^{\ast3};
\end{equation*}
by the evolution equation of $\partial_{t}g^{-1}$ above (\ref{3.73}), we arrive
at
\begin{eqnarray}
\partial_{t}|\nabla V|^{2}_{g}&=&2\langle\nabla V,\partial_{t}\nabla V\rangle_{g}
+g^{\ast-4}\ast{\rm Rm}\ast(\nabla V)^{\ast2}+g^{\ast-3}\ast(\nabla V)^{\ast3}\nonumber\\
&&+ \ g^{\ast-3}\ast(\nabla V)^{\ast2}\ast(\nabla\phi)^{\ast2}.\label{3.109}
\end{eqnarray}
Calculate
\begin{eqnarray}
&&\int_{B_{\tilde{g}}(x_{0},r+1)}|\nabla V|^{2}_{g}\xi^{2}\!\ dV \ \ = \ \ \int^{t}_{0}\bigg(\frac{d}{dt}\int_{B_{\tilde{g}}(x_{0},r+1)}
|\nabla V|^{2}_{g}\xi^{2}\!\ dV\bigg)dt\label{3.110}\\
&=&\int^{t}_{0}\int_{B_{\tilde{g}}(x_{0},r+1)}
|\nabla V|^{2}\xi^{2}\partial_{t}dV dt+\int^{t}_{0}\int_{B_{\tilde{g}}(x_{0},r+1)}
\partial_{t}|\nabla V|^{2}_{g}\!\ \xi^{2}\!\ dVdt.\nonumber
\end{eqnarray}
since $V=0$ at $t=0$. Plugging (\ref{3.108}), (\ref{3.109}) into (\ref{3.100}),
we get
\begin{eqnarray}
&&\int_{B_{\tilde{g}}(x_{0},r+1)}|\nabla V|^{2}_{g}\xi^{2}\!\ dV\nonumber\\
&=&2\int^{t}_{0}\int_{B_{\tilde{g}}(x_{0},r+1)}\langle\nabla V,\Delta\nabla V
\rangle_{g}\xi^{2}\!\ dVdt\nonumber\\
&&+ \ \int^{t}_{0}\int_{B_{\tilde{g}}(x_{0},r+1)}\bigg[g^{\ast-5}\ast\nabla\tilde{\nabla}g
\ast{\rm Rm}\ast\nabla V+g^{\ast-5}\ast\tilde{\nabla}g\ast\nabla{\rm Rm}
\ast\nabla V\nonumber\\
&&+ \ g^{\ast-4}\ast\nabla\tilde{\nabla}g\ast(\nabla V)^{\ast2}
+g^{\ast-4}\ast\tilde{\nabla}g\ast\nabla V\ast\nabla^{2}V
+g^{\ast-4}\ast{\rm Rm}\ast(\nabla V)^{\ast2}\label{3.111}\\
&&+ \ g^{\ast-3}\ast(\nabla V)^{\ast3}
+g^{\ast-4}\ast\nabla\tilde{\nabla}g\ast\nabla V\ast(\nabla\phi)^{\ast2}
\nonumber\\
&&+ \ g^{\ast-4}\ast\tilde{\nabla}g\ast\nabla V\ast\nabla\phi
\ast\nabla^{2}\phi+g^{\ast-3}\ast(\nabla V)^{\ast2}\ast(\nabla\phi)^{\ast2}
\bigg]\xi^{2}\!\ dVdt.\nonumber
\end{eqnarray}
The first term on the right-hand side of (\ref{3.111})was computed in \cite{S89} (
see the equation (104) in page 280):
\begin{eqnarray}
2\int^{t}_{0}\int_{B_{\tilde{g}}(x_{0},r+1)}
\langle\nabla V,\Delta\nabla V\rangle_{g}\xi^{2}dVdt
&\leq&-\frac{15}{8}\int^{t}_{0}\int_{B_{\tilde{g}}(x_{0},r+1)}
|\nabla^{2}V|^{2}_{g}\xi^{2}\!\ dVdt\nonumber\\
&&+ \ C_{22}\int^{t}_{0}\int_{B_{\tilde{g}}(x_{0},r+1)}
|\nabla V|^{2}_{g}dVdt.\label{3.112}
\end{eqnarray}
Define
\begin{eqnarray*}
I_{1}&:=&g^{\ast-5}\ast\tilde{\nabla}g\ast\nabla{\rm Rm}\ast\nabla V
+g^{\ast-4}\ast\tilde{\nabla}g\ast\nabla V\ast\nabla^{2}V,\\
I_{2}&:=&g^{\ast-5}\ast\nabla\tilde{\nabla}g\ast{\rm Rm}
\ast\nabla V,\\
I_{3}&:=&g^{\ast-4}\ast\nabla\tilde{\nabla}g\ast(\nabla V)^{\ast2}
+g^{\ast-4}\ast{\rm Rm}\ast(\nabla V)^{\ast2}+g^{\ast-3}
\ast(\nabla V)^{\ast3},\\
I_{4}&:=&g^{\ast-4}\ast\nabla\tilde{\nabla}g\ast\nabla V\ast(\nabla\phi)^{\ast2}
+g^{\ast-4}\ast\tilde{\nabla}g\ast\nabla V\ast\nabla\phi\ast\nabla^{2}\phi\\
&&+ \ g^{\ast-3}\ast(\nabla V)^{\ast2}\ast(\nabla\phi)^{\ast2}.
\end{eqnarray*}
According to (106) in page 280 of \cite{S89}, we have
\begin{eqnarray}
\int^{t}_{0}\int_{B_{\tilde{g}}(x_{0},r+1)}
I_{1}\xi^{2}\!\ dVdt&\leq&\frac{1}{16}\int^{t}_{0}\int_{B_{\tilde{g}}(x_{0},r+1)}
\left[|\nabla{\rm Rm}|^{2}_{g}+|\nabla^{2}V|^{2}_{g}\right]\xi^{2}\!\ dVdt\nonumber\\
&&+ \ C_{23}\int^{t}_{0}\int_{B_{\tilde{g}}(x_{0},r+1)}
|\nabla V|^{2}_{g}dVdt;\label{3.113}
\end{eqnarray}
according to (107) in page 280 and (112) in page 281 of \cite{S89}, we have
\begin{eqnarray}
\int^{t}_{0}\int_{B_{\tilde{g}}(x_{0},r+1)}
I_{2}\xi^{2}\!\ dVdt&\leq&\frac{1}{16}\int^{t}_{0}\int_{B_{\tilde{g}}(x_{0},
r+1)}\left[|\nabla{\rm Rm}|^{2}_{g}+|\nabla^{2}V|^{2}_{g}\right]\xi^{2}\!\ dVdt
\nonumber\\
&&+ \ C_{24}\int^{t}_{0}\int_{B_{\tilde{g}}(x_{0},r+1)}
\left[|{\rm Rm}|^{2}_{g}+|\nabla V|^{2}_{g}\right]dVdt,\label{3.114}\\
\int^{t}_{0}\int_{B_{\tilde{g}}(x_{0},r+1)}I_{3}\xi^{2}\!\ dVdt&\leq&
\frac{1}{8}\int^{t}_{0}\int_{B_{\tilde{g}}(x_{0},r+1)}
|\nabla^{2}V|^{2}_{g}\xi^{2}\!\ dVdt\nonumber\\
&&+ \ C_{25}\int^{t}_{0}\int_{B_{\tilde{g}}(x_{0},r+1)}
|\nabla V|^{2}_{g}\xi^{2}\!\ dVdt.\label{3.115}
\end{eqnarray}
Using (\ref{3.71}) implies
\begin{eqnarray}
&&\int^{t}_{0}\int_{B_{\tilde{g}}(x_{0},r+1)}
I_{4}\xi^{2}\!\ dVdt\nonumber\\
&\lesssim&\int^{t}_{0}\int_{B_{\tilde{g}}(x_{0},r+1)}
\bigg(|\nabla\tilde{\nabla}g|^{2}_{g}+|\nabla V|^{2}_{g}
+|\nabla^{2}\phi|^{2}_{g}\bigg)\xi^{2}\!\ dVdt.\label{3.116}
\end{eqnarray}
Substituting (\ref{3.112}), (\ref{3.113}), (\ref{3.114}), (\ref{3.115}),
(\ref{3.116}) into (\ref{3.111}), using
the fact that $\nabla V=g^{-1}\ast\nabla\tilde{\nabla}g$, and using Lemma \ref{l3.11}, we obtain
\begin{eqnarray}
\int_{B_{\tilde{g}}(x_{0},r+1)}|\nabla V|^{2}_{g}\xi^{2}\!\ dV_{t}
&\leq&-\frac{13}{8}\int^{t}_{0}\int_{B_{\tilde{g}}(x_{0},r+1)}
|\nabla^{2}V|^{2}_{g}\xi^{2}\!\ dVdt\nonumber\\
&&+ \ \frac{1}{8}\int^{t}_{0}\int_{B_{\tilde{g}}(x_{0},r+1)}
|\nabla{\rm Rm}|^{2}_{g}\xi^{2}\!\ dVdt+C_{26}.\label{3.117}
\end{eqnarray}
Choosing $\epsilon=11/40$ in (\ref{3.107}) and combining with (\ref{3.117}), we
arrive at
\begin{eqnarray*}
&&\int_{B_{\tilde{g}}(x_{0},r+1)}\left[|{\rm Rm}|^{2}_{g}
+|\nabla^{2}\phi|^{2}_{g}+|\nabla V|^{2}_{g}\right]\xi^{2}\!\ dV\\
&&+ \ \frac{1}{8}\int^{t}_{0}\int_{B_{\tilde{g}}(x_{0},
r+1)}\left[|\nabla{\rm Rm}|^{2}_{g}+|\nabla^{3}\phi|^{2}_{g}
+|\nabla^{2} V|^{2}_{g}\right]\xi^{2}\!\ dVdt \ \ \lesssim \ \ \ 1.
\end{eqnarray*}
In particular,
\begin{eqnarray*}
\max_{t\in[0,T]}\int_{B_{\tilde{g}}(x_{0},r)}\left[|{\rm Rm}|^{2}_{g}
+|\nabla^{2}\phi|^{2}_{g}+|\nabla V|^{2}_{g}\right]dV&\lesssim&1,\\
\int^{T}_{0}\int_{B_{\tilde{g}}(x_{0},r)}
\left[|\nabla{\rm Rm}|^{2}_{g}+|\nabla^{3}\phi|^{2}_{g}
+|\nabla^{2} V|^{2}_{g}\right]dVdt&\lesssim&1
\end{eqnarray*}
since $\xi\equiv1$ on $B_{\tilde{g}}(x_{0},r)$.
\end{proof}

Recall that
\begin{equation*}
\partial_{t}{\rm Rm}=\Delta{\rm Rm}
+g^{\ast-2}\ast{\rm Rm}^{\ast2}+g^{-1}\ast V\ast\nabla{\rm Rm}
+g^{-1}\ast{\rm Rm}\ast\nabla V+(\nabla^{2}\phi)^{\ast2}.
\end{equation*}
Since $\nabla(g^{-1}\ast V\ast{\rm Rm})=g^{-1}\ast V\ast\nabla{\rm Rm}
+g^{-1}\ast{\rm Rm}\ast\nabla V$, it follows that
\begin{equation}
\partial_{t}{\rm Rm}=\Delta{\rm Rm}+\nabla P_{1}+Q_{1},\label{3.118}
\end{equation}
where
\begin{equation*}
P_{1}:=g^{-1}\ast V\ast{\rm Rm}, \ \ \ Q_{1}:=g^{-1}\ast{\rm Rm}
\ast\nabla V+g^{\ast-2}\ast{\rm Rm}^{\ast2}+(\nabla^{2}\phi)^{\ast2}.
\end{equation*}
Recall the equation
\begin{equation*}
\partial_{t}\nabla V=\nabla(\partial_{t}V)+V\ast\partial_{t}\Gamma, \ \ \
\partial_{t}\Gamma=g^{-1}\ast\nabla(\partial_{t}g)
\end{equation*}
after (\ref{3.73}) and the equation $\partial_{t}g=g^{-1}\ast{\rm Rm}
+\nabla V+(\nabla\phi)^{\ast2}$. Hence
\begin{eqnarray*}
\partial_{t}\nabla V&=&\nabla
\left(\Delta V+g^{\ast-3}\ast\tilde{\nabla}g\ast{\rm Rm}
+g^{\ast-2}\ast\tilde{\nabla}g\ast\nabla V+g^{\ast-2}\ast\tilde{\nabla}g
\ast(\nabla\phi)^{\ast2}\right)\\
&&+ \ g^{\ast-2}\ast V\ast\nabla{\rm Rm}+g^{-1}\ast V\ast\nabla^{2}V
+g^{-1}\ast V\ast\nabla\phi\ast\nabla^{2}\phi.
\end{eqnarray*}
From the Ricci identity $\nabla\Delta V=\Delta\nabla V
+g^{\ast-2}\ast{\rm Rm}\ast\nabla V+g^{\ast-2}\ast V\ast\nabla{\rm Rm}$, it
follows that
\begin{eqnarray*}
\partial_{t}\nabla V&=&\Delta\nabla V+\nabla\left(g^{\ast-3}
\ast\tilde{\nabla}g\ast{\rm Rm}+g^{\ast-2}\ast\tilde{\nabla}g\ast\nabla V
+g^{\ast-2}\ast\tilde{\nabla}g\ast(\nabla\phi)^{\ast2}\right)\\
&&+ \ g^{\ast-2}\ast V\ast\nabla{\rm Rm}+g^{-1}\ast V\ast\nabla^{2}V
+g^{\ast-2}\ast{\rm Rm}\ast\nabla V\\
&&+ \ g^{-1}\ast V\ast\nabla\phi\ast\nabla^{2}\phi.
\end{eqnarray*}
Since
\begin{eqnarray*}
\nabla(g^{\ast-2}\ast V\ast{\rm Rm})
&=&g^{\ast-2}\ast V\ast\nabla{\rm Rm}
+g^{\ast-2}\ast{\rm Rm}\ast\nabla V,\\
\nabla(g^{-1}\ast V\ast\nabla V)
&=&g^{-1}\ast V\ast\nabla^{2} V
+g^{-1}\ast(\nabla V)^{\ast2}
\end{eqnarray*}
and $V=g^{-1}\ast\tilde{\nabla}g$, we obtain
\begin{eqnarray*}
g^{\ast-2}\ast V\ast\nabla{\rm Rm}&=&\nabla(g^{\ast-3}\ast\tilde{\nabla}g
\ast{\rm Rm})+g^{\ast-2}\ast{\rm Rm}\ast\nabla V,\\
g^{-1}\ast V\ast\nabla^{2}V&=&
\nabla(g^{\ast-2}\ast\tilde{\nabla}g\ast\nabla V)
+g^{-1}\ast(\nabla V)^{\ast2}
\end{eqnarray*}
and hence
\begin{equation}
\partial_{t}\nabla V=\Delta\nabla V+\nabla P_{2}+Q_{2},\label{3.119}
\end{equation}
where
\begin{eqnarray*}
P_{2}&=&g^{\ast-3}\ast\tilde{\nabla}g\ast{\rm Rm}
+g^{\ast-2}\ast\tilde{\nabla}g\ast\nabla V+g^{\ast-2}\ast\tilde{\nabla}g
\ast(\nabla\phi)^{\ast2},\\
Q_{2}&=&g^{\ast-2}\ast{\rm Rm}\ast\nabla V+g^{-1}\ast(\nabla V)^{\ast2}
+g^{-1}\ast V\ast\nabla\phi\ast\nabla^{2}\phi.
\end{eqnarray*}
Finally, according to (\ref{3.75}), we have
\begin{equation}
\partial_{t}\nabla^{2}\phi=\Delta\nabla^{2}\phi+\nabla P_{3}+Q_{3},\label{3.120}
\end{equation}
where
\begin{eqnarray*}
P_{3}&=&g^{-1}\ast\nabla\phi\ast\nabla^{2}\phi+g^{-1}\ast V\ast\nabla^{2}\phi,\\
Q_{3}&=&g^{\ast-2}\ast{\rm Rm}\ast\nabla^{2}\phi
+g^{-1}\ast(\nabla\phi)^{\ast2}\ast\nabla^{2}\phi
+g^{\ast-2}\ast{\rm Rm}\ast(\nabla\phi)^{\ast2}
+\beta_{2}\nabla^{2}\phi.
\end{eqnarray*}

\begin{lemma}\label{l3.13} For any integer $n\geq1$, we have
\begin{eqnarray*}
\int^{T}_{0}\int_{B_{\tilde{g}}(x_{0},r)}
u^{n-1}|\nabla\tilde{\nabla}g|^{2}_{g}dVdt&\lesssim&1,\\
\max_{t\in[0,T]}\int_{B_{\tilde{g}}(x_{0},r)}
u^{n}\!\ dV&\lesssim&1,\\
\int^{T}_{0}\int_{B_{\tilde{g}}(x_{0},r)}
u^{n-1}v\!\ dVdt&\lesssim&1.
\end{eqnarray*}
where
\begin{equation*}
u:=|{\rm Rm}|^{2}_{g}+|\nabla^{2}\phi|^{2}_{g}
+|\nabla V|^{2}_{g}, \ \ \ v:=|\nabla{\rm Rm}|^{2}_{g}
+|\nabla^{3}\phi|^{2}_{g}+|\nabla^{2}V|^{2}_{g},
\end{equation*}
and $\lesssim$ depends on $m, n, r, k_{0}, k_{1}, k_{2}, \alpha_{1},
\beta_{1}, \beta_{2}$.
\end{lemma}

\begin{proof} The case $n=1$ follows from Lemma \ref{l3.11} and Lemma \ref{l3.12}.
We now prove by induction on $n$. Suppose that for $s=1,\cdots,n-1$, we have
\begin{equation*}
\int^{T}_{0}\int_{B_{\tilde{g}}(x_{0},r)}
u^{s-1}|\nabla\tilde{\nabla}g|^{2}_{g}dVdt, \
\max_{t\in[0,T]}\int_{B_{\tilde{g}}(x_{0},r)}u^{s}\!\ dV, \ \int^{T}_{0}\int_{B_{\tilde{g}}(x_{0},r)}
u^{s-1}v\!\ dVdt\lesssim1.
\end{equation*}
For convenience, define
\begin{equation*}
w:=|{\rm Rm}|^{2}_{\tilde{g}}+|\nabla^{2}\phi|^{2}_{\tilde{g}}
+|\nabla V|^{2}_{\tilde{g}}.
\end{equation*}
By (66) in page 272 of \cite{S89} and (\ref{3.71}), we have $|\nabla\tilde{\nabla}g|^{2}_{g}\leq 2|\tilde{\nabla}^{2}g|^{2}_{g}
+C_{1}\leq32|\tilde{\nabla}^{2}g|^{2}_{\tilde{g}}+C_{1}$ and hence
\begin{equation*}
\int^{T}_{0}\int_{B_{\tilde{g}}(x_{0},r+1)}
u^{n-1}|\nabla\tilde{\nabla}g|^{2}_{g}dVdt
\leq32\int^{T}_{0}\int_{B_{\tilde{g}}(x_{0},r+1)}
u^{n-1}|\tilde{\nabla}^{2}g|^{2}_{\tilde{g}}dVdt+C_{2}
\end{equation*}
by (\ref{3.78}) and (\ref{3.71}). Since $\frac{1}{16}w\leq u\leq 16 w$, to
estimate
\begin{equation*}
\int^{T}_{0}\int_{B_{\tilde{g}}(x_{0},r+1)}
u^{n-1}|\nabla\tilde{\nabla}g|^{2}_{g}dVdt
\end{equation*}
we suffice to estimate
\begin{equation*}
\int^{T}_{0}\int_{B_{\tilde{g}}(x_{0},r+1)}
w^{n-1}|\tilde{\nabla}^{2}g|^{2}_{\tilde{g}}d\tilde{V}dt
\end{equation*}
since $dV\leq 2^{m/2}d\tilde{V}$. Consider the same cutoff function $\xi(x)$
used in the proof of Lemma \ref{l3.10}. Calculate
\begin{eqnarray*}
K&:=&\frac{d}{dt}\int_{B_{\tilde{g}}(x_{0},r+1)}w^{n-1}
|\tilde{\nabla}g|^{2}_{\tilde{g}}\xi^{2}\!\ d\tilde{V}\\
&=&\int_{B_{\tilde{g}}(x_{0},r+1)}
w^{n-1}2\langle\tilde{\nabla}g,\partial_{t}\tilde{\nabla}g\rangle_{\tilde{g}}
\xi^{2}\!\ d\tilde{V}\\
&&+ \ \int_{B_{\tilde{g}}(x_{0},r+1)}
|\tilde{\nabla}g|^{2}_{\tilde{g}}(n-1)w^{n-2}\left[
\partial_{t}|{\rm Rm}|^{2}_{\tilde{g}}+\partial_{t}|\nabla^{2}\phi|^{2}_{\tilde{g}}
+\partial_{t}|\nabla V|^{2}_{\tilde{g}}\right]\xi^{2}\!\ d\tilde{V}\\
&:=&I+J.
\end{eqnarray*}
Using the evolution equation of $\tilde{\nabla}g$ after (\ref{3.77}) yields
\begin{eqnarray*}
I&=&2\int_{B_{\tilde{g}}(x_{0},r+1)}
\xi^{2}w^{n-1}\bigg\langle\tilde{\nabla}g, g^{\alpha\beta}\tilde{\nabla}_{\alpha}
\tilde{\nabla}_{\beta}\tilde{\nabla}g+g^{-1}\ast g\ast\tilde{\nabla}
\widetilde{{\rm Rm}}+\tilde{\nabla}\phi\ast\tilde{\nabla}^{2}\phi\\
&&+ \ g^{\ast-2}\ast g\ast\tilde{\nabla}g
\ast\widetilde{{\rm Rm}}+g^{\ast-2}\ast\tilde{\nabla}g
\ast\tilde{\nabla}^{2}g+g^{\ast-3}\ast(\tilde{\nabla}g)^{\ast3}
\bigg\rangle_{\tilde{g}}
d\tilde{V}\\
&\leq&C_{3}\int_{B_{\tilde{g}}(x_{0},r+1)}
\left(1+|\tilde{\nabla}^{2}g|_{\tilde{g}}\right)w^{n-1}\xi^{2}\!\ d\tilde{V}
+I_{1}+I_{2}+I_{3}
\end{eqnarray*}
by (\ref{3.71}), where
\begin{eqnarray*}
I_{1}&=&2\int_{B_{\tilde{g}}(x_{0},r+1)}
\xi^{2}w^{n-1}\langle\tilde{\nabla}g,g^{\alpha\beta}\tilde{\nabla}_{\alpha}
\tilde{\nabla}_{\beta}\tilde{\nabla}g\rangle_{\tilde{g}}d\tilde{V},\\
I_{2}&=&\int_{B_{\tilde{g}}(x_{0},r+1)}
\xi^{2}w^{n-1}\ast g^{-1}\ast g\ast\tilde{\nabla}g\ast\tilde{\nabla}
\widetilde{{\rm Rm}}\!\ d\tilde{V},\\
I_{3}&=&\int_{B_{\tilde{g}}(x_{0},r+1)}
\xi^{2}w^{n-1}\ast\tilde{\nabla}g\ast\tilde{\nabla}
\phi\ast\tilde{\nabla}^{2}\phi\!\ d\tilde{V}.
\end{eqnarray*}
By integration by parts, the term $I_{1}$ can be estimated by
\begin{eqnarray*}
I_{1}&=&-2\int_{B_{\tilde{g}}(x_{0},r+1)}
\langle\tilde{\nabla}_{\beta}\tilde{\nabla}g,\tilde{\nabla}_{\alpha}
(g^{\alpha\beta}\xi^{2}w^{n-1}\tilde{\nabla}g)\rangle_{\tilde{g}}d\tilde{V}\\
&=&-2\int_{B_{\tilde{g}}(x_{0},r+1)}\bigg\langle\tilde{\nabla}_{\beta}\tilde{\nabla}g,
g^{\alpha\beta}\xi^{2}w^{n-1}\tilde{\nabla}_{\alpha}\tilde{\nabla}g\\
&&+ \ g^{\ast-2}\ast(\tilde{\nabla}g)^{\ast2}\ast\xi^{2}\ast w^{n-1}+
g^{-1}\ast\xi\ast\tilde{\nabla}\xi\ast\tilde{\nabla}g\ast w^{n-1}\\
&&+ \ g^{-1}\ast\tilde{\nabla}g\ast \xi^{2}w^{n-2}
\ast\left({\rm Rm}\ast\tilde{\nabla}{\rm Rm}+\nabla V\ast\tilde{\nabla}
\nabla V+\nabla^{2}\phi\ast\tilde{\nabla}\nabla^{2}\phi\right)\bigg\rangle_{\tilde{g}}d\tilde{V}\\
&\leq&\int_{B_{\tilde{g}}(x_{0},r+1)}
\bigg[-|\tilde{\nabla}^{2}g|^{2}_{\tilde{g}}\xi^{2}w^{n-1}+C_{4}|\tilde{\nabla}^{2}g|_{\tilde{g}}
\xi w^{n-1}+\xi^{2}w^{n-2}\ast g^{-1}\ast\tilde{\nabla}g\\
&& \ \ast\tilde{\nabla}^{2}g\ast
\left({\rm Rm}\ast\tilde{\nabla}{\rm Rm}+\nabla V\ast\tilde{\nabla}
\nabla V+\nabla^{2}\phi\ast\tilde{\nabla}\nabla^{2}\phi\right)\bigg]d\tilde{V}
\end{eqnarray*}
by (\ref{3.71}). The Cauchy-Schwarz inequality implies
\begin{eqnarray*}
&&C_{4}\int_{B_{\tilde{g}}(x_{0},r+1)}|\tilde{\nabla}^{2}g|_{\tilde{g}}
\xi w^{n-1}\!\ d\tilde{V} \ \ = \ \ C_{4}\int_{B_{\tilde{g}}(x_{0},r+1)}
\left(|\tilde{\nabla}^{2}g|_{\tilde{g}}w^{\frac{n-1}{2}}\xi\right)
w^{\frac{n-1}{2}}\!\ d\tilde{V}\\
&\leq&\frac{1}{8}\int_{B_{\tilde{g}}(x_{0},r+1)}
|\tilde{\nabla}^{2}g|^{2}_{\tilde{g}}\xi^{2}w^{n-1}\!\ d\tilde{V}
+2C^{2}_{4}\int_{B_{\tilde{g}}(x_{0},r+1)}w^{n-1}\!\ d\tilde{V}\\
&\leq&\frac{1}{8}\int_{B_{\tilde{g}}(x_{0},r+1)}
|\tilde{\nabla}^{2}g|^{2}_{\tilde{g}}\xi^{2}w^{n-1}\!\ d\tilde{V}+C_{5}
\end{eqnarray*}
by the inductive hypothesis and $\frac{1}{16}w\leq u\leq 16 w$. Similarly,
\begin{eqnarray*}
&&\int_{B_{\tilde{g}}(x_{0},r+1)}
\xi^{2}w^{n-2}\ast g^{-1}\ast\tilde{\nabla}g\ast\tilde{\nabla}^{2}g
\ast\bigg({\rm Rm}\ast\tilde{\nabla}{\rm Rm}+\nabla V
\ast\tilde{\nabla}\nabla V\\
&&+ \ \nabla^{2}\phi\ast\tilde{\nabla}
\nabla^{2}\phi\bigg)\!\ d\tilde{V}\\
&\leq&C_{6}\int_{B_{\tilde{g}}(x_{0},r+1)}
\left(w^{\frac{n-2}{2}}\xi|\tilde{\nabla}^{2}g|_{\tilde{g}}
|{\rm Rm}|_{\tilde{g}}\right)\left(w^{\frac{n-2}{2}}\xi|\tilde{\nabla}{\rm Rm}
|_{\tilde{g}}\right)d\tilde{V}\\
&&+ \ C_{6}\int_{B_{\tilde{g}}(x_{0},r+1)}
\left(w^{\frac{n-2}{2}}\xi|\tilde{\nabla}^{2}g|_{\tilde{g}}
|\nabla V|_{\tilde{g}}\right)\left(w^{\frac{n-2}{2}}\xi|\tilde{\nabla}
\nabla V|_{\tilde{g}}\right)d\tilde{V}\\
&&+ \ C_{6}\int_{B_{\tilde{g}}(x_{0},r+1)}
\left(w^{\frac{n-2}{2}}\xi|\tilde{\nabla}^{2}g|_{\tilde{g}}
|\nabla^{2}\phi|_{\tilde{g}}\right)
\left(w^{\frac{n-2}{2}}\xi|\tilde{\nabla}\nabla^{2}\phi|_{\tilde{g}}\right)
d\tilde{V}\\
&\leq&\frac{1}{8}\int_{B_{\tilde{g}}(x_{0},r+1)}
|\tilde{\nabla}^{2}g|^{2}_{\tilde{g}}w^{n-1}\xi^{2}\!\ d\tilde{V}\\
&&+ \ C_{7}\int_{B_{\tilde{g}}(x_{0},r+1)}
w^{n-2}\xi^{2}\left(|\tilde{\nabla}{\rm Rm}|^{2}_{\tilde{g}}
+|\tilde{\nabla}\nabla V|^{2}_{\tilde{g}}+|\tilde{\nabla}\nabla^{2}
\phi|^{2}_{\tilde{g}}\right)d\tilde{V}.
\end{eqnarray*}
According to $\Gamma-\tilde{\Gamma}=g^{-1}\ast\tilde{\nabla}g$, we have
\begin{eqnarray*}
\tilde{\nabla}{\rm Rm}&=&\nabla{\rm Rm}+g^{-1}\ast\tilde{\nabla}g
\ast{\rm Rm}, \ \ \ \tilde{\nabla}\nabla V \ \ = \ \ \nabla^{2}V+g^{-1}\ast\tilde{\nabla}g
\ast\nabla V,\\
\tilde{\nabla}\nabla^{2}\phi&=&\nabla^{3}\phi+g^{-1}\ast\tilde{\nabla}g\ast
\nabla^{2}\phi,
\end{eqnarray*}
from which we can conclude that $|\tilde{\nabla}{\rm Rm}|^{2}_{\tilde{g}}
+|\tilde{\nabla}\nabla V|^{2}_{\tilde{g}}+|\tilde{\nabla}\nabla^{2}\phi|^{2}_{\tilde{g}}\lesssim u+v$ and then
\begin{eqnarray*}
&&\int_{B_{\tilde{g}}(x_{0},r+1)}\xi^{2}w^{n-2}
\ast g^{-1}\ast\tilde{\nabla}g\ast\tilde{\nabla}^{2}g\ast
\bigg({\rm Rm}\ast\tilde{\nabla}{\rm Rm}+\nabla V\ast\tilde{\nabla}
\nabla V\\
&&+ \ |\tilde{\nabla}\nabla^{2}\phi|^{2}_{\tilde{g}}\bigg)d\tilde{V}\\
&\leq&\frac{1}{8}\int_{B_{\tilde{g}}(x_{0},r+1)}
|\tilde{\nabla}^{2}g|^{2}_{\tilde{g}}w^{n-1}\xi^{2}\!\ d\tilde{V}
+C_{8}\int_{B_{\tilde{g}}(z_{0},r+1)}w^{n-2}\xi^{2}
(u+v)\!\ d\tilde{V}\\
&\leq&\frac{1}{8}\int_{B_{\tilde{g}}(x_{0},r+1)}
|\tilde{\nabla}g|^{2}_{\tilde{g}}w^{n-1}\xi^{2}\!\ d\tilde{V}
+C_{9}\int_{B_{\tilde{g}}(x_{0},r+1)}u^{n-2}\xi^{2}(u+v)\!\ d\tilde{V}\\
&\leq&\frac{1}{8}\int_{B_{\tilde{g}}(x_{0},r+1)}
|\tilde{\nabla}g|^{2}_{\tilde{g}}w^{n-1}\xi^{2}\!\ d\tilde{V}
+C_{9}\int_{B_{\tilde{g}}(x_{0},r+1)}
u^{n-2}v\xi^{2}\!\ d\tilde{V}+C_{10}
\end{eqnarray*}
by the inductive hypothesis. Consequently,
\begin{equation}
I_{1}\leq-\frac{3}{4}\int_{B_{\tilde{g}}(x_{0},r+1)}
|\tilde{\nabla}g|^{2}_{\tilde{g}}w^{n-1}
\xi^{2}\!\ d\tilde{V}+C_{11}\int_{B_{\tilde{g}}(x_{0},r+1)}
u^{n-2}v\!\ d\tilde{V}+C_{12}.\label{3.121}
\end{equation}
If we directly use the inequalities (\ref{3.71}), we can get the uniform
upper bound for $I_{2}$ by the inductive hypothesis. However, in this case
the bound shall depend on an upper bound of $|\tilde{\nabla}\widetilde{{\rm
Rm}}|_{\tilde{g}}$. To fund the dependence of $k_{0}$, we will argue as
follows. Again using the integration by parts, we get
\begin{eqnarray}
I_{2}&=&-\int_{B_{\tilde{g}}(x_{0},r+1)}
\left\langle\widetilde{{\rm Rm}},\tilde{\nabla}
\left(g^{-1}\ast g\ast\tilde{\nabla}g\ast\xi^{2}\ast w^{n-1}\right)\right\rangle_{\tilde{g}}
d\tilde{V}\nonumber\\
&=&\int_{B_{\tilde{g}}(x_{0},r+1)}
\bigg\langle\widetilde{{\rm Rm}},g^{\ast-2}\ast(\tilde{\nabla}g)^{\ast2}
\ast g\ast\xi^{2}\ast w^{n-1}+\left((\tilde{\nabla}g)^{\ast2}
+g\ast\tilde{\nabla}^{2}g\right)\nonumber\\
&& \ \ast g^{-1}\ast\xi^{2}\ast w^{n-1}+g^{-1}\ast g\ast\tilde{\nabla}g
\ast\xi\ast\tilde{\nabla}\xi\ast w^{n-1}
+g^{-1}\ast g\ast\tilde{\nabla}g\ast\xi^{2}\nonumber\\
&& \ \ast w^{n-2}
\ast\left({\rm Rm}\ast\tilde{\nabla}{\rm Rm}+\nabla V\ast\tilde{\nabla}\nabla
V+\nabla^{2}\phi\ast\tilde{\nabla}\nabla^{2}\phi\right)
\bigg\rangle_{\tilde{g}}d\tilde{V}\label{3.122}\\
&\leq&C_{13}+C_{13}\int_{B_{\tilde{g}}(x_{0},r+1)}
w^{n-1}|\tilde{\nabla}^{2}g|_{\tilde{g}}\xi^{2}\!\ d\tilde{V}
+C_{13}\int_{B_{\tilde{g}}(x_{0},r+1)}
w^{n-2}\bigg(|{\rm Rm}|^{2}_{\tilde{g}}\nonumber\\
&&+ \ |\nabla V|^{2}_{\tilde{g}}+|\tilde{\nabla}{\rm Rm}|^{2}_{\tilde{g}}+|\tilde{\nabla}\nabla V|^{2}_{\tilde{g}}
+|\nabla^{2}\phi|^{2}_{\tilde{g}}+|\tilde{\nabla}\nabla^{2}
\phi|^{2}_{\tilde{g}}\bigg)d\tilde{V}\nonumber\\
&\leq&\frac{1}{8}\int_{B_{\tilde{g}}(x_{0},r+1)}
w^{n-1}|\tilde{\nabla}^{2}g|^{2}_{\tilde{g}}\xi^{2}\!\ d\tilde{V}
+C_{14}\int_{B_{\tilde{g}}(x_{0},r+1)}
w^{n-2}(u+v)\!\ d\tilde{V}+C_{14}\nonumber\\
&\leq&\frac{1}{8}\int_{B_{\tilde{g}}(x_{0},r+1)}
w^{n-1}|\tilde{\nabla}^{2}g|^{2}_{\tilde{g}}\xi^{2}\!\ d\tilde{V}
+C_{15}\int_{B_{\tilde{g}}(x_{0},r+1)}
u^{n-2}v\!\ d\tilde{V}+C_{15}\nonumber
\end{eqnarray}
by the inductive hypothesis, (\ref{3.71}) and the previous estimates. From
(\ref{3.71}), we also have
\begin{equation}
I_{3}\lesssim\int_{B_{\tilde{g}}(x_{0},r+1)}
w^{n-1}|\tilde{\nabla}^{2}\phi|_{\tilde{g}}\xi\!\ d\tilde{V}.\label{3.123}
\end{equation}
Since $C_{3}|\tilde{\nabla}^{2}g|_{\tilde{g}}\leq\frac{1}{8}|\tilde{\nabla}^{2}
g|^{2}_{\tilde{g}}+2C^{2}_{3}$, we infer from (\ref{3.121}), (\ref{3.122}),
and (\ref{3.123}) that
\begin{eqnarray}
I&\leq&-\frac{1}{2}\int_{B_{\tilde{g}}(x_{0},r+1)}
|\tilde{\nabla}^{2}g|^{2}_{\tilde{g}}w^{n-1}\xi^{2}\!\ d\tilde{V}
+C_{16}\int_{B_{\tilde{g}}(x_{0},r+1)}u^{n-2}v\!\ d\tilde{V}\nonumber\\
&&+ \ C_{16}\int_{B_{\tilde{g}}(x_{0},r+1)}
w^{n-1}|\tilde{\nabla}^{2}\phi|_{\tilde{g}}\xi\!\ d\tilde{V}+C_{16}\label{3.124}
\end{eqnarray}
by the inductive hypothesis. Note that the estimate (\ref{3.82}) is a special
case of (\ref{3.124}). According to (\ref{3.71}),
\begin{eqnarray*}
J&\leq&C_{17}\int_{B_{\tilde{g}}(x_{0},r+1)}
w^{n-2}\left[2\langle{\rm Rm},\partial_{t}{\rm Rm}\rangle_{\tilde{g}}
+2\langle\nabla^{2}\phi,\partial_{t}\nabla^{2}\phi\rangle_{\tilde{g}}\right.\\
&& \ +\left.2\langle\nabla V,\partial_{t}\nabla V\rangle_{\tilde{g}}\right]
\xi^{2}\!\ d\tilde{V} \ \ := \ \ C_{17}(J_{1}+J_{2}+J_{3}),
\end{eqnarray*}
where
\begin{eqnarray*}
J_{1}&:=&\int_{B_{\tilde{g}}(x_{0},r+1)}
w^{n-2}2\langle{\rm Rm},\partial_{t}{\rm Rm}\rangle_{\tilde{g}}\xi^{2}\!\
d\tilde{V}, \\
J_{2}&:=&\int_{B_{\tilde{g}}(x_{0},r+1)}
w^{n-2}2\langle\nabla V,\partial_{t}\nabla V\rangle_{\tilde{g}}\xi^{2}\!\ d\tilde{V},  \\
J_{3}&:=&\int_{B_{\tilde{g}}(x_{0},r+1)}
w^{n-2}2\langle\nabla^{2}\phi,\partial_{t}\nabla^{2}\phi\rangle_{\tilde{g}}
\xi^{2}\!\ d\tilde{V}.
\end{eqnarray*}
Substituting (\ref{3.72}) into $J_{1}$ we find that
\begin{eqnarray}
J_{1}&=&\int_{B_{\tilde{g}}(x_{0},r+1)}
w^{n-2}\xi^{2}2\bigg\langle{\rm Rm},\Delta{\rm Rm}+g^{\ast-2}
\ast{\rm Rm}^{\ast2}\nonumber\\
&&+ \ g^{-1}\ast V\ast\nabla{\rm Rm}+g^{-1}\ast{\rm Rm}\ast\nabla V
+(\nabla^{2}\phi)^{\ast2}\bigg\rangle_{\tilde{g}}d\tilde{V}.\label{3.125}
\end{eqnarray}
We now estimate each term in (\ref{3.125}). Since $\Delta{\rm Rm}=g^{\alpha\beta}
\nabla_{\alpha}\nabla_{\beta}{\rm Rm}$, the first term on the right-hand side of (\ref{3.125}) is bounded by
\begin{eqnarray*}
&&-2\int_{B_{\tilde{g}}(x_{0},r+1)}
\langle\nabla_{\beta}{\rm Rm},\nabla_{\alpha}
(\xi^{2}w^{n-2}g^{\alpha\beta}{\rm Rm})\rangle_{\tilde{g}}d\tilde{V}\\
&=&-2\int_{B_{\tilde{g}}(x_{0},r+1)}
\langle\nabla_{\beta}{\rm Rm},\xi^{2}w^{n-2}g^{\alpha\beta}\nabla_{\alpha}
{\rm Rm}\rangle_{\tilde{g}}d\tilde{V}\\
&&+ \ \int_{B_{\tilde{g}}(x_{0},r+1)}
\nabla{\rm Rm}\ast\left(g^{\ast-2}\ast\nabla g
\ast{\rm Rm}\ast\xi^{2}+g^{-1}\ast{\rm Rm}\ast\xi\ast\nabla\xi\right)
w^{n-2}\!\ d\tilde{V}\\
&&+ \ \int_{B_{\tilde{g}}(x_{0},r+1)}
\nabla{\rm Rm}\ast g^{-1}\ast{\rm Rm}
\ast\xi^{2}\ast w^{n-3}\ast\bigg({\rm Rm}\ast\tilde{\nabla}{\rm Rm}
+\nabla V\ast\nabla^{2}V\\
&&+ \ \nabla^{2}\phi\ast\nabla^{3}\phi\bigg)d\tilde{V}\\
&\leq&-\int_{B_{\tilde{g}}(x_{0},r+1)}
|\nabla{\rm Rm}|^{2}_{\tilde{g}}\xi^{2}w^{n-2}\!\ d\tilde{V}
+C_{18}\int_{B_{\tilde{g}}(x_{0},r+1)}|\nabla{\rm Rm}|_{\tilde{g}}
|{\rm Rm}|_{\tilde{g}}\xi w^{n-2}\!\ d\tilde{V}\\
&&+ \ C_{18}\int_{B_{\tilde{g}}(x_{0},r+1)}
|\nabla{\rm Rm}|_{\tilde{g}}|{\rm Rm}|_{\tilde{g}}
\xi^{2}w^{n-3}\bigg(|{\rm Rm}|_{\tilde{g}}
|\nabla{\rm Rm}|_{\tilde{g}}
+|\nabla V|_{\tilde{g}}|\nabla^{2}V|_{\tilde{g}}\\
&&+ \ |\nabla^{2}\phi|_{\tilde{g}}|\nabla^{3}\phi|_{\tilde{g}}\bigg)d\tilde{V}
\end{eqnarray*}
since $\nabla g=\tilde{\nabla}g+g^{-1}\ast\tilde{\nabla}g\ast g\lesssim1$. By
the inductive hypothesis, we have
\begin{eqnarray*}
&&\int_{B_{\tilde{g}}(x_{0},r+1)}
|\nabla{\rm Rm}|_{\tilde{g}}|{\rm Rm}|_{\tilde{g}}\xi w^{n-2}\!\ d\tilde{V}\\
&\lesssim&\int_{B_{\tilde{g}}(x_{0},r+1)}
w^{n-2}|\nabla{\rm Rm}|^{2}_{\tilde{g}}d\tilde{V}
+\int_{B_{\tilde{g}}(x_{0},r+1)}
w^{n-2}|{\rm Rm}|^{2}_{\tilde{g}}d\tilde{V}\\
&\lesssim&\int_{B_{\tilde{g}}(x_{0},r+1)}
u^{n-2}v\!\ d\tilde{V}+\int_{B_{\tilde{g}}(x_{0},r+1)}
u^{n-1}\!\ d\tilde{V}\\
&\lesssim&1+\int_{B_{\tilde{g}}(x_{0},r+1)}
u^{n-2}v\!\ dV,
\end{eqnarray*}
and
\begin{eqnarray*}
&&\int_{B_{\tilde{g}}(x_{0},r+1)}
w^{n-3}|\nabla{\rm Rm}|_{\tilde{g}}|{\rm Rm}|_{\tilde{g}}
\bigg(|{\rm Rm}|_{\tilde{g}}|\nabla{\rm Rm}|_{\tilde{g}}
+|\nabla V|_{\tilde{g}}|\nabla^{2}V|_{\tilde{g}}\\
&&+ \ |\nabla^{2}\phi|_{\tilde{g}}|\nabla^{3}\phi|_{\tilde{g}}\bigg)d\tilde{V}\\
&\lesssim&\int_{B_{\tilde{g}}(x_{0},r+1)}
w^{n-3}uv\!\ d\tilde{V}+\int_{B_{\tilde{g}}(x_{0},r+1)}w^{n-3}u^{1/2}v^{1/2}
\left(u^{1/2}v^{1/2}+u^{1/2}v^{1/2}\right)d\tilde{V}\\
&\lesssim&\int_{B_{\tilde{g}}(x_{0},r+1)}
u^{n-2}v\!\ dV;
\end{eqnarray*}
hence the first term on the right-hand side of (\ref{3.125}) is bounded from
above by
\begin{equation*}
1+\int_{B_{\tilde{g}}(x_{0},r+1)}u^{n-2}v\!\ dV
\end{equation*}
up to a uniform positive multiple. Since $|{\rm Rm}|^{2}_{g}\lesssim1+|\nabla
\tilde{\nabla}g|^{2}_{\tilde{g}}$ by the equation (90) in page 277 of
\cite{S89}, it follows that the sum of the third and forth terms of the right-hand side of (\ref{3.125})
is bounded from above by
\begin{eqnarray*}
&&\int_{B_{\tilde{g}}(x_{0},r+1)}
w^{n-2}\xi^{2}\ast2{\rm Rm}\ast\left(g^{-1}
\ast V\ast\nabla{\rm Rm}+g^{-1}\ast{\rm Rm}\ast\nabla V\right)d\tilde{V}\\
&\leq&C_{19}\int_{B_{\tilde{g}}(x_{0},r+1)}
w^{n-2}\xi^{2}\left(|\nabla{\rm Rm}|_{\tilde{g}}|{\rm Rm}|_{\tilde{g}}
+|{\rm Rm}|^{2}_{\tilde{g}}|\nabla V|_{\tilde{g}}\right)d\tilde{V}\\
&\leq&\frac{1}{2}C_{19}\int_{B_{\tilde{g}}(x_{0},r+1)}
w^{n-2}\xi^{2}\left(|\nabla{\rm Rm}|^{2}_{\tilde{g}}+|{\rm Rm}|^{2}_{\tilde{g}}
\right)d\tilde{V}\\
&&+ \ \frac{1}{2}C_{19}
\int_{B_{\tilde{g}}(x_{0},r+1)}w^{n-1}\xi^{2}|{\rm Rm}|_{\tilde{g}}d\tilde{V}\\
&\leq&\frac{1}{2}C_{19}\int_{B_{\tilde{g}}(x_{0},r+1)}
w^{n-2}v\!\ d\tilde{V}+\frac{1}{2}C_{19}\int_{B_{\tilde{g}}(x_{0},r+1)}
w^{n-1}\!\ d\tilde{V}\\
&&+ \ C_{20}\int_{B_{\tilde{g}}(x_{0},r+1)}
w^{n-1}\xi^{2}\left(1+|\nabla\tilde{\nabla}g|_{\tilde{g}}\right)\xi^{2}\!\ d\tilde{V}\\
&\leq&\frac{1}{8C_{17}}\int_{B_{\tilde{g}}(x_{0},r+1)}
w^{n-1}\xi^{2}|\tilde{\nabla}^{2}g|^{2}_{\tilde{g}}\!\ d\tilde{V}
+C_{21}\int_{B_{\tilde{g}}(x_{0},r+1)}u^{n-2}v\!\ dV
+C_{21},
\end{eqnarray*}
because of the inductive hypothesis and
$V=g^{-1}\ast\tilde{\nabla}g\lesssim1$. Similarly, the second term on
the right-hand side of (\ref{3.125}) is bounded from above by (up to a
uniform positive multiple)
\begin{eqnarray*}
&&\int_{B_{\tilde{g}}(x_{0},r+1)}w^{n-2}\xi^{2}
\ast g^{\ast-2}\ast{\rm Rm}^{\ast3}\!\ d\tilde{V}\\
&\lesssim&\int_{B_{\tilde{g}}(x_{0},r+1)}
w^{n-2}\xi^{2}\ast g^{\ast-2}\ast{\rm Rm}^{\ast2}\ast(1+\nabla\tilde{\nabla}
g)\!\ d\tilde{V}\\
&\lesssim&\int_{B_{\tilde{g}}(x_{0},r+1)}
w^{n-1}\!\ d\tilde{V}+
\int_{B_{\tilde{g}}(x_{0},r+1)}w^{n-2}\xi^{2}\ast{\rm Rm}^{\ast2}
\ast g^{\ast-2}\ast\nabla\tilde{\nabla}g\!\ d\tilde{V}\\
&\lesssim&1+
\int_{B_{\tilde{g}}(x_{0},r+1)}
\tilde{\nabla}g\ast\left(\nabla\xi\ast\xi\ast{\rm Rm}^{\ast2}
+\xi^{2}\ast\nabla{\rm Rm}\ast{\rm Rm}\right)\ast g^{\ast-2}\ast w^{n-2}\!\ d
\tilde{V}\\
&&+ \ \int_{B_{\tilde{g}}(x_{0},r+1)}{\rm Rm}^{\ast2}\ast w^{n-3}\ast\left(
{\rm Rm}\ast\nabla{\rm Rm}+\nabla V\ast\nabla^{2}V
+\nabla^{2}\phi\ast\nabla^{3}\phi\right)d\tilde{V}\\
&\lesssim&1+\int_{B_{\tilde{g}}(x_{0},r+1)}
(u+v)u^{n-2}\!\ dV+\int_{B_{\tilde{g}}(x_{0},r+1)}
uu^{n-3}(u+v)\!\ dV\\
&\lesssim&1+\int_{B_{\tilde{g}}(x_{0},r+1)}
u^{n-2}v\!\ dV,
\end{eqnarray*}
by the equation (84) in page 276 of \cite{S89}, $g$ is equivalent to $\tilde{g}$,
and the inductive hypothesis. The last term on the right-hand side of (\ref{3.125})
is bounded from above by (up to a uniform positive multiple)
\begin{eqnarray*}
&&\int_{B_{\tilde{g}}(x_{0},r+1)}w^{n-2}\xi^{2}\ast{\rm Rm}\ast\nabla^{2}\phi
\ast\nabla^{2}\phi\!\ d\tilde{V}\\
&\lesssim&\int_{B_{\tilde{g}}(x_{0},r+1)}\nabla\phi\ast\nabla\left(
w^{n-2}\xi^{2}\ast{\rm Rm}\ast\nabla^{2}\phi\right)d\tilde{V}\\
&=&\int_{B_{\tilde{g}}(x_{0},r+1)}\nabla\phi\ast\bigg(
\xi\ast\nabla\xi\ast{\rm Rm}\ast\nabla^{2}\phi
+\xi^{2}\ast\nabla{\rm Rm}\ast\nabla^{2}\phi\\
&&+ \ \xi^{2}\ast{\rm Rm}
\ast\nabla^{3}\phi\bigg)w^{n-2}\!\ d\tilde{V}+\int_{B_{\tilde{g}}(x_{0},r+1)}
\nabla\phi\ast\xi^{2}\ast{\rm Rm}\ast\nabla^{2}\phi
\ast w^{n-3}\\
&& \ \ast\left({\rm Rm}\ast\nabla{\rm Rm}
+\nabla V\ast\nabla^{2}V+\nabla^{2}\phi\ast\nabla^{3}\phi\right)d\tilde{V}\\
&\lesssim&\int_{B_{\tilde{g}}(x_{0},r+1)}
w^{n-2}\left(w+v^{1/2}w^{1/2}\right)d\tilde{V}
+\int_{B_{\tilde{g}}(x_{0},r+1)}ww^{n-3}(w+v)d\tilde{V}\\
&\lesssim&\int_{B_{\tilde{g}}(x_{0},r+1)}
u^{n-1}\!\ dV+\int_{B_{\tilde{g}}(x_{0},r+1)}u^{n-2}v\!\ dV \ \ \lesssim \ \ 1+\int_{B_{\tilde{g}}(x_{0},r+1)}
u^{n-2}v\!\ dV.
\end{eqnarray*}
Note that when we do the integration by parts, we may replace $\tilde{g}$ by
$g$ since $g$ is equivalent $\tilde{g}$, so that we have no extra terms $\nabla\tilde{g}$ and $\nabla\tilde{g}^{-1}$. Therefore, substituting those estimates into (\ref{3.125}) implies
\begin{equation}
C_{17}J_{1}\leq\frac{1}{8}\int_{B_{\tilde{g}}(x_{0},r+1)}
w^{n-1}\xi^{2}|\tilde{\nabla}^{2}g|^{2}_{\tilde{g}}
d\tilde{V}+C_{22}\int_{B_{\tilde{g}}(x_{0},r+1)}u^{n-2}v\!\ dV+C_{22}.\label{3.126}
\end{equation}
By (\ref{3.125}), we have
\begin{eqnarray}
J_{2}&=&\int_{B_{\tilde{g}}(x_{0},r+1)}w^{n-2}\xi^{2}2\bigg\langle
\nabla V,\Delta\nabla V+g^{\ast-3}\ast\nabla\tilde{\nabla}g\ast{\rm Rm}
+g^{\ast-3}\ast\tilde{\nabla}g\ast\nabla{\rm Rm}\nonumber\\
&&+ \ g^{\ast-2}\ast\nabla\tilde{\nabla}g\ast\nabla V
+g^{\ast-2}\ast\tilde{\nabla}g\ast\nabla^{2}V\label{3.127}\\
&&+ \ g^{\ast-2}\ast\nabla\tilde{\nabla}g\ast(\nabla\phi)^{\ast2}
+g^{\ast-2}\ast\tilde{\nabla}g\ast\nabla\phi\ast\nabla^{2}\phi\bigg\rangle_{\tilde{g}}
d\tilde{V}.\nonumber
\end{eqnarray}
As before, the first term on the right-hand side of (\ref{3.127}) is bounded from
above by (up to a uniform positive multiple)
\begin{eqnarray*}
&&2\int_{B_{\tilde{g}}(x_{0},r+1)}
w^{n-2}\xi^{2}\langle\nabla V,\Delta\nabla V\rangle_{\tilde{g}}d\tilde{V}\\
&=&-2\int_{B_{\tilde{g}}(x_{0},r+1)}
\langle\nabla_{\beta}\nabla V,\nabla_{\alpha}
(g^{\alpha\beta}\xi^{2}w^{n-2}\nabla V)\rangle_{\tilde{g}}d\tilde{V}\\
&\leq&-\int_{B_{\tilde{g}}(x_{0},r+1)}|\nabla^{2}V|^{2}_{\tilde{g}}
\xi^{2}w^{n-2}\!\ d\tilde{V}\\
&&+ \ \int_{B_{\tilde{g}}(x_{0},r+1)}
\nabla^{2}V\ast\nabla V\ast g^{-1}\bigg[\xi\ast\nabla\xi\ast w^{n-2}
+\xi\ast w^{n-3}\ast({\rm Rm}\ast\nabla{\rm Rm}\\
&&+ \ \nabla V\ast\nabla^{2}V+\nabla^{2}\phi\ast\nabla^{3}\phi)\bigg]d\tilde{V}\\
&\lesssim&\int_{B_{\tilde{g}}(x_{0},r+1)}
w^{n-2}v^{1/2}w^{1/2}\!\ d\tilde{V}+\int_{B_{\tilde{g}}(x_{0},r+1)}
v^{1/2}w^{1/2}w^{n-3}v^{1/2}w^{1/2}\!\ d\tilde{V}\\
&\lesssim&\int_{B_{\tilde{g}}(x_{0},r+1)}u^{n-1}\!\ dV
+\int_{B_{\tilde{g}}(x_{0},r+1)}
u^{n-2}v\!\ dV \ \ \lesssim \ \ 1+\int_{B_{\tilde{g}}(x_{0},r+1)}u^{n-2}v\!\ dV
\end{eqnarray*}
by the inductive hypothesis. By the Cauchy-Schwarz inequality, the sum of the
second, third, forth, and fifth terms on the right-hand side of (\ref{3.127})
is bounded from above by
\begin{eqnarray*}
&&C_{23}\int_{B_{\tilde{g}}(x_{0},r+1)}
w^{n-2}|\nabla V|_{\tilde{g}}\left[w^{1/2}|\tilde{\nabla}^{2}g|_{\tilde{g}}
+|\nabla{\rm Rm}|_{\tilde{g}}+|\nabla^{2}V|_{\tilde{g}}\right]\xi^{2}\!\ d\tilde{V}\\
&\leq&\frac{1}{8C_{17}}\int_{B_{\tilde{g}}(x_{0},r+1)}
w^{n-1}\xi^{2}|\tilde{\nabla}^{2}g|^{2}_{\tilde{g}}\!\ d\tilde{V}
+C_{24}\int_{B_{\tilde{g}}(x_{0},r+1)}
w^{n-2}v\!\ d\tilde{V}+C_{24}\\
&\leq&\frac{1}{8C_{17}}\int_{B_{\tilde{g}}(x_{0},r+1)}
w^{n-1}\xi^{2}|\tilde{\nabla}^{2}g|^{2}_{\tilde{g}}\!\ d\tilde{V}
+C_{25}\int_{B_{\tilde{g}}(x_{0},r+1)}u^{n-2}v\!\ dV+C_{25}.
\end{eqnarray*}
The rest terms on the right-hand side of (\ref{3.127}) is bounded from
above by (up to a uniform positive multiple)
\begin{eqnarray*}
&&\int_{B_{\tilde{g}}(x_{0},r+1)}
w^{n-2}\xi^{2}\ast\nabla V\ast g^{\ast-2}\ast\left(\nabla\tilde{\nabla}g
\ast(\nabla\phi)^{\ast2}+\tilde{\nabla}g\ast\nabla\phi\ast\nabla^{2}\phi\right)
d\tilde{V}\\
&\lesssim&\int_{B_{\tilde{g}}(x_{0},r+1)}
w^{n-2}w^{1/2}|\nabla\tilde{\nabla}g|_{\tilde{g}}\!\ d\tilde{V}
+\int_{B_{\tilde{g}}(x_{0},r+1)}w^{n-2}w^{1/2}w^{1/2}\!\ d\tilde{V}\\
&\lesssim&\int_{B_{\tilde{g}}(x_{0},r+1)}
u^{n-1}\!\ dV+\int_{B_{\tilde{g}}(x_{0},r+1)}
u^{n-2}|\nabla\tilde{\nabla}g|^{2}_{\tilde{g}}\!\ d\tilde{V} \ \ \lesssim \ \ 1
\end{eqnarray*}
by the inductive hypothesis. Hence
\begin{equation}
C_{17}J_{2}\leq\frac{1}{8}\int_{B_{\tilde{g}}(x_{0},r+1)}
w^{n-1}\xi^{2}|\tilde{\nabla}^{2}g|^{2}_{\tilde{g}}\!\ d\tilde{V}
+C_{26}\int_{B_{\tilde{g}}(x_{0},r+1)}
u^{n-2}v\!\ dV+C_{26}.\label{3.128}
\end{equation}
According to (\ref{3.75}),
\begin{eqnarray}
J_{3}&=&\int_{B_{\tilde{g}}(x_{0},r+1)}
w^{n-2}\xi^{2}2\bigg\langle\nabla^{2}\phi,\Delta\nabla^{2}\phi
+g^{\ast-2}\ast{\rm Rm}\ast\nabla^{2}\phi+\beta_{2}\nabla^{2}\phi\nonumber\\
&&+ \ g^{-1}\ast(\nabla\phi)^{\ast2}\ast\nabla^{2}\phi
+g^{-1}\ast\nabla\phi\ast\nabla^{2}\phi+g^{-1}\ast(\nabla^{2}\phi)^{\ast2}\label{3.129}\\
&&+ \ g^{\ast-2}\ast{\rm Rm}\ast(\nabla\phi)^{\ast2}
+g^{-1}\ast V\ast\nabla^{3}\phi+g^{-1}\ast\nabla V\ast\nabla^{2}\phi\bigg\rangle_{\tilde{g}}
d\tilde{V}.\nonumber
\end{eqnarray}
By the integration by parts, the first term on the right-hand side
of (\ref{3.129}) equals
\begin{eqnarray*}
&&-2\int_{B_{\tilde{g}}(x_{0},r+1)}
\langle\nabla_{\beta}\nabla^{2}\phi,\nabla_{\alpha}(\xi^{2} w^{n-2}g^{\alpha\beta}\nabla^{2}
\phi )\rangle_{\tilde{g}}d\tilde{V}\\
&=&-2\int_{B_{\tilde{g}}(x_{0},r+1)}
\langle\nabla_{\beta}\nabla^{2}\phi,\xi^{2} g^{\alpha\beta}w^{n-2}
\nabla_{\alpha}\nabla^{2}\phi\rangle_{\tilde{g}}d\tilde{V}+
\int_{B_{\tilde{g}}(x_{0},r+1)}
\xi^{2}\nabla^{3}\phi\ast
\nabla^{2}\phi\\
&& \ \ast g^{-1}\ast w^{n-3}\ast\left({\rm Rm}\ast\nabla{\rm Rm}
+\nabla V\ast\nabla^{2}V+\nabla^{2}\phi\ast\nabla^{3}\phi\right)d\tilde{V}\\
&&+ \ \int_{B_{\tilde{g}}(x_{0},r+1)}
\nabla^{3}\phi\ast\nabla^{2}\phi\ast g^{-1}\ast w^{n-2}\ast\nabla\xi\ast\xi \!\
d\tilde{V}\\
&\lesssim&-\int_{B_{\tilde{g}}(x_{0},r+1)}
\xi^{2}w^{n-2}|\nabla^{3}\phi|^{2}_{\tilde{g}}\!\ d\tilde{V}
+\int_{B_{\tilde{g}}(x_{0},r+1)}
v^{1/2}w^{1/2}w^{n-3}w^{1/2}v^{1/2}\!\ d\tilde{V}\\
&&+ \ \int_{B_{\tilde{g}}(x_{0},r+1)}
v^{1/2}w^{1/2}w^{n-2}\!\ d\tilde{V} \ \
\lesssim \ \ 1+\int_{B_{\tilde{g}}(x_{0},r+1)}
w^{n-2}v\!\ d\tilde{V}.
\end{eqnarray*}
The rest terms on the right-hand side of (\ref{3.129}) are bounded from above
by (up to a uniform positive multiple)
\begin{eqnarray*}
\int_{B_{\tilde{g}}(x_{0},r+1)}
w^{n-2}(w+w^{1/2}+v^{1/2})\!\ d\tilde{V}
&\lesssim&\int_{B_{\tilde{g}}(x_{0},r+1)}
w^{n-2}(1+w+v)\!\ d\tilde{V}\\
&\lesssim&1+\int_{B_{\tilde{g}}(x_{0},r+1)}
u^{n-2}v\!\ dV
\end{eqnarray*}
by the inductive hypothesis. Hence
\begin{equation}
J_{3}\lesssim1+\int_{B_{\tilde{g}}(x_{0},r+1)}
u^{n-2}v\!\ dV.\label{3.130}
\end{equation}
Combining (\ref{3.126}), (\ref{3.128}), and (\ref{3.130}), we find that
\begin{equation}
J\leq\frac{1}{4}\int_{B_{\tilde{g}}(x_{0},r+1)}
w^{n-1}\xi^{2}|\tilde{\nabla}^{2}g|^{2}_{\tilde{g}}\!\ d\tilde{V}+C_{27}\int_{B_{\tilde{g}}(x_{0},r+1)}
u^{n-2}v\!\ dV+C_{27}.\label{3.131}
\end{equation}
Substituting (\ref{3.124}) and (\ref{3.131}) into the definition of $K$ yields
\begin{eqnarray}
\frac{d}{dt}\int_{B_{\tilde{g}}(x_{0},r+1)}
w^{n-1}|\tilde{\nabla}g|^{2}_{\tilde{g}}\xi^{2}\!\ d\tilde{V}
&\leq&-\frac{1}{4}\int_{B_{\tilde{g}}(x_{0},r+1)}
w^{n-1}\xi^{2}|\tilde{\nabla}^{2}g|^{2}_{\tilde{g}}\!\ d\tilde{V}+C_{28}\nonumber\\
&&+ \ C_{28}\int_{B_{\tilde{g}}(x_{0},r+1)}
u^{n-2}v\!\ dV\label{3.132}\\
&&+ \ C_{28}\int_{B_{\tilde{g}}(x_{0},r+1)}
w^{n-1}|\tilde{\nabla}^{2}\phi|_{\tilde{g}}\xi\!\ d\tilde{V}.\nonumber
\end{eqnarray}

Define
\begin{eqnarray*}
L&:=&\frac{d}{dt}\int_{B_{\tilde{g}}(x_{0},r+1)}w^{n-1}|\tilde{\nabla}
\phi|^{2}_{\tilde{g}}\xi^{2}\!\ d\tilde{V} \ \ = \ \ 2\int_{B_{\tilde{g}}(x_{0},
r+1)}w^{n-1}\xi^{2}\langle\tilde{\nabla}\phi,\partial_{t}\tilde{\nabla}
\phi\rangle_{\tilde{g}}d\tilde{V}\\
&=&2\int_{B_{\tilde{g}}(x_{0},r+1)}
w^{n-1}\xi^{2}\langle\tilde{\nabla}\phi,g^{\alpha\beta}\tilde{\nabla}_{\alpha}
\tilde{\nabla}_{\beta}\tilde{\nabla}\phi\rangle_{\tilde{g}}d\tilde{V}\\
&&+ \ \int_{B_{\tilde{g}}(x_{0},r+1)}
w^{n-1}\xi^{2}\tilde{\nabla}\phi\ast\bigg(g^{\ast-2}
\ast\widetilde{{\rm Rm}}\ast\tilde{\nabla}\phi+g^{\ast-2}
\ast\tilde{\nabla}g\ast\tilde{\nabla}^{2}\phi\\
&&+ \ g^{\ast-2}\ast\tilde{\nabla}g\ast(\tilde{\nabla}\phi)^{\ast2}
+g^{-1}\ast\tilde{\nabla}\phi\ast\tilde{\nabla}^{2}\phi
+\tilde{\nabla}\phi\bigg)d\tilde{V} \ \ := \ \ L_{1}+L_{2}.
\end{eqnarray*}
For $L_{2}$, we have
\begin{equation*}
L_{2}\lesssim\int_{B_{\tilde{g}}(x_{0},r+1)}
w^{n-1}\xi(1+|\tilde{\nabla}^{2}\phi|_{\tilde{g}})d\tilde{V}
\lesssim1+\int_{B_{\tilde{g}}(x_{0},r+1)}
w^{n-1}\xi|\tilde{\nabla}^{2}\phi|_{\tilde{g}}d\tilde{V}
\end{equation*}
by the inductive hypothesis. Taking the integration by parts on $L_{1}$ implies
\begin{eqnarray*}
L_{1}&=&-2\int_{B_{\tilde{g}}(x_{0},r+1)}
\langle\tilde{\nabla}_{\beta}\tilde{\nabla}\phi,\tilde{\nabla}_{\alpha}(
g^{\alpha\beta}w^{n-1}\xi^{2}\tilde{\nabla}\phi)\rangle_{\tilde{g}}d\tilde{V}\\
&=&-2\int_{B_{\tilde{g}}(x_{0},r+1)}
\langle\tilde{\nabla}_{\beta}\tilde{\nabla}\phi,g^{\alpha\beta}
w^{n-1}\xi^{2}\tilde{\nabla}_{\alpha}\tilde{\nabla}\phi\rangle_{\tilde{g}}d\tilde{V}\\
&&+ \ \int_{B_{\tilde{g}}(x_{0},r+1)}
\tilde{\nabla}^{2}\phi\ast\tilde{\nabla}\phi\ast\bigg(g^{\ast-2}
\ast\tilde{\nabla}g
\ast w^{n-1}\ast\xi^{2}+g^{-1}\ast\xi\ast\tilde{\nabla}\xi\ast w^{n-1}\\
&&+ \ g^{-1}\ast\xi^{2}\ast w^{n-2}\ast({\rm Rm}\ast\tilde{\nabla}{\rm Rm}
+\nabla V\ast\tilde{\nabla}\nabla V+\nabla^{2}\phi\ast\tilde{\nabla}
\nabla^{2}\phi)\bigg)d\tilde{V}\\
&\leq&-\int_{B_{\tilde{g}}(x_{0},r+1)}
w^{n-1}\xi^{2}|\tilde{\nabla}^{2}\phi|^{2}_{\tilde{g}}d\tilde{V}
+C_{29}\int_{B_{\tilde{g}}(x_{0},r+1)}w^{n-1}\xi|\tilde{\nabla}^{2}
\phi|_{\tilde{g}}\!\ d\tilde{V}\\
&&+ \ C_{29}\int_{B_{\tilde{g}}(x_{0},r+1)}
w^{n-2}\xi^{2}|\tilde{\nabla}^{2}\phi|_{\tilde{g}}w^{1/2}\left(w^{1/2}
+v^{1/2}\right)d\tilde{V}\\
&=&-\int_{B_{\tilde{g}}(x_{0},r+1)}
w^{n-1}\xi^{2}|\tilde{\nabla}^{2}\phi|^{2}_{\tilde{g}}\!\ d\tilde{V}
+2C_{29}\int_{B_{\tilde{g}}(x_{0},r+1)}
\xi w^{n-1}|\tilde{\nabla}^{2}\phi|_{\tilde{g}}\!\ d\tilde{V}\\
&&+ \ C_{29}\int_{B_{\tilde{g}}(x_{0},r+1)}
\left(\xi w^{\frac{n-1}{2}}|\tilde{\nabla}^{2}\phi|_{\tilde{g}}\right)
\left(\xi w^{\frac{n-2}{2}}v^{1/2}\right)\!\ d\tilde{V}\\
&\leq&-\frac{3}{4}\int_{B_{\tilde{g}}(x_{0},r+1)}
w^{n-1}\xi^{2}|\tilde{\nabla}^{2}\phi|^{2}_{\tilde{g}}\!\ d\tilde{V}
+C_{30}\int_{B_{\tilde{g}}(x_{0},r+1)}
w^{n-1}\xi|\tilde{\nabla}^{2}\phi|_{\tilde{g}}\!\ d\tilde{V}\\
&&+ \ C_{30}\int_{B_{\tilde{g}}(x_{0},r+1)}
u^{n-2}v\!\ dV,
\end{eqnarray*}
since
\begin{eqnarray*}
\tilde{\nabla}{\rm Rm}&=&\nabla{\rm Rm}+g^{-1}\ast\tilde{\nabla}g
\ast{\rm Rm} \ \ \lesssim \ \ v^{1/2}+w^{1/2},\\
\tilde{\nabla}\nabla V&=&\nabla^{2}V+g^{-1}\ast\tilde{\nabla}g\ast\nabla
V \ \ \lesssim \ \ v^{1/2}+w^{1/2},\\
\tilde{\nabla}\nabla^{2}\phi&=&\nabla^{3}\phi+g^{-1}\ast\tilde{\nabla}g\ast
\nabla^{2}\phi \ \ \lesssim \ \ v^{1/2}+w^{1/2}.
\end{eqnarray*}
Therefore
\begin{equation}
L\leq-\frac{1}{2}\int_{B_{\tilde{g}}(x_{0},r+1)}
w^{n-1}\xi^{2}|\tilde{\nabla}^{2}\phi|^{2}_{\tilde{g}}\!\ d\tilde{V}+C_{31}\int_{B_{\tilde{g}}(x_{0},r+1)}
u^{n-2}v\!\ dV+C_{31}.\label{3.133}
\end{equation}
From (\ref{3.132}) and (\ref{3.133}), we arrive at
\begin{eqnarray*}
&&\frac{d}{dt}\bigg(\int_{B_{\tilde{g}}(x_{0},r+1)}
w^{n-1}\xi^{2}\left(|\tilde{\nabla}g|^{2}_{\tilde{g}}
+|\tilde{\nabla}\phi|^{2}_{\tilde{g}}\right)d\tilde{V}\bigg)\\
&\leq&-\frac{1}{4}\int_{B_{\tilde{g}}(x_{0},r+1)}
w^{n-1}\xi^{2}\left(|\tilde{\nabla}^{2}g|^{2}_{\tilde{g}}
+|\tilde{\nabla}^{2}\phi|^{2}_{\tilde{g}}\right)d\tilde{V}\\
&&+ \ C_{32}\int_{B_{\tilde{g}}(x_{0},r+1)}u^{n-2}v\!\ dV+C_{32}.
\end{eqnarray*}
Consequently,
\begin{equation*}
\int^{T}_{0}\int_{B_{\tilde{g}}(x_{0},r+1)}
w^{n-1}\xi^{2}|\tilde{\nabla}g|^{2}_{\tilde{g}}\!\ d\tilde{V}dt\lesssim1;
\end{equation*}
in particular,
\begin{equation}
\int^{T}_{0}\int_{B_{\tilde{g}}(x_{0},r+1)}
u^{n-1}|\nabla\tilde{\nabla}g|^{2}_{g}\xi^{2}\!\ dVdt\lesssim1.\label{3.134}
\end{equation}
By the equations (90) in page 277 and (108) in page 281 of \cite{S89},
together with $\nabla^{2}\phi=\tilde{\nabla}^{2}+g^{-1}\ast\tilde{\nabla}g
\ast\tilde{\nabla}\phi$, we have
\begin{equation*}
|{\rm Rm}|^{2}_{g}+|\nabla V|^{2}_{g}+|\nabla^{2}\phi|^{2}_{g}
\lesssim1+|\nabla\tilde{\nabla}g|^{2}_{g};
\end{equation*}
using the estimate (\ref{3.134}), we obtain
\begin{equation*}
\int^{T}_{0}\int_{B_{\tilde{g}}(x_{0},r+1)}
u^{n}\!\ dVdt\lesssim1.
\end{equation*}
As in the proof of Lemma \ref{l3.12}, we can show that
\begin{equation*}
\int_{B_{\tilde{g}}(x_{0},r+1)}
u^{n}\xi^{2}\!\ dV+\frac{1}{8}\int^{t}_{0}
\int_{B_{\tilde{g}}(x_{0},r+1)}
u^{n-1}v\xi^{2}\!\ dVdt\lesssim1
\end{equation*}
by the previous estimates and inductive hypothesis. Thus the lemma is also
true for $s=n$.
\end{proof}

We now can prove the following theorem, as in \cite{S89} where use the equations
(\ref{3.118}), (\ref{3.119}), (\ref{3.120}), and Lemma \ref{l3.13}.

\begin{theorem}\label{t3.14} We have
\begin{equation}
\sup_{M\times[0,T]}|{\rm Rm}|^{2}_{g}\lesssim1, \ \ \sup_{M\times[0,T]}
|\nabla V|^{2}_{g}\lesssim1, \ \ \sup_{M\times[0,T]}|\nabla^{2}\phi|^{2}_{g}
\lesssim1,\label{3.135}
\end{equation}
where $\lesssim$ depend on $n, k_{0}, k_{1}, k_{2}, \alpha_{1},\beta_{1},
\beta_{2}$.
\end{theorem}

By the same argument used in \cite{L05, S89}, we have

\begin{theorem}\label{t3.15} Let $(M,\tilde{g})$ be a complete noncompact
Riemannian $m$-manifold with bounded Riemann curvature $|{\rm Rm}_{\tilde{g}}|^{2}_{\tilde{g}}
\leq k_{0}$, and $\tilde{\phi}$ a smooth function on $M$ satisfying
\begin{equation*}
|\tilde{\phi}|^{2}+|\nabla_{\tilde{g}}\tilde{\phi}|^{2}_{\tilde{g}}
\leq k_{1}, \ \ \ |\nabla^{2}_{\tilde{g}}\tilde{\phi}|^{2}_{\tilde{g}}
\leq k_{2}.
\end{equation*}
Then there exists a positive constant $T=T(m,k_{0},k_{1},\alpha_{1},\beta_{1},
\beta_{2})>0$ such that the regular-$(\alpha_{1},0,\beta_{1},\beta_{2})$-flow
\begin{eqnarray*}
\partial_{t}\hat{g}(t)&=&-2{\rm Ric}_{\hat{g}(t)}+2\alpha_{1}\nabla_{\hat{g}(t)}
\hat{\phi}(t)\otimes\nabla_{\hat{g}(t)}\hat{\phi}(t),\\
\partial_{t}\hat{\phi}(t)&=&\Delta_{\hat{g}(t)}
\hat{\phi}(t)+\beta_{1}|\nabla_{\hat{g}(t)}\hat{\phi}(t)|^{2}_{\hat{g}(t)}
+\beta_{2}\hat{\phi}(t),\\
(\hat{g}(0),\hat{\phi}(0))&=&(\tilde{g},\tilde{\phi})
\end{eqnarray*}
has a smooth solution $(\hat{g}(t),\hat{\phi}(t))$
on $M\times[0,T]$ and satisfies the following estimate
\begin{equation*}
\frac{1}{C_{1}}\tilde{g}\leq\hat{g}(t)\leq C_{1}\tilde{g}, \ \ \
|{\rm Rm}_{\hat{g}(t)}|^{2}_{\hat{g}(t)}+|\hat{\phi}(t)|^{2}
+|\nabla_{\hat{g}(t)}
\hat{\phi}(t)|^{2}_{\hat{g}(t)}+|\nabla^{2}_{\hat{g}(t)}\hat{\phi}(t)
|^{2}_{\hat{g}(t)}
\leq C_{2}
\end{equation*}
on $M\times[0,T]$, where $C_{1}, C_{2}$ are uniform positive constants
depending only on $m, k_{0}, k_{1}, k_{2}, \alpha_{1},\beta_{1},\beta_{2}$.
\end{theorem}

Suppose that $(\hat{g}(t),\hat{\phi}(t))$ is a smooth solution to the
regular-$(\alpha_{1},0,\beta_{1}-\alpha_{2},\beta_{2})$-flow
\begin{eqnarray*}
\partial_{t}\hat{g}(t)&=&-2{\rm Ric}_{\hat{g}(t)}+2\alpha_{1}\nabla_{\hat{g}(t)}
\hat{\phi}(t)\otimes\nabla_{\hat{g}(t)}\hat{\phi}(t),\\
\partial_{t}\hat{\phi}(t)&=&\Delta_{\hat{g}(t)}
\hat{\phi}(t)+(\beta_{1}-\alpha_{2})|\nabla_{\hat{g}(t)}\hat{\phi}(t)|^{2}_{\hat{g}(t)}
+\beta_{2}\hat{\phi}(t),\\
(\hat{g}(0),\hat{\phi}(0))&=&(\tilde{g},\tilde{\phi}).
\end{eqnarray*}
Consider a $1$-parameter family of diffeomorphisms $\Phi(t): M\to M$ by
\begin{equation}
\frac{d}{dt}\Phi(t)=\alpha_{2}\nabla_{\hat{g}(t)}\hat{\phi}(t), \ \ \ \Phi(0)
={\rm Id}_{M}.\label{3.136}
\end{equation}
If we define
\begin{equation}
g(t):=[\Phi(t)]^{\ast}\hat{g}(t), \ \ \
\phi(t):=[\Phi(t)]^{\ast}\hat{\phi}(t),\label{3.137}
\end{equation}
then
\begin{eqnarray*}
\partial_{t}g(t)&=&[\Phi(t)]^{\ast}\bigg(\partial_{t}\hat{g}(t)\bigg)
+\alpha_{1}[\Phi(t)]^{\ast}\bigg(\mathscr{L}_{\nabla_{\hat{g}(t)}
\hat{\phi}(t)}\hat{g}(t)\bigg)\\
&=&[\Phi(t)]^{\ast}\bigg(-2{\rm Rm}_{\hat{g}(t)}
+2\alpha_{1}\nabla_{\hat{g}(t)}\hat{\phi}(t)\otimes\nabla_{\hat{g}(t)}
\hat{\phi}(t)\bigg)+2\alpha_{2}[\Phi(t)]^{\ast}\nabla^{2}_{\hat{g}(t)}
\hat{\phi}(t)\\
&=&-2{\rm Rm}_{g(t)}+2\alpha_{1}\nabla_{g(t)}
\phi(t)\otimes\nabla_{g(t)}\phi(t)+2\alpha_{2}
\nabla^{2}_{g(t)}\phi(t),\\
\partial_{t}\phi(t)&=&\partial_{t}\bigg(\hat{\phi}(t)\cdot\Phi(t)\bigg)\\
&=&\partial_{t}\hat{\phi}(t)\circ\Phi(t)+\alpha_{2}|\nabla_{\hat{g}(t)}
\hat{\phi}(t)|^{2}_{\hat{g}(t)}\\
&=&\Delta_{g(t)}\phi(t)+\beta_{1}|\nabla_{g(t)}\phi(t)|^{2}_{g(t)}
+\beta_{2}\phi(t).
\end{eqnarray*}
If we furthermore have $|\hat{\phi}(t)|^{2}\lesssim1$ and $|\nabla_{\hat{g}(t)}
\hat{\phi}(t)|^{2}_{\hat{\phi}}\lesssim1$ on $M\times[0,T]$, using the standard
theory of ordinary differential equations we have that the system (\ref{3.136})
has a unique smooth solution $\Phi(t)$ on $M\times[0,T]$. Therefore $(g(t),
\phi(t))$ defined in (\ref{3.137}) are also smooth on $M\times[0,T]$ and satisfies
the above system of equations.

\begin{theorem}\label{t3.16} Let $(M,\tilde{g})$ be a complete noncompact
Riemannian $m$-manifold with bounded Riemann curvature $|{\rm Rm}_{\tilde{g}}|^{2}_{\tilde{g}}
\leq k_{0}$, and $\tilde{\phi}$ a smooth function on $M$ satisfying
\begin{equation*}
|\tilde{\phi}|^{2}+|\nabla_{\tilde{g}}\tilde{\phi}|^{2}_{\tilde{g}}
\leq k_{1}, \ \ \ |\nabla^{2}_{\tilde{g}}\tilde{\phi}|^{2}_{\tilde{g}}
\leq k_{2}.
\end{equation*}
Then there exists a positive constant $T=T(m,k_{0},k_{1},\alpha_{1}, \alpha_{2}, \beta_{1},
\beta_{2})>0$ such that the regular-$(\alpha_{1},\alpha_{2},\beta_{1},\beta_{2})$-flow
\begin{eqnarray*}
\partial_{t}g(t)&=&-2{\rm Ric}_{g(t)}+2\alpha_{1}\nabla_{g(t)}
\phi(t)\otimes\nabla_{g(t)}\phi(t),\\
\partial_{t}\phi(t)&=&\Delta_{g(t)}
\phi(t)+\beta_{1}|\nabla_{g(t)}\phi(t)|^{2}_{g(t)}
+\beta_{2}\phi(t),\\
(g(0),\phi(0))&=&(\tilde{g},\phi)
\end{eqnarray*}
has a smooth solution $(g(t),\phi(t))$
on $M\times[0,T]$ and satisfies the following estimate
\begin{equation*}
\frac{1}{C_{1}}\tilde{g}\leq g(t)\leq C_{1}\tilde{g}, \ \ \
|{\rm Rm}_{g(t)}|^{2}_{g(t)}+|\phi(t)|^{2}
+|\nabla_{g(t)}
\phi(t)|^{2}_{g(t)}+|\nabla^{2}_{g(t)}\phi(t)
|^{2}_{g(t)}
\leq C_{2}
\end{equation*}
on $M\times[0,T]$, where $C_{1}, C_{2}$ are uniform positive constants
depending only on $m, k_{0}, k_{1}, k_{2}, \alpha_{1}, \alpha_{2}, \beta_{1},\beta_{2}$.
\end{theorem}

\subsection{Higher order derivatives estimates}\label{subsection3.5}

To complete the proof of Theorem \ref{t3.1}, we need only to prove the higher
order derivatives estimates (\ref{3.1}). Suppose we have a smooth
solution $(g(t),
\phi(t))$ on $M\times[0,T]$ and satisfies
\begin{eqnarray}
\partial_{t}g(t)&=&-2{\rm Ric}_{g(t)}+2\alpha_{2}\nabla_{g(t)}
\phi(t)\otimes\nabla_{g(t)}\phi(t)+2\alpha_{2}\nabla^{2}_{g(t)}
\phi(t),\nonumber\\
\partial_{t}\phi(t)&=&\Delta_{g(t)}\phi(t)
+\beta_{1}|\nabla_{g(t)}\phi(t)|^{2}_{g(t)}+\beta_{2}\phi(t),\label{3.138}\\
(g(0),\phi(0))&=&(\tilde{g},\tilde{\phi}),\nonumber
\end{eqnarray}
where $(M,\tilde{g})$ is a complete noncompact Riemannian $m$-manifold with bounded
curvature $|{\rm Rm}_{\tilde{g}}|^{2}_{\tilde{g}}\leq k_{0}$ and $\tilde{\phi}$ is a smooth function on $M$ satisfying $|\tilde{\phi}|^{2}+|\nabla_{\tilde{g}}\tilde{\phi}|^{2}_{\tilde{g}}\leq k_{1}$ and $|\nabla^{2}_{\tilde{g}}
\tilde{\phi}|^{2}_{\tilde{g}}\leq k_{2}$, and
\begin{equation}
g(t)\approx\tilde{g}, \ \ \ |{\rm Rm}_{g(t)}|^{2}_{g(t)}
+|\phi(t)|^{2}+|\nabla_{g(t)}\phi(t)|^{2}_{g(t)}
+|\nabla^{2}_{g(t)}\phi(t)|^{2}_{g(t)}\lesssim1\label{3.139}
\end{equation}
on $M\times[0,T]$, where $\lesssim$ or $\approx$ depends only on $m, k_{0},
k_{1}, k_{2}, \alpha_{1}, \alpha_{2}, \beta_{1},
\beta_{2}$.

\begin{lemma}\label{l3.17} For any nonnegative integer $n$, there exist uniform
positive constants $C_{k}$ depending only on $m, n, k_{0}, k_{1}, k_{2}, \alpha_{1},
\alpha_{2}, \beta_{1}, \beta_{2}$ such that
\begin{equation}
\left|\nabla^{n}_{g(t)}{\rm Rm}_{g(t)}\right|^{2}_{g(t)}+\left|\nabla^{n+2}_{g(t)}
\phi(t)\right|^{2}_{g(t)}\leq\frac{C_{n}}{t^{n}}\label{3.140}
\end{equation}
on $M\times[0,T]$.
\end{lemma}

\begin{proof} As before, we always write ${\rm Rm}:={\rm Rm}_{g(t)}$, $\phi:=
\phi(t)$, etc. From Lemma \ref{l2.5}, we have
\begin{equation}
\partial_{t}{\rm Rm}=\Delta{\rm Rm}+g^{\ast-2}\ast{\rm Rm}^{\ast2}
+(\nabla^{2}\phi)^{\ast2}+g^{-1}\ast\nabla{\rm Rm}\ast\nabla\phi
+g^{-1}\ast{\rm Rm}\ast\nabla^{2}\phi.\label{3.141}
\end{equation}
Then the norm $|{\rm Rm}|^{2}$ of Riemann curvature evolves by
\begin{equation*}
\partial_{t}|{\rm Rm}|^{2}=2\langle{\rm Rm},\partial_{t}{\rm Rm}\rangle_{g}
+g^{\ast-3}\ast\partial_{t}g^{-1}\ast{\rm Rm}^{\ast2};
\end{equation*}
substituting (\ref{3.138}) and (\ref{3.141}) into above yields
\begin{eqnarray}
\partial_{t}|{\rm Rm}|^{2}&=&\Delta|{\rm Rm}|^{2}-2|\nabla{\rm Rm}|^{2}
+g^{\ast-6}\ast{\rm Rm}^{\ast3}\nonumber\\
&&+ \ g^{\ast-4}\ast{\rm Rm}\ast(\nabla^{2}\phi)^{\ast2}+g^{\ast-5}\ast{\rm Rm}^{\ast 2}\ast(\nabla\phi)^{\ast2}\label{3.142}\\
&&+ \ g^{\ast-5}\ast{\rm Rm}^{\ast2}\ast\nabla^{2}\phi+g^{\ast-5}
\ast{\rm Rm}\ast\nabla{\rm Rm}\ast\nabla\phi.\nonumber
\end{eqnarray}
Introduce a family of vector-valued tensor fields
\begin{equation}
\boldsymbol{\Lambda}:=({\rm Rm},\nabla\phi),\label{3.143}
\end{equation}
and define
\begin{equation*}
|\nabla^{k}\boldsymbol{\Lambda}|^{2}:=
|\nabla^{k}{\rm Rm}|^{2}+|\nabla^{k+1}\phi|^{2}
\end{equation*}
for each nonnegative integer $k$. According to (\ref{3.139}), (\ref{3.142}) and the
Cauchy-Schwarz inequality, we have
\begin{equation}
\partial_{t}|{\rm Rm}|^{2}\leq\Delta|{\rm Rm}|^{2}-\frac{3}{2}|\nabla{\rm Rm}|^{2}
+C_{1}.\label{3.144}
\end{equation}
Since $\partial_{t}\nabla{\rm Rm}=\nabla\partial_{t}{\rm Rm}+\partial_{t}\Gamma
\ast{\rm Rm}$ and
\begin{eqnarray*}
\partial_{t}\Gamma&=&g^{-1}\ast\nabla\partial_{t}g \ \ = \ \
g^{-1}\ast\nabla\left(g^{-1}\ast{\rm Rm}+(\nabla\phi)^{\ast2}+\nabla^{2}\phi\right)\\
&=&g^{\ast-2}\ast\nabla{\rm Rm}+g^{-1}\ast\nabla\phi\ast\nabla^{2}\phi
+g^{-1}\ast\nabla^{3}\phi,
\end{eqnarray*}
it follows that
\begin{eqnarray*}
\partial_{t}\nabla{\rm Rm}&=&\nabla\Delta{\rm Rm}+g^{\ast-2}\ast{\rm Rm}
\ast\nabla{\rm Rm}+\nabla^{2}\phi\ast\nabla^{3}\phi+g^{-1}\ast\nabla^{2}{\rm Rm}
\ast\nabla\phi\\
&&+ \ g^{-1}\ast\nabla{\rm Rm}\ast\nabla^{2}\phi
+g^{-1}\ast{\rm Rm}\ast\nabla^{3}\phi
+g^{-1}\ast{\rm Rm}\ast\nabla\phi\ast\nabla^{2}\phi\\
&=&\Delta\nabla{\rm Rm}+g^{\ast-2}\ast{\rm Rm}\ast\nabla{\rm Rm}
+\nabla^{2}\phi\ast\nabla^{3}\phi+g^{-1}\ast{\rm Rm}\ast\nabla^{3}\phi\\
&&+ \ g^{-1}\ast{\rm Rm}\ast\nabla\phi\ast\nabla^{2}\phi
+g^{-1}\ast\nabla{\rm Rm}\ast\nabla^{2}\phi+g^{-1}
\ast\nabla^{2}{\rm Rm}\ast\nabla\phi.
\end{eqnarray*}
From
\begin{eqnarray*}
\partial_{t}|\nabla{\rm Rm}|^{2}&=&2\langle\nabla{\rm Rm},
\partial_{t}\nabla{\rm Rm}\rangle+g^{\ast-4}\ast\partial_{t}g^{-1}\ast(\nabla{\rm Rm})^{\ast2}\\
&=&2\langle\nabla{\rm Rm},\partial_{t}\nabla{\rm Rm}\rangle
+g^{-6}\ast\left(g^{-1}\ast{\rm Rm}+(\nabla\phi)^{2}+\nabla^{2}\phi\right)
\ast(\nabla{\rm Rm})^{\ast2}
\end{eqnarray*}
we conclude that
\begin{eqnarray*}
\partial_{t}|\nabla{\rm Rm}|^{2}
&=&\Delta|\nabla{\rm Rm}|^{2}-2|\nabla^{2}{\rm Rm}|^{2}
+g^{\ast-7}\ast{\rm Rm}\ast(\nabla{\rm Rm})^{\ast2}\\
&&+ \ g^{\ast-5}
\ast\nabla{\rm Rm}\ast\nabla^{2}\phi\ast\nabla^{3}\phi+g^{\ast-6}\ast{\rm Rm}\ast\nabla{\rm Rm}\ast\nabla\phi
\ast\nabla^{2}\phi\\
&&+ \ g^{\ast-6}\ast(\nabla{\rm Rm})^{\ast2}\ast\nabla^{2}\phi
+g^{\ast-6}\ast\nabla{\rm Rm}\ast\nabla^{2}{\rm Rm}\ast\nabla\phi\\
&&+g^{\ast-6}\ast(\nabla{\rm Rm})^{\ast2}\ast(\nabla\phi)^{\ast2}
+g^{\ast-6}\ast{\rm Rm}\ast\nabla{\rm Rm}\ast\nabla^{3}\phi.
\end{eqnarray*}
Consequently,
\begin{eqnarray}
\partial_{t}|\nabla{\rm Rm}|^{2}&\leq&\Delta|\nabla{\rm Rm}|^{2}
-2|\nabla^{2}{\rm Rm}|^{2}+C_{2}|\nabla{\rm Rm}||\nabla^{2}{\rm Rm}|
+C_{2}|\nabla{\rm Rm}||\nabla^{3}\phi|\nonumber\\
&&+ \ C_{2}|\nabla{\rm Rm}|+C_{2}|\nabla{\rm Rm}|^{2}.\label{3.145}
\end{eqnarray}

On the other hand, Lemma \ref{l2.7} yields
\begin{eqnarray}
\partial_{t}\nabla^{2}\phi&=&\Delta\nabla^{2}\phi+g^{\ast-2}\ast{\rm Rm}
\ast\nabla^{2}\phi+\beta_{2}\nabla^{2}\phi+|\nabla\phi|^{2}\nabla^{2}\phi\nonumber\\
&&+ \ \nabla\phi\ast\nabla^{3}\phi+g^{-1}\ast(\nabla^{2}\phi)^{\ast2}
+g^{\ast-2}\ast{\rm Rm}\ast(\nabla\phi)^{\ast2}.\label{3.146}
\end{eqnarray}
Plugging (\ref{3.146}) into $\partial_{t}|\nabla^{2}\phi|^{2}
=2\langle\nabla^{2}\phi,\partial_{t}\nabla^{2}\phi\rangle
+g^{\ast-3}\ast\partial_{t}g\ast(\nabla^{2}\phi)^{\ast2}$, we arrive at
\begin{eqnarray*}
\partial_{t}|\nabla^{2}\phi|^{2}&=&
\Delta|\nabla^{2}\phi|^{2}-2|\nabla^{3}\phi|^{2}+g^{\ast-4}\ast{\rm Rm}
\ast(\nabla^{2}\phi)^{\ast2}+g^{\ast-2}\ast(\nabla^{2}\phi)^{\ast2}\\
&&+ \ g^{\ast-3}\ast(\nabla\phi)^{\ast2}
\ast(\nabla^{2}\phi)^{\ast2}
+g^{\ast-2}\ast\nabla\phi\ast\nabla^{2}\phi\ast\nabla^{3}\phi\\
&&+ \ g^{\ast-3}\ast(\nabla^{2}\phi)^{\ast3}
+g^{\ast-4}\ast{\rm Rm}\ast(\nabla\phi)^{\ast2}\ast\nabla^{2}\phi.
\end{eqnarray*}
Consequently,
\begin{equation}
\partial_{t}|\nabla^{2}\phi|^{2}\leq\Delta|\nabla^{2}\phi|^{2}
-2|\nabla^{3}\phi|^{2}+C_{3}|\nabla^{3}\phi|+C_{3}.\label{3.147}
\end{equation}
Combining (\ref{3.145}) and (\ref{3.147}), we have
\begin{equation}
\partial_{t}|\nabla\boldsymbol{\Lambda}|^{2}
\leq\Delta|\nabla\boldsymbol{\Lambda}|^{2}-\frac{3}{2}|\nabla^{2}
\boldsymbol{\Lambda}|^{2}+C_{4}|\nabla\boldsymbol{\Lambda}|^{2}+C_{4}.
\label{3.148}
\end{equation}
According to Lemma \ref{l2.6} and (\ref{3.144}), for any given positive
number $a$, we get
\begin{equation}
\partial_{t}\left(a+|\boldsymbol{\Lambda}|^{2}\right)
\leq\Delta\left(a+|\boldsymbol{\Lambda}|^{2}\right)-\frac{3}{2}|\nabla\boldsymbol{\Lambda}|^{2}
+C_{5}.\label{3.149}
\end{equation}
Therefore
\begin{eqnarray}
\partial_{t}\left[\left(a+|\boldsymbol{\Lambda}|^{2}\right)
|\nabla\boldsymbol{\Lambda}|^{2}\right]
&\leq&\Delta\left[\left(a+|\boldsymbol{\Lambda}|^{2}\right)|\nabla
\boldsymbol{\Lambda}|^{2}\right]-2\left\langle\nabla|\boldsymbol{\Lambda}|^{2},\nabla|\nabla
\boldsymbol{\Lambda}|^{2}\right\rangle\nonumber\\
&&- \ \frac{3}{2}|\nabla\boldsymbol{\Lambda}|^{4}
+C_{5}|\nabla\boldsymbol{\Lambda}|^{2}-\frac{3}{2}\left(a+|\boldsymbol{\Lambda}|^{2}
\right)|\nabla^{2}\boldsymbol{\Lambda}|^{2}\label{3.150}\\
&&+ \ C_{4}\left(a+|\boldsymbol{\Lambda}|^{2}\right)
|\nabla\boldsymbol{\Lambda}|^{2}+C_{4}\left(a+|\boldsymbol{\Lambda}|^{2}\right).\nonumber
\end{eqnarray}
By the definition, the second term on the right-hand side of (\ref{3.150}) is bounded from above by
\begin{eqnarray*}
&&g^{pq}\nabla_{p}\left(|{\rm Rm}|^{2}+|\nabla\phi|^{2}\right)
\nabla_{q}\left(|\nabla{\rm Rm}|^{2}+|\nabla^{2}\phi|^{2}\right)\\
&=&g^{-1}\ast\left(g^{-4}\ast{\rm Rm}\ast\nabla{\rm Rm}
+g^{-1}\ast\nabla\phi\ast\nabla^{2}\phi\right)\\
&&\ast\left(g^{\ast-5}\ast\nabla{\rm Rm}\ast\nabla^{2}{\rm Rm}
+g^{\ast-2}\ast\nabla^{2}\phi\ast\nabla^{3}\phi\right)\\
&\leq&C_{6}|\nabla{\rm Rm}|^{2}|\nabla^{2}{\rm Rm}|
+C_{6}|\nabla{\rm Rm}||\nabla^{2}{\rm Rm}|+C_{6}|\nabla{\rm Rm}||\nabla^{3}
\phi|+C_{6}|\nabla^{3}\phi|\\
&\leq&C_{6}|\nabla\boldsymbol{\Lambda}|^{2}|\nabla^{2}\boldsymbol{\Lambda}|
+C_{6}|\nabla\boldsymbol{\Lambda}||\nabla^{2}\boldsymbol{\Lambda}|
+C_{6}|\nabla\boldsymbol{\Lambda}||\nabla^{2}\boldsymbol{\Lambda}|
+C_{6}|\nabla^{2}\boldsymbol{\Lambda}|\\
&\leq&\frac{3}{8}a|\nabla^{2}\boldsymbol{\Lambda}|^{2}
+\frac{2C^{2}_{6}}{3a}|\nabla\boldsymbol{\Lambda}|^{4}
+\frac{6}{8}a|\nabla^{2}\boldsymbol{\Lambda}|^{2}+\frac{4C^{2}_{6}}{3a}
|\nabla\boldsymbol{\Lambda}|^{2}+\frac{3}{8}a|\nabla^{2}\boldsymbol{\Lambda}|^{2}
+\frac{2C^{2}_{6}}{3a}\\
&\leq&\frac{3}{2}a|\nabla^{2}\boldsymbol{\Lambda}|^{2}
+\frac{2C^{2}_{6}}{a}|\nabla\boldsymbol{\Lambda}|^{4}+
\frac{2C^{2}_{6}}{a}
\end{eqnarray*}
where we used the inequality $x^{2}\leq x^{4}+1$ for any $x\geq0$. Substituting
this inequality into (\ref{3.150}) implies
\begin{eqnarray}
\partial_{t}\left[\left(a+|\boldsymbol{\Lambda}|^{2}\right)
|\nabla\boldsymbol{\Lambda}|^{2}\right]
&\leq&\Delta\left[\left(a+|\boldsymbol{\Lambda}|^{2}\right)
|\nabla\boldsymbol{\Lambda}|^{2}\right]
-\left(\frac{3}{2}-\frac{2C^{2}_{6}}{a}\right)|\nabla\boldsymbol{\Lambda}|^{4}
+C_{5}|\nabla\boldsymbol{\Lambda}|^{2}\nonumber\\
&&+ \ C_{4}\left(a+|\boldsymbol{\Lambda}|^{2}\right)|\nabla\boldsymbol{\Lambda}|^{2}
+C_{4}\left(a+|\boldsymbol{\Lambda}|^{2}\right)+\frac{2C^{2}_{6}}{a}.\label{3.151}
\end{eqnarray}
By (\ref{3.139}), we can choose $a$ so that
\begin{equation*}
a\geq 4C^{2}_{6}, \ \ \ a\geq\max_{M\times[0,T]}\left(|{\rm Rm}|^{2}+|\nabla^{2}\phi|^{2}\right);
\end{equation*}
then $\frac{3}{2}-\frac{2C^{2}_{6}}{a}\geq1$ and $a\leq a+|\boldsymbol{\Lambda}|^{2}
\leq 2a$. Consequently, we can deduce from (\ref{3.151}) that
\begin{eqnarray*}
\partial_{t}\left[\left(a+|\boldsymbol{\Lambda}|^{2}\right)|\nabla\boldsymbol{\Lambda}|^{2}\right]
&\leq&\Delta\left[\left(a+|\boldsymbol{\Lambda}|^{2}\right)|\nabla\boldsymbol{\Lambda}|^{2}\right]
-|\nabla\boldsymbol{\Lambda}|^{4}\\
&&+ \ \left(C_{4}+\frac{C_{4}}{a}\right)
\left(a+|\boldsymbol{\Lambda}|^{2}\right)|\nabla\boldsymbol{\Lambda}|^{2}
+2C_{4}a+\frac{2C^{2}_{6}}{a}\\
&\leq&\Delta\left[\left(a+|\boldsymbol{\Lambda}|^{2}\right)
|\nabla\boldsymbol{\Lambda}|^{2}\right]-\frac{1}{4a^{2}}
\left[\left(a+|\boldsymbol{\Lambda}|^{2}\right)|\nabla\boldsymbol{\Lambda}|^{2}
\right]^{2}\\
&&+ \ \left(C_{4}+\frac{C_{4}}{a}\right)
\left(a+|\boldsymbol{\Lambda}|^{2}\right)|\nabla\boldsymbol{\Lambda}|^{2}
+2C_{4}a+\frac{2C^{2}_{6}}{a}\\
&\leq&\Delta\left[\left(a+|\boldsymbol{\Lambda}|^{2}\right)
|\nabla\boldsymbol{\Lambda}|^{2}\right]-\frac{1}{8a^{2}}
\left[\left(a+|\boldsymbol{\Lambda}|^{2}\right)|\nabla\boldsymbol{\Lambda}|^{2}
\right]^{2}+C_{7}.
\end{eqnarray*}
Consider the function
\begin{equation}
u:=\left(a+|\boldsymbol{\Lambda}|^{2}\right)|\nabla\boldsymbol{\Lambda}|^{2}t
\label{3.152}
\end{equation}
defined on $M\times[0,T]$. Then
\begin{eqnarray}
u&=&0 \ \ \ \text{on} \ M\times\{0\},\label{3.153}\\
\partial_{t}u&\leq&\Delta u-\frac{1}{8a^{2}t}u^{2}
+C_{7}t+\frac{1}{t}u \ \ \ \text{on} \ M\times[0,T].\label{3.154}
\end{eqnarray}

Fix a point $x_{0}\in M$, consider the cutoff function $\xi(x)$ on $M$, introduced
in \cite{S89}, such that
\begin{eqnarray*}
\xi&=&1 \ \ \ \text{on} \ B_{\tilde{g}}(x_{0},1), \ \ \
\xi \ \ = \ \ 0 \ \ \ \text{on} \ M\setminus B_{\tilde{g}}(x_{0},2), \ \ \ \
0\leq\xi\leq1,\\
|\tilde{\nabla}\xi|^{2}_{\tilde{g}}&\leq&16\xi, \ \ \ \tilde{\nabla}^{2}
\xi \ \ \geq \ \ -c\tilde{g}, \ \ \ \text{on} \ M,
\end{eqnarray*}
where $c$ depends on $k_{0}$. As in \cite{S89}, we define
\begin{equation}
F:=\xi u\label{3.155}
\end{equation}
on $M\times[0,T]$. Then from the definition, we have
\begin{equation*}
F=0 \ \text{on} \ M\times\{0\}, \ \
F=0 \ \text{on} \ (M\setminus B_{\tilde{g}}(x_{0},2))\times[0,T], \ \
F\geq0 \ \text{on} \ M\times[0,T].
\end{equation*}
Without loss of generality, we may assume that $F$ is not identically zero on
$M\times[0,T]$. In this case, we can find a point $(x_{1},t_{1})\in B_{\tilde{g}}(x_{0},2)\times[0,T]$ such that
\begin{equation*}
F(x_{1},t_{1})=\max_{M\times[0,T]}F(x,t)>0,
\end{equation*}
which implies $t_{1}>0$ and
\begin{equation*}
\partial_{t}F(x_{1},t_{1})>0, \ \ \ \nabla F(x_{1},t_{1})
=0, \ \ \ \Delta F(x_{1},t_{1})\leq0.
\end{equation*}
If $u(x_{1},t_{1})\leq1$, then $F(x_{1},t_{1})\leq1$ by (\ref{3.155}) and hence
\begin{equation*}
at|\nabla\boldsymbol{\Lambda}|^{2}\leq u
=F\leq1 \ \ \ \text{on} \ B_{\tilde{g}}(x_{0},1)\times[0,T];
\end{equation*}
in particular,
\begin{equation}
|\nabla{\rm Rm}|^{2}+|\nabla^{2}\phi|^{2}\lesssim\frac{1}{t} \ \ \
\text{on} \ M\times[0,T].\label{3.156}
\end{equation}

In the following we assume that $u(x_{1},t_{1})\geq1$. Under this assumption,
together with (\ref{3.154}) and $\xi\partial_{t}u\geq0$ at $(x_{1},t_{1})$,
we arrive at, at the point $(x_{1},t_{1})$,
\begin{eqnarray*}
0&\leq&\xi\partial_{t}u \ \ \leq \ \ \xi\Delta u+
\left(\frac{1}{t}u+C_{7}t-\frac{1}{8a^{2}t}u^{2}\right)\xi\\
&\leq&\xi\Delta u+\left(\frac{1}{t}u+C_{7}tu-\frac{1}{8a^{2}t}u^{2}\right)
\xi\\
&\leq&\xi\Delta u+\frac{u}{t}\left(1+C_{7}t^{2}-\frac{1}{8a^{2}}u\right)
\xi \ \ \leq \ \ \xi\Delta u+\frac{u}{t}\left(C_{8}-C_{9}u\right)\xi;
\end{eqnarray*}
thus
\begin{equation*}
\xi\Delta u+\frac{\xi u}{t}\left(C_{8}-C_{9}u\right)\geq0 \ \ \ \text{at} \
(x_{1},t_{1}).
\end{equation*}
By the same argument in \cite{S89} (equations (28)--(35) in page 293), we find that
\begin{equation}
\frac{\xi u}{t}(C_{9}u-C_{8})\leq C_{10}u-u\Delta\xi \ \ \ \text{at} \ (x_{1},t_{1}).
\label{3.157}
\end{equation}
According to the equation (38) in page 294 of \cite{S89}, we get
\begin{equation}
-\Delta\xi\leq C_{11}+g^{\alpha\beta}(\Gamma^{\gamma}_{\alpha\beta}
-\tilde{\Gamma}^{\gamma}_{\alpha\beta})\tilde{\nabla}_{\gamma}\xi.\label{3.158}
\end{equation}
Since
\begin{equation*}
\partial_{t}\Gamma=g^{\ast-2}\ast\nabla{\rm Rm}+g^{-1}\ast\nabla^{3}\phi
+g^{-1}\ast\nabla\phi\ast\nabla^{2}\phi,
\end{equation*}
it follows that
\begin{equation*}
|\partial_{t}\Gamma|\leq C_{11}|\nabla\boldsymbol{\Lambda}|+C_{12}
\leq\frac{C_{11}}{\sqrt{at}}u^{1/2}+C_{12};
\end{equation*}
Since $\xi(x_{1})u(x_{1},t)=F(x_{1},t)\leq F(x_{1},t_{1})$ for $t\in[0,T]$, we obtain
\begin{equation*}
|\partial_{t}(x_{1},t)|\leq C_{11}\left[\frac{F(x_{1},t_{1})}{a\xi(x_{1})}\right]^{1/2}
\frac{1}{\sqrt{t}}+C_{12}
\end{equation*}
for any $t\in[0,T]$. As showed in \cite{S89} (the equation (45) in page 295),
together with $|\tilde{\nabla}\xi|_{\tilde{g}}\leq 4\xi^{1/2}$, we find that
\begin{equation}
g^{\alpha\beta}(\Gamma^{\gamma}_{\alpha\beta}
-\tilde{\Gamma}^{\gamma}_{\alpha\beta})\tilde{\nabla}_{\gamma}
\xi\leq C_{13}F(x_{1},t_{1})^{1/2}+C_{14};\label{3.159}
\end{equation}
substituting (\ref{3.159}) into (\ref{3.158}) implies
\begin{equation}
-\Delta\xi\leq C_{15}+C_{13}F(X_{1},t_{1})^{1/2} \ \ \ \text{at} \ (x_{1},t_{1}).
\label{3.160}
\end{equation}
According to (\ref{3.157}) and (\ref{3.160}), we have the following inequality
\begin{equation*}
\xi u(C_{9}u-C_{8})\leq C_{10}tu+C_{15}tu+C_{13}tu F^{1/2}
\leq C_{16}u+C_{16}u F^{1/2} \ \ \ \text{at} \ (x_{1},t_{1});
\end{equation*}
multiplying by $\xi(x_{1})$ yields
\begin{equation*}
C_{9}F^{2}\leq C_{17}F+C_{16}F^{3/2} \ \ \ \text{at} \ (x_{1},t_{1}),
\end{equation*}
from which we deduce that $F(x_{1},t_{1})\lesssim1$ and therefore
\begin{equation*}
\xi u\lesssim 1 \ \ \ \text{on} \ M\times[0,T].
\end{equation*}
In particular, $u\lesssim1$ on $M\times[0,T]$ since $x_{0}$ was arbitrary. From
the definition (\ref{3.152}), this tells us the estimate $|\nabla\boldsymbol{\Lambda}|^{2}
\lesssim1/t$ on $M\times[0,T]$, where $\lesssim$ depends only on $m, k_{0}, k_{1},
k_{2}, \alpha_{1}, \alpha{2}, \beta_{1}, \beta_{2}$. Hence the lemma holds for $n=1$.

By induction, suppose for $s=1,\cdots,n-1$ we have
\begin{equation*}
\left|\nabla^{s}{\rm Rm}\right|^{2}+\left|\nabla^{s+2}\phi\right|^{2}
\lesssim\frac{1}{t^{s}}
\end{equation*}
on $M\times[0,T]$. As in \cite{S89}, we define a function
\begin{equation*}
v:=\left(a+t^{n-1}|\nabla^{n-1}\boldsymbol{\Lambda}|^{2}
\right)|\nabla^{n}\boldsymbol{\Lambda}|^{2}t^{n}
\end{equation*}
and choose $a$ sufficiently large. Similarly to (\ref{3.154}) we can show
that
\begin{equation*}
\partial_{t}v\leq\Delta v-\frac{C_{18}}{a^{2}t}v^{2}+C_{19}+\frac{C_{20}}{t}v
\end{equation*}
on $M\times[0,T]$. Using the same cutoff function $\xi$ and arguing in the
same way, we obtain that $v\lesssim1$ on $M\times[0,T]$. Hence the inequality
(\ref{3.140}) holds for $s=n$.
\end{proof}



\end{document}